\documentclass[12pt,a4paper]{article}

\usepackage{amsfonts, amssymb, amsmath, amsthm}

\makeatletter
\renewcommand{\mod}[1]{\allowbreak\mkern 12mu{\operator@font mod}\,\,#1}
\makeatother

\allowdisplaybreaks 
\usepackage{bbold} 
\usepackage{mathtools}
\usepackage[utf8]{inputenc} 
\usepackage{comment} 
\usepackage[usenames,dvipsnames,svgnames,table]{xcolor}
\usepackage[margin=0.75in]{geometry} 
\usepackage{fancyhdr} 
\usepackage{epigraph} 
\usepackage{floatrow} 
\usepackage{afterpage}
\usepackage{longtable}

\usepackage{makecell}

\usepackage{tikz} 
\usepackage{tkz-euclide} 
\usepackage{tikz-cd} 
\usepackage{tkz-graph}
\usetikzlibrary{arrows}
\usetikzlibrary{arrows.meta} 
\usetikzlibrary{bending} 
\usetikzlibrary{calc}
\usetikzlibrary{matrix}
\usetikzlibrary{intersections}
\usetikzlibrary{patterns}
\usetikzlibrary{math} 
\usetikzlibrary{shapes} 
\usetikzlibrary{shapes.geometric} 
\usetikzlibrary{through,backgrounds}

\usepackage{dynkin-diagrams}

\usetikzlibrary{lindenmayersystems}
\pgfdeclarelindenmayersystem{cayley}{
  \rule{F -> F [ R [F] [+F] [-F] ]}
  \symbol{R}{
    \pgflsystemstep=0.5\pgflsystemstep
  }
}

\usepackage{pgfplots}
\usepgfplotslibrary{polar}
\pgfplotsset{compat=newest}
\usepackage{nicematrix} 
\usepackage{capt-of} 
\usepackage[tableposition=top]{caption}
\usepackage{subcaption}
\usepackage[shortlabels]{enumitem}
\usepackage[ruled,linesnumbered]{algorithm2e} 
\usepackage{titlesec} 
\usepackage{booktabs} 
\usepackage{multicol} 
\usepackage[hypertexnames=false]{hyperref} 
\makeatletter
\usepackage{etoolbox}
\patchcmd\Hy@EveryPageBoxHook{\Hy@EveryPageAnchor}{\Hy@hypertexnamestrue\Hy@EveryPageAnchor}{}{\fail}
\makeatother
\usepackage[nameinlink]{cleveref}

\usepackage[backend=biber,style=numeric,sorting=nty,hyperref=true,
maxbibnames=99]{biblatex} 

\usepackage{makeidx} 
\makeindex 
\usepackage[totoc]{idxlayout}

\usepackage{braket} 
\renewcommand\Set[1]{%
 \csname Set \endcsname{\mskip-\medmuskip#1\mskip-\medmuskip}} 

\usepackage{libertine} 
\usepackage[libertine]{newtxmath} 

\DeclareMathAlphabet{\mathcal}{OMS}{cmsy}{m}{n}

\definecolor{codegreen}{rgb}{0,0.6,0}
\definecolor{codegray}{rgb}{0.5,0.5,0.5}
\definecolor{codepurple}{rgb}{0.58,0,0.82}
\definecolor{backcolour}{rgb}{0.95,0.95,0.92}

\makeatletter%
\@ifclassloaded{article}%
{\fancyhf{}

}
\makeatother%

\makeatletter%
\@ifclassloaded{book}%
{\renewcommand\tableofcontents{%
\chapter*{\contentsname
\@mkboth{%
\MakeUppercase\contentsname}{}}
\@starttoc{toc}%
}}
\makeatother%
\makeatletter%
\@ifclassloaded{book}%
{\titleformat{\part}[display]{\centering\normalfont\sffamily\Huge\bfseries\color{black}}
{\partname \space \thepart}{20pt}{\Huge}}%
\makeatother%

\titleformat{\section}{\normalfont\sffamily\Large\bfseries\color{black}}
{\thesection}{1em}{}
\titleformat{\subsection}{\normalfont\sffamily\large\bfseries\color{black}}
{\thesubsection}{1em}{}
\titleformat{\subsubsection}
{\normalfont\sffamily\normalsize\bfseries\color{black}}
{\thesubsubsection}{1em}{}
\titleformat{\paragraph}{\normalfont\sffamily}
{\paragraph \theparagraph}{1em}{}

\makeatletter%
\@ifclassloaded{book}%
{\titleformat{\chapter}[frame]
{\normalfont}{\filright\enspace \@chapapp~\thechapter\enspace}
{8pt}{\Large\bfseries\sffamily\filcenter}
\titlespacing*{\chapter}
{0pt}{0pt}{20pt}}
\makeatother%

\usepackage{thmtools}
\usepackage[most]{tcolorbox}
\tcbuselibrary{theorems,skins,breakable,xparse} 

\def\boxnumber{section}

\newtheoremstyle{notation}
{\topsep}   
{\topsep}   
{\normalfont}  
{0pt}       
{\scshape} 
{.}         
{5pt plus 1pt minus 1pt} 
{}          

\theoremstyle{notation}

\newtheorem{definition}{Definition}[\boxnumber]

\newtheorem{theorem}{Theorem}[\boxnumber]
\newtheorem{lemma}[theorem]{Lemma}
\newtheorem{corollary}[theorem]{Corollary}

\makeatletter%
\renewenvironment{proof}[1][\proofname]{%
 \par\pushQED{\qed}\normalfont%
 \topsep6\p@\@plus6\p@\relax%
 \trivlist\item[\hskip\labelsep\itshape\sffamily#1\@addpunct{.}]
 \ignorespaces%
}{%
 \popQED\endtrivlist\@endpefalse%
}%
\makeatother%

\hypersetup{
 colorlinks=true,
 linktoc=all,
 citecolor=black,
 filecolor=black,
 linkcolor=black, 
 urlcolor=black
}

\newcommand*{\fullref}[1]{\textsf{\hyperref[{#1}]{\autoref*{#1}}}}
\tikzset{>={Stealth[inset=0pt, angle'=30]}}
\tikzcdset{arrow style=tikz}

\newcommand{\bF}{\mathbb{F}}

\newcommand{\bZ}{\mathbb{Z}}

\newcommand{\cI}{\mathcal{I}}

\newcommand{\cL}{\mathcal{L}}

\newcommand{\cP}{\mathcal{P}}

\let\originalleft\left
\let\originalright\right
\renewcommand{\left}{\mathopen{}\mathclose\bgroup\originalleft}
\renewcommand{\right}{\aftergroup\egroup\originalright}

\newcommand{\paren}[1]{\left( #1 \right)}

\newcommand{\ang}[1]{\left\langle #1 \right\rangle}

\newcommand{\defi}{\vcentcolon=}

\newcommand{\idgp}{\vmathbb{1}}

\newcommand{\subgp}{\leq}
\newcommand{\propsubgp}{<}
\newcommand{\nsubgp}{\trianglelefteq}

\newcommand{\soc}[1]{\textrm{soc}(#1)}

\newcommand{\divides}{\mid}

\newcommand{\field}{\mathbb{K}}
\newcommand{\Centraliser}[2]{\textbf{C}_{#1}\left(#2\right)}

\newcommand{\Normaliser}[2]{\textbf{N}_{#1}(#2)}

\newcommand{\tenk}[1]{\otimes}

\newcommand{\restr}[2]{{
\left.\kern-\nulldelimiterspace 
#1 
\vphantom{\big|}
\right|_{#2} 
}}
\newcommand{\corestr}[2]{{
\left.\kern-\nulldelimiterspace 
#1 
\vphantom{\big|}
\right|^{#2} 
}}
\newcommand{\birestr}[3]{{
\left.\kern-\nulldelimiterspace 
#1 
\vphantom{\big|}
\right|_{#2}^{#3} 
}}

\def\GL{\mathsf{GL}}

\def\SU{\mathsf{SU}}

\def\PSL{\mathsf{PSL}}
\def\PSU{\mathsf{PSU}}
\def\PGU{\mathsf{PGU}}
\def\PSP{\mathsf{PSp}}

\def\SP{\mathsf{Sp}}

\def\PGL{\mathsf{PGL}}

\newcommand{\GLn}[2]{\GL_{#1}(#2)}

\newcommand{\SPn}[2]{\SP_{#1}(#2)}
\newcommand{\SUn}[2]{\SU_{#1}(#2)}

\newcommand{\PGLn}[2]{\PGL_{#1}({#2})}
\newcommand{\PSLn}[2]{\PSL_{#1}({#2})}
\newcommand{\PSUn}[2]{\PSU_{#1}({#2})}
\newcommand{\PGUn}[2]{\PGU_{#1}({#2})}
\newcommand{\PSPn}[2]{\PSP_{#1}({#2})}

\newcommand{\Sz}[1]{\prescript{2}{}{\mathrm{B}_2}\paren{#1}}
\newcommand{\Ree}[1]{\prescript{2}{}{\mathrm{G}_2}\paren{#1}}
\newcommand{\LargeRee}[1]{\prescript{2}{}{\mathrm{F}_4}\paren{#1}}

\def\Aut{\textsf{Aut}}
\def\Inn{\textsf{Inn}}
\def\Out{\textsf{Out}}

\newcommand{\ds}{\displaystyle}

\usepackage{pdfpages}

\renewcommand{\epsilon}{\varepsilon}

\Crefname{lstlisting}{Listing}{Listings}

\usepackage{authblk}

\author[]{Vishnuram Arumugam}
\affil[]{\small Center for the Mathematics of Symmetry and Computation,
  The University of Western Australia, Perth 6009, Australia\\
\href{mailto:Vishnuram.Arumugam@uwa.edu.au}{Vishnuram.Arumugam@uwa.edu.au}}

\title{\sc Large Ree groups of Lie type acting point and line-primitively on 
finite generalised quadrangles}

\begin{document}
\hypersetup{linkcolor=WildStrawberry}
\pagenumbering{arabic}
\setcounter{page}{1}
\pagestyle{plain}

\maketitle
\vspace*{-0.5cm}

\abstract{
Suppose we have a finite thick 
generalised quadrangle whose automorphism group $G$ acts
primitively on both the set of points and the set of lines. Then $G$ must be
almost simple. Furthermore, $\soc{G}$ cannot be isomorphic to $\Sz{2^{2m+1}}$
or $\Ree{3^{2m+1}}$, where $m$ is a positive integer. 
In this paper, we show that $\soc{G}$ cannot be isomorphic to 
$\LargeRee{2^{2m+1}}$. 
}

\section{Introduction}
A generalised polygon is a type of incidence structure introduced by Jacques
Tits (1959) 
\cite{tits} to realise groups of Lie type as symmetries (more precisely,
automorphism groups) of geometric objects. Let $n$ be a positive integer. 
A generalised $n$-gon is an incidence 
structure whose incidence graph is a bipartite graph with diameter $n$ and 
girth $2n$. We say that a generalised
$n$-gon has order $(s,t)$ if every line has $s+1$ points incident with it and
every point is incident with $t+1$ lines. Furthermore, a generalised $n$-gon 
is said to be thick if it has order $(s,t)$ where $s,t > 1$. Feit and Higman
showed in \cite{feit} that thick generalised $n$-gons exist if and only
if $n \in \{3,4,6,8\}$. Our focus from now will only be on thick generalised
polygons. Therefore, we refer to a thick generalised polygon as simply a 
generalised polygon. The examples of 
generalised polygons that arise from 
groups of Lie type via the construction of Tits are called \emph{classical}. 
Since then, many non-classical examples of projective planes (generalised 
$3$-gons) and generalised quadrangles ($4$-gons) have been found 
\cite[Section 3.7]{maldy}. In 
the case of generalised hexagons ($6$-gons) and generalised octagons
($8$-gons), the only known examples are the classical ones. \\

Many attempts have been made to construct new
examples. These attempts involve introducing various symmetry conditions on
the automorphism groups and analysing the possible groups acting on
generalised polygons with these symmetry conditions. Buekenhout and Van
Maldeghem (1994) \cite{buek} 
showed that if a group acts distance-transitively on a generalised
$n$-gon ($n \geqslant 4)$, then in fact, it must act point-primitively on that 
generalised $n$-gon. 
We update \Cref{tab:genhex,tab:genquad} from \cite{vishnuram}, including recent work by 
Ou and Lu \cite{meiziou} on the action of $\PSPn{4}{q}$ on generalised quadrangles. 
Here, we let $\Gamma$ be a
generalised $n$-gon (generalised hexagon or octagon in \Cref{tab:genhex} and
generalised quadrangle in \Cref{tab:genquad}) and $G \subgp \Aut(\Gamma)$. 
For a point $\alpha$ in $\Gamma$, we write $G_{\alpha}$ to denote the 
point-stabiliser of $\alpha$ in $G$. Also, $q$ is assumed to be a prime power. 
Finally, the column of assumptions refers to the action of $G$ on $\Gamma$.

\begin{longtable}{llll}
\caption{Summary of results on generalised hexagons and octagons \cite[Table
1]{vishnuram}.}
\label{tab:genhex}
\endfirsthead
\endhead
\toprule
\thead{Assumptions} & \thead{Conclusion} & \thead{Year} & \thead{Ref.} \\
\midrule
\makecell{Point-primitive, \\ line-primitive \\ and flag-transitive} & 
$G$ is almost simple of Lie type & 2008 & \cite{schneider} \\
\midrule
Point-primitive & $G$ is almost simple of Lie type & 2017 & \cite{bamberg} \\
\midrule
\makecell{Point-primitive and \\ 
$\soc{G} \cong \PSLn{n}{q}$ for $n \geqslant 2$} &
$G_{\alpha}$ acts irreducibly on $V = \bF_q^n$ & 2021 & \cite{glasbyoct} \\
\midrule
\makecell{Point-primitive and \\ 
$\soc{G} \cong \Sz{2^{2m+1}}$, \\ $\Ree{3^{2m+1}}$ or 
$\LargeRee{2^{2m+1}}$} & \makecell{$\soc{G} \cong \LargeRee{2^{2m+1}}$ and \\
$\Gamma$ is the classical generalised octagon or its dual} & 2016 & 
\cite{morgan} \\
\bottomrule
\end{longtable}

\begin{longtable}{llll}
\caption{Summary of results on generalised quadrangles \cite[Table
2]{vishnuram}.}
\label{tab:genquad}
\endfirsthead
\endhead
\toprule
\thead{Assumptions} & \thead{Conclusion} & \thead{Year} & \thead{Ref.} \\
\midrule
\makecell{Point-primitive and \\ line-primitive} & $G$ is almost simple &
2012 & \cite{bgmrs} \\
\midrule
\makecell{Point-primitive, \\ line-primitive \\ and flag-transitive} & 
$G$ is almost simple of Lie type & 2012 & \cite{bgmrs} \\
\midrule
\makecell{Point-primitive, \\ flag-transitive and \\ $\soc{G} \cong A_n$ with
$n \geqslant 5$} & \makecell{$G \subgp S_6$ and \\ 
$\Gamma$ is the symplectic quadrangle $W(2)$} & 2012  & \cite{bgmrs} \\
\midrule
Point-primitive & \makecell{$\soc{G}$ is not isomorphic to a sporadic group} 
& 2021 & \cite{bambergevans} \\
\midrule
\makecell{Point-primitive, \\ line-primitive and \\ 
$\soc{G} \cong \PSLn{2}{q}$ for $q \geqslant 4$} &
\makecell{$q = 9$ and \\ $\Gamma$ is the symplectic quadrangle $W(2)$} & 
2023 & \cite{fenglu} \\
\midrule
\makecell{Point-primitive and \\ 
line-primitive} & \makecell{$\soc{G}$ is not isomorphic to $\PSUn{3}{q}$ \\ 
for $q \geqslant 3$} & 2024 & \cite{luzhangzou} \\
\midrule
\makecell{Point-primitive and \\ line-primitive} & \makecell{$\soc{G}$ is not
isomorphic to \\ $\Sz{2^{2m+1}}$ or $\Ree{3^{2m+1}}$} & 2025 & 
\cite{vishnuram} \\
\midrule
\makecell{Point-primitive, \\ line-primitive and \\ 
$\soc{G} \cong \PSPn{4}{q}$ for $q \geqslant 4$} &
\makecell{$\Gamma$ is the symplectic quadrangle $W(3,q)$} & 
2026 & \cite{meiziou} \\
\bottomrule
\end{longtable}

The main theorem in this paper (\Cref{thm:main}) is a continuation of the work
in \cite{vishnuram}, where the authors showed that if an almost simple group
$G$ acts primitively on both the set of points and the set of lines of a
generalised quadrangle, then its socle, $\soc{G}$, cannot be isomorphic to a
Suzuki or a small Ree group. In this paper, we show that $\soc{G}$ cannot be
isomorphic to a large Ree group. Both \cite{vishnuram} and this paper are
motivated by the work of Morgan and Popiel \cite{morgan}, where they showed
that if $G$ acts primitively on the set of points of a generalised hexagon or
octagon $\Gamma$ with $\soc{G} \cong \Sz{2^{2m+1}}$, $\Ree{3^{2m+1}}$ or 
$\LargeRee{2^{2m+1}}$, then the only possibility is that $\soc{G} \cong
\LargeRee{2^{2m+1}}$ and $\Gamma$ is the classical generalised octagon (up to
duality).

\begin{theorem}
\label{thm:main}
Let $m$ be a positive integer. An almost simple group with socle isomorphic to 
$\LargeRee{2^{2m+1}}$ cannot act primitively on both the set of
points and the set of lines of a generalised quadrangle.
\end{theorem}

\section{Incidence Geometry}
\label{chap:incidence}
We collect some preliminary definitions and results from 
incidence geometry and group theory.

\subsection{Definitions and Results}
An incidence geometry is a triple $(\cP,\cL,\cI)$ where $\cP$ is
called the point-set, $\cL$ is called the line-set and $\cI$ is called the
incidence relation. We refer to the elements of $\cP$, $\cL$ and $\cI$ as
points, lines and flags, respectively. A point 
$\alpha \in \cP$ is \emph{incident} to a line $L \in \cL$ if $(\alpha,L) \in 
\cI$. Two points $\alpha,\beta \in \cP$ are \emph{collinear} if there
exists a line $L \in \cL$ such that $(\alpha,L) \in \cI$ and $(\beta,L) \in 
\cI$. Finally, the \emph{dual} of an incidence structure $(\cP,\cL,\cI)$ is an
incidence structure $(\cP^{\prime},\cL^{\prime},\cI^{\prime})$ where
the set of points $\cP^{\prime} = \cL$, the set of lines $\cL^{\prime} = \cP$ 
and the incident pair $(L,\alpha) \in \cI^{\prime}$ precisely when 
$(\alpha,L) \in \cI$. \\

We may also view an incidence structure as a bipartite graph. More precisely, 
given an incidence structure $\Gamma = (\cP,\cL,\cI)$, we construct a graph 
called the
\emph{incidence graph} of $\Gamma$ in the following way: we take the vertex
set to be $\cP \cup \cL$ and we join an edge between $\alpha \in \cP$ and 
$L \in \cL$ if $(\alpha,L) \in \cI$. Note the disjointness of $\cP$ 
and $\cL$ ensures that the incidence graph does not contain loops and that
this graph is indeed a bipartite graph with the two parts being $\cP$ and
$\cL$. Finally, an
\emph{automorphism} $\sigma$ of $\Gamma$ is a bijection 
$\sigma: \cP \cup \cL \rightarrow \cP \cup \cL$ that maps points to points
and lines to lines and $\sigma$ preserves the incidence relation,
\emph{i.e.}, if $\alpha \in \cP$ is incident to $L \in \cL$ then
$\alpha^{\sigma}$ is incident to $L^{\sigma}$. The collection of all 
automorphisms of $\Gamma$ is called the \emph{automorphism group} and it is 
denoted by $\Aut(\Gamma)$. \\
While we have defined generalised polygons via this incidence graph, we will
often think of them as incidence geometries. For more
information on incidence structures and generalised polygons, we refer to
\cite{kiss,payne,maldy}. \\

From here on, a generalised quadrangle is assumed to be finite, \emph{i.e.}, 
it has a finite set of points and set of lines. Furthermore, a generalised 
quadrangle with order $(s,t)$ has $(s+1)(st+1)$ points and $(t+1)(st+1)$ 
lines \cite{payne}. \\ 

From this, we can obtain the number of points that are collinear to a given
point and the number of points non-collinear to a given point. 
\begin{lemma}[Number of Collinear and Non-Collinear Points]
\label{cor:numcol}
Suppose $\Gamma = (\cP,\cL,\cI)$ is a generalised 
quadrangle with order $(s,t)$ 
and let $\alpha \in \cP$. Then the number of points collinear to $\alpha$ is
given by $s(t+1)$ and the number of points non-collinear to $\alpha$ is 
$s^2t$. 
\end{lemma}
\begin{proof}
There are $t+1$ lines incident to $\alpha$ and each of these lines contain $s$
points other than $\alpha$. Thus, the number of points collinear to $\alpha$
is $s(t+1)$. Excluding the points collinear to $\alpha$ and $\alpha$ itself, 
we obtain number of points non-collinear to $\alpha$, which is given by
$|\cP|-s(t+1)-1 = (s+1)(st+1)-s(t+1)-1 = s^2t$. 
\end{proof}

The following lemma provides conditions on the order $(s,t)$ of a generalised
polygon. 
\begin{lemma}[{Conditions on the Order \cite[Theorems 1.7.1 and 1.7.2]{maldy}}]
\label{lem:condsonst}
Consider a finite generalised quadrangle 
with order $(s,t)$. Then the following statements hold: 
\begin{enumerate}[(a)] 
\item \label{item:higman} (Higman's Inequality) $s \leqslant t^2$ and 
$t \leqslant s^2$; and 
\item \label{item:higman2} $s+t$ divides $st(s+1)(t+1)$.
\end{enumerate}
\end{lemma}

The following three lemmata provide some useful bounds. We note in
particular, \Cref{lem:bound} is an analogue of \cite[Lemma 2]{morgan} for 
generalised quadrangles.

\begin{lemma}[Bound on the Number of Points]
\label{lem:bound}
Let $a,s,t > 1$ be positive integers. 
\begin{enumerate}[(i)]
\item \label{item:s} If $2^a$ divides $(s+1)(s^2+1)$, then $(s+1)(s^2+1) >
2^{a-1}(2^{2a-4}+1)$. 
\item \label{item:st} If $2^a$ divides $s+1$ and $s \leqslant t^2$,
then $(s+1)(st+1) > 2^a(2^{(3a-3)/2} + 1)$. 
\end{enumerate}
\end{lemma}
\begin{proof}
We show (i). First, we claim that $s$ is odd. 
Indeed, if $s$ is even, then both $s+1$ and $s^2+1$
are odd contrary to the assumption that $(s+1)(s^2+1)$ is divisible by $2^a$.
Furthermore, writing $s = 2k+1$, we see that $s^2+1 = 4k^2+4k+2 \equiv 2 
\pmod 4$. Since $4$ does not divide $s^2+1$, we conclude that $2^{a-1}$
divides $s+1$. Therefore, $2^{a-1} \leqslant s+1$ and so, $2^{a-2} \leqslant 
s$. But since $s$ is odd, we can write $2^{a-2} < s$. From this, we conclude 
that 
\[
(s+1)(s^2+1) > 2^{a-1}(2^{2a-4}+1). 
\] 

Now for (ii), notice that $s$ is odd since $2^a$ divides $s+1$. Next, 
we observe that $2^{a} \leqslant s+1$ and so, $2^{a-1} \leqslant
s$. But since $s$ is odd, it follows that $2^{a-1} < s$. Also, we have that
$2^{a-1} < s \leqslant t^2$ by Higman's inequality (\Cref{lem:condsonst}
\ref{item:higman}), whence, 
$2^{(a-1)/2} < t$. From this, we deduce that 
\[
(s+1)(st+1) > 2^a(2^{a-1} 2^{(a-1)/2} + 1) = 2^a(2^{(3a-3)/2} + 1). \qedhere
\]
\end{proof}
 
\begin{lemma}
\label{lem:s2t}
Let $s$ and $t$ be integers greater than $1$ satisfying Higman's inequality
for generalised quadrangles, \emph{i.e.}, $s \leqslant t^2$ and 
$t \leqslant s^2$. Then 
\begin{enumerate}[(i)]
\item $s(t+1) < s^2t$; and \label{lem:s2ti}
\item $\ds \paren{\frac{2}{3}}^{4/3} 
\paren{s(t+1)}^{4/3} \leqslant s^2 t \leqslant \paren{s(t+1)}^{5/3}$.
\label{lem:s2tii}
\end{enumerate}
\end{lemma}
\begin{proof}
Since $s,t 
\geqslant 2$, we find that $1/(t+1) \leqslant 1/3$ and so, $-1/(t+1) \geqslant
-1/3$. Also, $2s/3 > 1$. Therefore, 
\[
\frac{st}{t+1} = s\paren{1-\frac{1}{t+1}} \geqslant s\paren{1-\frac{1}{3}}
\geqslant \frac{2s}{3} > 1.
\]
Consequently, $st > t+1$ and hence, $s^2 t > s(t+1)$. \\

To prove the second inequality, observe that, by Higman's inequality
(\Cref{lem:condsonst} \ref{item:higman}), 
\[
s^2 t = s^{2-k} s^k t \geqslant s^{2-k} t^{k/2} t = s^{2-k} t^{1+k/2},
\]
where $k > 0$. Now, substituting $k = 2/3$, we obtain 
$s^2 t \geqslant (st)^{4/3}$. Since $t \geqslant 2$, we have that $3t
\geqslant 2t + 2 = 2(t+1)$. Hence, $t \geqslant 2(t+1)/3$. Therefore,
\[
s^2 t \geqslant (st)^{4/3} \geqslant \paren{s \cdot \frac{2}{3} (t+1)}^{4/3} =
\paren{\frac{2}{3}}^{4/3} \paren{s(t+1)}^{4/3}.
\]
Similarly, 
\[
s^2 t = s^{2-k} s^k t \leqslant s^{2-k} t^{2k+1},
\]
by Higman's inequality (\Cref{lem:condsonst} \ref{item:higman}). 
Substituting $k = 1/3$, we obtain 
\[
s^2 t \leqslant (st)^{5/3} \leqslant (s(t+1))^{5/3}. \qedhere
\]
\end{proof}

The next result provides some inequalities relating the number of points
and the number of lines. 
\begin{lemma}
\label{lem:modhigman}
Let $\Gamma = (\cP,\cL,\cI)$ be a generalised quadrangle of order $(s,t)$
where $s,t \geqslant 2$. Then 
\begin{enumerate}[(i)]
\item \label{item:modhigman1} 
$|\cP| > (s+1)^2(t+1)/2$ and $|\cL| > (t+1)^2(s+1)/2$; and 
\item \label{item:modhighman2} 
$|\cP|^4 < |\cL|^5$ and $|\cL|^4 < |\cP|^5$.
\end{enumerate}
\end{lemma}
\begin{proof}
The proof of (i) can be found in \cite[Lemma 2.2 (1)]{luzhangzou}. We include
it here for completeness. Since $s,t > 1$, it follows that $st-s-t+1 = 
(s-1)(t-1) > 0$. Hence, $st+1 > s+t$. Thus, 
\[
st+1 > \frac{1}{2} (st+1) + \frac{1}{2} (s+t) = \frac{(s+1)(t+1)}{2}.
\]
From which, it immediately follows that $|\cP| = (s+1)(st+1) >
(s+1)^2(t+1)/2$. By duality, we also have $|\cL| = (t+1)^2(s+1)/2$. \\

We prove (ii). Consider the function $f(x) =
(x^{3/2}+1)(x^{1/2}+1)^5-(x+1)^4$ defined over the positive real numbers.
We show that $f(x) > 0$ for every $x \geqslant 1$. 
Using Mathematica (see \Cref{sec:derivative}), we find that the 
derivative of $f$ is given by 
\[
f^{\prime}(x) = 6+\frac{5+x(33+16x^{1/2} + 55x + 36x^{3/2} +
35x^2)}{2x^{1/2}} > 0
\]
for every $x > 0$. This proves that $f$ is strictly increasing for every $x >
0$. Next, we note that $f(1) = 48 > 0$. This proves that $f(x) > 0$ for every
$x \geqslant 1$. Thus, 
\begin{equation}
(x^{3/2}+1)(x^{1/2}+1)^5 > (x+1)^4 \label{ineq:xx}
\end{equation}
for every $x \geqslant 1$. Using Higman's inequality (\Cref{lem:condsonst}
\ref{item:higman}) and \eqref{ineq:xx}, we find that 
\[
(st+1)(t+1)^5 \geqslant (s^{3/2} + 1)(s^{1/2}+1)^5 > (s+1)^4. 
\]
Therefore, 
\[
|\cL| = (t+1)(st+1) > \paren{\frac{s+1}{t+1}}^4 =
\paren{\frac{|\cP|}{|\cL|}}^4.
\]
Multiplying through by $|\cL|^4$, we obtain $|\cP|^4 < |\cL|^5$. By duality,
we also obtain $|\cL|^4 < |\cP|^5$. 
\end{proof}
We note that 
\Cref{lem:modhigman} \ref{item:modhighman2} provides 
slightly sharper bounds than the ones found in \cite[Lemma 2.2 (2)]{luzhangzou}.
The bounds in that paper are equivalent to 
$|\cL|^4 < 2|\cP|^5$ and $|\cP|^4 < 2|\cL|^5$. \\

The next number theoretic lemma concerns the solutions of $s$ and $t$ given a
certain number of points and lines in a generalised quadrangle.  
\begin{lemma}
\label{lem:suzupts}
Let $a,b,s,t$ be positive integers and $p$ be a prime. If  
\begin{equation}
(s+1)(st+1) = p^a \ \textrm{and} \ (t+1)(st+1) = p^b,
\label{eqn:suzupts}
\end{equation}
then $p = 2$. Furthermore, we have: $(s = 1 \ \textrm{and} \ b = a/2+1)$ or
$(t = 1 \ \textrm{and} \ a = b/2+1)$.
\end{lemma}
\begin{proof}
Let $d \defi \gcd(s+1,t+1,st+1)$. Then $d$ divides $t(s+1)-(st+1) = t-1$.
Therefore, $d$ divides $t+1-(t-1) = 2$. Hence, $d = 1$ or $d = 2$. Since 
$s+1,t+1,st+1 \geqslant 2$, it follows that $p$ divides $s+1,t+1$ and $st+1$, 
and so $p$ divides $d$. From which, we deduce that $d = 2$ and $p = 2$. Now,
suppose that $s+1 \neq 2$ and $t+1 \neq 2$. Then $s+1 \equiv 0 \pmod 4$ and
$t+1 \equiv 0 \pmod 4$. Thus, $st+1 \equiv (-1)(-1)+1 \equiv 2 \pmod 4$ and 
so, $st+1=2$. 
This is a contradiction since $st+1 > s+1 > 2$. Therefore, $s+1 = 2$
or $t+1 = 2$. Without loss of generality, 
say $t = 1$. Substituting for $t$ and $p$ in \eqref{eqn:suzupts}, we obtain 
$(s+1)^2 = 2^a$ and $2(s+1) = 2^b$. 
Hence, $s+1 = 2^{a/2} = 2^{b-1}$, from which, we obtain $b = a/2+1$. If we 
consider the case where $s = 1$, then by the same argument, we obtain $a =
b/2+1$.
\end{proof}

The following lemma gives divisibility conditions that will be used later
along with \Cref{lem:condsonst} \ref{item:higman2} to provide bounds on the 
parameters of a generalised quadrangle. 

\begin{lemma}[Modified Divisibility Condition I]
\label{lem:moddivcond}
Let $A, s$ and $t$ be integers and $d \defi \gcd(s,t)$ such that $s+t$ 
divides $Ast$. Then 
$s+t$ divides $Ad^2$ and $A(s-t)^2$. Furthermore, if $A = 4a$ for some 
integer $a$, then $s+t$ divides $a(s-t)^2$. 
\end{lemma}
\begin{proof}
Suppose $k$ is an integer that divides $s+t$ and $s$. Then $0 \equiv s+t 
\equiv t \pmod k$. Thus, $k$ divides $t$ and consequently, $k$ divides $d$. 
Similarly, any integer $k$ that divides $s+t$ and $t$ must necessarily divide 
$d$. Hence, $s+t$ divides $Ad^2$. \\

Note that $s+t$ divides $As(s+t)-Ast = As^2+Ast-Ast = As^2$. Similarly, $s+t$
divides $At^2$. Therefore, $s+t$ divides $As^2-2Ast+At^2 = A(s-t)^2$. Now,
writing $A = 4a$, we see that $s+t$ divides $a(s+t)^2-4ast = a(s^2+2st+t^2) -
4ast = a(s-t)^2$. 
\end{proof}

Using \Cref{lem:moddivcond} and some further assumptions, we obtain another 
(stronger) divisibility condition on the parameters of a generalised 
quadrangle.
\begin{lemma}[Modified Divisibility Condition II]
\label{lem:moddivcond2}
Suppose $s$ and $t$ are the parameters of a thick generalised quadrangle.
Define $d \defi \gcd(s,t)$ and $k \defi \gcd(s+1,t+1)$. Suppose $k \equiv 2 
\pmod 4$. Write $s+1 = kx$ and $t+1 = ky$. Then $s+t$ divides $4xyd^2$.
\end{lemma}
\begin{proof}
By \Cref{lem:condsonst} \ref{item:higman2}, $s+t$ divides $st(s+1)(t+1) =
stk^2 xy$. Next, we observe that if $p$ is a prime that divides $s+t$ and 
$k$, then $p = 2$. 
This is because $s+1 = kx \equiv 0 \pmod p$ and $t+1 = ky \equiv 0 \pmod
p$ and so, $s+t \equiv -2 \pmod p$. Since $s+t \equiv 0 \pmod p$, we find that
$p = 2$. Now, because $k \equiv 2 \pmod 4$, the highest power of $2$ that
divides $k$ is $2$. Hence, $s+t$ divides $st \cdot 4xy$. Applying 
\Cref{lem:moddivcond} with $A \defi 4xy$, we deduce that $s+t$ divides 
$Ad^2 = 4xyd^2$. 
\end{proof}

\section{Group Actions}
\label{chap:groupactions}

\subsection{Schur--Zassenhaus Theorem}
In our study of groups of Lie type, we will come across subgroups that are
semidirect products of two groups with coprime order. A classical result that
will be useful in studying these subgroups is the Schur--Zassenhaus Theorem. 
\begin{theorem}[{Schur--Zassenhaus Theorem \cite[Section 3B]{isaacsfin}}]
Let $G$ be a group with a normal subgroup $N$ such that $\gcd(|N|,|G/N|) = 1$.
Then there exists a subgroup $K \subgp G$ such that $G = N \rtimes K$,
\emph{i.e.}, $G = NK$ and $N \cap K = \idgp$. We call $K$ a \emph{complement}
of $N$ in $G$. Moreover, all complements of $N$ in $G$ are conjugate to $K$.  
\end{theorem}

\begin{corollary}
\label{lem:schur}
Let $G = N \rtimes K$ where $\gcd(|N|,|K|) = 1$. Suppose $M \subgp G$
with $|M| = |K|$. Then $M = K^x$ for some $x \in G$. 
\end{corollary}
\begin{proof}
Observe that $N \cap M = \idgp$ since they have coprime orders. Thus,
$|NM| = |N||M|/|N \cap M| = |N||K| = |G|$. Hence,
$M$ is a complement of $N$ in $G$. By the Schur--Zassenhaus Theorem,
all complements of $N$ are conjugate and so, $M = K^x$ for some $x \in G$. 
\end{proof}

\begin{corollary}
\label{cor:conj}
Let $G = N \rtimes K$ where $\gcd(|N|,|K|) = 1$ and suppose $g \in G$ such
that $|g|$ divides $|K|$. Then $g \in K^h$ for some $h \in G$. In particular,
the number of conjugacy classes of elements with order $|g|$ in $G$ is at most 
the number of conjugacy classes of elements with order $|g|$ in $K$.
\end{corollary}
\begin{proof}
If $g = 1$, then the statement holds since there is only one conjugacy class
containing $1$. So let us suppose that $g \neq 1$. Now, observe that $g$ acts 
on $N$ by conjugation. 
Since $|g|$ divides $|K|$ and $g \neq 1$, it follows that $g \notin N$. 
Hence, we have the subgroup 
$M \defi N \rtimes \ang{g} \subgp G$. Note that $KM =
K(N\rtimes \ang{g}) = KN\ang{g} = G$ and $|G| = |K||M|/|K \cap M|$. Thus, 
\[
|K \cap M| = \frac{|K||M|}{|G|} = \frac{|K||N||\ang{g}|}{|G|} =
\frac{|K||N||g|}{|K||N|} = |g|,
\]
and so, $K \cap M$ is a complement of $N$ in $M$. 
By the Schur--Zassenhaus Theorem, all complements of $N$ in $M$ are 
conjugate. Therefore, $K \cap M = \ang{g}^x$ for some $x \in M$. 
Thus, $\ang{g} \subgp K^{x^{-1}}$. 
The number of conjugacy classes of elements with order $|g|$ is at most the
number in $K$ because any element of $G$ with order $|g|$ is conjugate (in
$G$) to an element of $K$. 
\end{proof}

\subsection{Fixed Point Theory}
Our investigation involves groups of Lie type acting on generalised
polygons. The following lemma provides a formula for calculating the
number of fixed points of an element with respect to a group action. 
\begin{lemma}[Formula for the Number of Fixed Points]
\label{lem:fixpts}
Let $G$ be a finite group acting transitively on a set $\Omega$. Let $\alpha
\in \Omega$ and $g \in G$. Then the number of fixed points of $g$, denoted 
$\pi(g)$, is given by
\begin{equation}
\pi(g) = \frac{|\Omega| |g^G \cap G_{\alpha}|}{|g^G|}.
\label{eqn:fixpts}
\end{equation}
\end{lemma}

\begin{lemma}[{\cite[Lemma 2.4]{luzhangzou}}]
\label{lem:cgtranss}
Let $G$ be a group acting transitively on a set $\Omega$ and $g$ be a
non-identity element in $G$. Consider a point $\alpha \in \Omega_g$, where
$\Omega_g$ is the set of fixed points of $g$. Then $\Centraliser{G}{g}$ acts
transitively on $\Omega_g$ if and only if $g^G \cap G_{\alpha} =
g^{G_{\alpha}}$, \emph{i.e.}, the conjugacy class of $g$ in $G$ does not split
into multiple conjugacy classes in $G_{\alpha}$. In this case, $|\Omega_g| =
|\Centraliser{G}{g} : \Centraliser{G}{g} \cap G_{\alpha}|$. 
\end{lemma}

\subsection{Action of Normal Subgroups}
The following lemma is an elementary result from permutation group theory that
provides conditions for when a normal subgroup acts transitively. 
\begin{lemma}
\label{lem:soctrans}
Let $G$ be a group acting primitively on a set $\Omega$ and take a point
$\alpha \in \Omega$. Suppose that $N \nsubgp G$ such that $N$ is non-trivial
and $N$ is not contained in $G_{\alpha}$. Then $N$ acts transitively on
$\Omega$. 
\end{lemma}
\begin{proof}
Since $N$ is not contained in $G_{\alpha}$, it follows that 
$G_{\alpha} \propsubgp G_{\alpha} N \subgp G$. Because $G$ acts 
primitively on $\Omega$, the point-stabiliser $G_{\alpha}$ is
maximal in $G$ and so, $G_{\alpha} N = G$. Now, take two arbitrary points
$\beta, \gamma \in \Omega$. Since $G$ is transitive on $\Omega$, there exist
$g_1, g_2 \in G$ such that $\alpha^{g_1} = \beta$ and $\alpha^{g_2} = \gamma$.
Since $G = G_{\alpha} N$, we can write $g_1 = a_1 n_1$ and $g_2 = a_2 n_2$ for
some $a_1,a_2 \in G_{\alpha}$ and $n_1, n_2 \in N$. Thus, $\beta =
\alpha^{g_1} = \alpha^{a_1 n_1} = \alpha^{n_1}$ and $\gamma = \alpha^{g_2} =
\alpha^{a_2 n_2} = \alpha^{n_2}$. As a result, we find that $\gamma =
\alpha^{n_2} = (\beta^{n_1^{-1}})^{n_2} = \beta^{n_1^{-1} n_2}$. Hence, we
have found an element in $N$ that maps $\beta$ to $\gamma$. Since $\beta$ and
$\gamma$ are arbitrary, it follows that $N$ is transitive on $\Omega$.
\end{proof}
We use \Cref{lem:soctrans} in the case where $G$ is an almost simple group 
acting primitively on a set and $N$ is the socle of $G$. 

\subsection{Groups Acting on Generalised Quadrangles}
We now present some preliminary lemmata regarding groups acting on generalised
quadrangles. 
\begin{lemma}[{\cite[Corollary 2.3]{fenglu}}]
\label{lem:suzsub}
Let $G$ be a group acting on a generalised quadrangle $\Gamma$ and suppose $g
\in G$. 
Let $\Gamma_g = (\cP_g,\cL_g,\cI_g)$ be the fixed substructure of $g$. If 
$|\cP_g| \geqslant 2$, $|\cL_g| \geqslant 2$ and $\Gamma_g$ admits a group of
automorphisms $H$ that is transitive on both points and lines, then
$\Gamma_g$ is a generalised quadrangle of order $(s^{\prime},t^{\prime})$ for
some positive integers $s^{\prime}$ and $t^{\prime}$.
\end{lemma}

\begin{lemma}
\label{lem:2trans}
A group $G$ cannot act $2$-transitively on the set of points of a generalised 
quadrangle.
\end{lemma}
\begin{proof}
Suppose $G$ acts $2$-transitively on the set of points of a generalised 
quadrangle $\Gamma = (\cP,\cL,\cI)$. Then consider two collinear points 
$(\alpha,\beta)$ 
and two non-collinear points $(\gamma,\delta)$. By $2$-transitivity, there
exists an element $g \in G$ such that $(\alpha,\beta)^g = (\gamma,\delta)$,
which means that $G$ does not preserve the incidence relation. Therefore, $G$
cannot act $2$-transitively on the set of points of a generalised quadrangle. 
\end{proof}

In \Cref{chap:largeree}, we consider almost simple 
groups of type Suzuki, small Ree and large Ree acting primitively on both the
set of points and the set of lines of a generalised quadrangle. We first show
that the simple group itself acts primitively on the set of points and the set
of lines of the generalised quadrangle. We start by recalling the definition 
of an almost simple group. 

\begin{definition}[Almost Simple Group]
Let $N$ be a non-abelian simple group. A group $G$ is \textsf{almost simple}
(of type $N$) if $N \cong \Inn(N) \nsubgp G \subgp \Aut(N)$. 
\end{definition}
We will often identify $\Inn(N)$ by $N$ and write $N \subgp G$. Note also that
$N$ is the minimal normal subgroup of $G$, and so, the socle of $G$ is $N$.
Next, we recall the notion of a novelty subgroup. This concept will be 
required later when we wish to show that a simple group $N$ acts primitively
on a set.

\begin{definition}[Novelty Subgroup]
\label{novelty subgroup}
Let $G$ be a group with a subgroup $N$ and a maximal subgroup $M$. We call $M$
a \textsf{novelty subgroup} if $M \cap N$ is not maximal in $N$. 
\end{definition}

Now we state and prove the results which allow us to consider the action of a
simple group instead of the action of the almost simple group. 
\begin{lemma}[Reduction to Simple Group]
\label{lem:reductiongen}
Let $G$ be an almost simple group with socle $N$ where the outer
automorphism group is cyclic, \emph{i.e.}, $\Out(N) = \Aut(N)/\Inn(N) \cong
C_n$ for some positive integer $n$. Suppose also that $G$ has no novelty
subgroups. If $G$ acts primitively on a set 
$\Omega$ where $|\Omega|$ is not prime, then $N$ acts primitively on $\Omega$.
\end{lemma}
\begin{proof}
We follow the argument in \cite{morgan}. 
Take a point $\alpha \in \Omega$. If $N \subgp G_{\alpha}$, then
$G_{\alpha}$ is of the form $N: C_k$ where $k$ divides $n$. Since $G$ acts
primitively on the set of points, $n/k$ must be a prime, $p$. Thus,
$p = |G:G_{\alpha}| = |\Omega|$, which is a contradiction since $|\Omega|$ is
not prime. Therefore, $N$ is not contained in $G_{\alpha}$. 
By \Cref{lem:soctrans}, $N$ acts transitively on
$\Omega$. Since $G$ has no novelty maximal subgroups, it follows that 
$G_{\alpha} \cap N$ is a maximal subgroup of $N$. Thus, $N$ acts primitively 
on $\Omega$. 
\end{proof}

\begin{corollary}[Reduction to Simple Group]
\label{cor:reduction}
Let $G$ be an almost simple group with socle $N$ where $\Out(N)$ is cyclic.
Suppose $G$ has no novelty subgroups. Let $\Gamma = (\cP,\cL,\cI)$ be a 
generalised quadrangle of
order $(s,t)$ where $s,t \geqslant 2$. If $G$ acts primitively on $\cP$
and $\cL$, then $N$ acts primitively on $\cP$ and $\cL$. 
\end{corollary}
\begin{proof}
Note that $|\cP| = (s+1)(st+1)$ and $|\cL| =
(t+1)(st+1)$ are not prime numbers since $s+1$, $t+1$ and $st+1$ are greater
than $2$. Therefore, by \Cref{lem:reductiongen}, $N$ acts primitively on 
$\cP$ and $\cL$. 
\end{proof}

\section{Large Ree Groups}
\label{chap:largeree}
In this section, we consider Ree groups of type $F_4$ \cite{reef4}, 
which we shall refer to as
large Ree groups. Throughout this section, we denote a large Ree group by 
$\LargeRee{q}$, where $q = 2^{2m+1}$ for some positive integer $m$. 

\subsection{Structure of Large Ree Groups}
In this subsection, we provide some background information on the family of 
large Ree groups. We use this information in \Cref{sec:largereeacts} to 
show that an almost simple group with socle isomorphic
to a large Ree group cannot act primitively on the point-set and 
the line-set of a generalised quadrangle. \\

Our first theorem involves the maximal subgroups of $\LargeRee{q}$. 
Malle (1991) \cite{mallef4} gave a list of maximal subgroups of 
$\LargeRee{q}$. However, Craven (2023) \cite{newf4} found that one maximal 
subgroup was missing in Malle's list. More specifically, Craven found that 
there are three conjugacy classes of maximal subgroups isomorphic to 
$\PGLn{2}{13}$ in $\LargeRee{q}$ when $q = 8$. 
The complete list is provided in the following theorem. 
\begin{theorem}[{Maximal Subgroups of a Large Ree Group \cite{newf4,mallef4}}]
\label{thm:largereemaxs} \index{maximal subgroups of!large Ree group}
Let $G = \LargeRee{q}$. Then a maximal subgroup of $G$ is 
isomorphic to one of the following: 
\begin{enumerate}[(i)]
\item $P_1 \defi [q^{10}] : (\Sz{q} \times C_{q-1})$; 
\item $P_2 \defi [q^{11}] : \GLn{2}{q}$; 
\item $\SUn{3}{q} : C_2$; 
\item $\PGUn{3}{q} : C_2$;
\item $\Sz{q} \wr C_2$; 
\item $\SPn{4}{q} : C_2$;
\item $\LargeRee{q_0}$, where $q_0 = 2^{n_0} > 2$ and $q = q_0^{r_0}$ for some
prime $r_0$; 
\item $C_{q+1}^2 : \GLn{2}{3}$; 
\item $C_{q +\epsilon \sqrt{2q} + 1}^2 : 4S_4$, where $\epsilon \in \{1,-1\}$; 
\item $C_{q^2+q+1+\epsilon\sqrt{2q}(q+1)} : C_{12}$, where $\epsilon \in \{1,-1\}$; and 
\item $\PGLn{2}{13}$ if $q = 8$.
\end{enumerate}
Moreover, there is one conjugacy class of subgroups for each of the cases 
(i)-(x) and three conjugacy classes of subgroups for case (xi). 
\end{theorem}

In the case where $G$ is a Suzuki or a small Ree group acting transitively on
a set $\Omega$ with a parabolic subgroup as the point-stabiliser, it is
well-known that the action is $2$-transitive. This is not the case with the
large Ree groups. The subdegrees are given by the following lemma.

\begin{lemma}[Subdegrees of Parabolic Maximal Subgroups]
\label{lem:largereesubdegrees}
Let $G = \LargeRee{q}$ act primitively on a set $\Omega$. Consider a point
$\alpha \in \Omega$. If $G_{\alpha} \cong P_1 = [q^{10}] : (\Sz{q} \times
C_{q-1})$, then the subdegrees are 
\[
n_1 \defi q(q^2+1), \ n_2 \defi q^4(q^2+1), \ n_3 \defi q^7(q^2+1) \ 
\textrm{and}  \ n_4 \defi q^{10}.
\] 
If $G_{\alpha} \cong P_2 = [q^{11}] : \GLn{2}{q}$, then the subdegrees are 
\[
n_1 \defi q^2(q+1), \ n_2 \defi q^5(q+1), \ 
n_3 \defi q^8(q+1) \ \textrm{and} \ n_4 \defi q^{11}. 
\]
Moreover, in both cases, we have $n_1 + n_2 + n_3 \leqslant n_4$.
\end{lemma}
\begin{proof}
The subdegrees can be found in \cite[Section 4.9.4]{wilson}. We prove just the
inequality. Consider $G_{\alpha} \cong P_1$. Then we find 
$n_1 + n_2 + n_3 = q(q^2 +1) + q^4(q^2 + 1) + q^7(q^2+1) = (q+q^4+q^7)(q^2+1)
\leqslant (q^7+q^7+q^7)(q^2+q^2) = 6q^9 \leqslant 8q^9 \leqslant q^{10} =
n_4$, since $m \geqslant 1$ and $q = 2^{2m+1} \geqslant 2^3 = 8$. 
Now consider $G_{\alpha} \cong P_2$. Then we obtain 
$n_1+n_3+n_3 = q^2(q+1) + q^5(q+1) + q^8(q+1) = (q^2+q^5+q^8)(q+1) \leqslant
(q^8+q^8+q^8)(q+q) = 6q^9 \leqslant 8 q^9 \leqslant q^{11} =
n_4$. 
\end{proof}

Many of the arguments in this section make use of bounds on the number of
points and lines of a generalised quadrangle. For this reason, we have 
\Cref{tab:bounds}. 

\begin{lemma}[Upper and Lower Bounds on the Index]
Let $G = \LargeRee{q}$ and $H$ be a maximal subgroup of 
$G$. Then \Cref{tab:bounds} contains lower and upper bounds for $|G:H|$. 
\afterpage{
\begin{longtable}{lllll}
\caption{Lower and upper bounds for $|G:H|$.}
\label{tab:bounds}
\endfirsthead
\endhead
\toprule
\thead{$i$} & 
\thead{$H$} & \thead{$|G:H|$} & \thead{Lower} & \thead{Upper} \\
\midrule
1&\makecell{$\SUn{3}{q}:C_2$ \\ or $\PGUn{3}{q}:C_2$} & 
$\ds \frac{1}{2} q^9 (q^6+1)(q^2+1)(q-1)$  
& $\ds \frac{q^{18}}{4}$  & $\ds \frac{q^{18}}{2}$ \\
\midrule
2&$\Sz{q} \wr C_2$ & $\ds \frac{1}{2} q^8 (q^4-q^2+1)(q^3+1)(q+1)$  & $\ds
\frac{q^{16}}{2}$ & $q^{16}$ \\
\midrule
3&$\SPn{4}{q} : C_2$ & $\ds \frac{1}{2} q^8 (q^6+1)(q^2-q+1)$ & $\ds
\frac{q^{16}}{4}$ & $\ds
\frac{q^{16}}{2}$ \\
\midrule
4&$\LargeRee{q_0}$ & $\ds
\frac{q^{12}(q^6+1)(q^4-1)(q^3+1)(q-1)}{q_0^{12} 
(q_0^6+1)(q_0^4-1)(q_0^3+1)(q_0-1)}$ & $\ds \frac{1}{2}
\paren{\frac{q}{q_0}}^{26}$ & $\ds 2\paren{\frac{q}{q_0}}^{26}$ \\
\midrule
5&$C_{q+1}^2:\GLn{2}{3}$ & $\ds \frac{1}{48}
q^{12}(q^6+1)(q^2+1)(q-1)^2(q^2-q+1)$ & $\ds \frac{q^{24}}{96}$  & 
$\ds \frac{q^{24}}{48}$ \\
\midrule
6&$C_{q+\epsilon\sqrt{2q}+1}^2 : 4S_4$ & \makecell{$\ds \frac{1}{96} q^{12}
(q-\epsilon\sqrt{2q}+1)^2 (q^4-q^2+1)(q^2-1)$ \\ $\qquad \qquad \qquad \qquad 
(q^3+1)(q-1)$} & $\ds \frac{q^{24}}{384}$ & $\ds \frac{q^{24}}{48}$ \\
\midrule
7&$C_{q^2+q+1+\epsilon\sqrt{2q}(q+1)} : C_{12}$ & \makecell{$\ds \frac{1}{12} q^{12}
(q^2+1)(q^2+q+1-\epsilon\sqrt{2q}(q+1))$ \\ $\qquad \qquad 
\qquad \qquad (q^4-1)(q^3+1)(q-1)$} & $\ds \frac{q^{24}}{24}$ & 
$\ds \frac{q^{24}}{6}$ \\
\bottomrule
\end{longtable}
}
\end{lemma}
\begin{proof}
We use Mathematica to prove this result. See \Cref{sec:bounds} for the code. 
For $i \in \{1,2,3,5\}$, we let $f(q) = |G:H|$. Then we take 
$g(q) = q^n$ where $n$ is the degree of $f$. 
We use the \textsc{Reduce} command in Mathematica to find real numbers 
$a,b>0$ such that 
$ag(q) \leqslant f(q) \leqslant bg(q)$. For $i = 4$, we let $f(q) =
q^{12}(q^6+1)(q^4-1)(q^3+1)(q-1)$ and $g(q) = q^{26}$. Using the
\textsc{Reduce} command in Mathematica again, we find that $g(q) \leqslant
f(q) \leqslant 2g(q)$. Thus, 
\[
\frac{g(q)}{2g(q_0)} \leqslant \frac{f(q)}{f(q_0)} \leqslant
\frac{2g(q)}{g(q_0)},
\]
which is the desired bound for $i = 4$. Finally, for $i \in \{6,7\}$, we let
$u = \sqrt{2q}$ and write $|G:H|$ as a polynomial in $u$. Letting $f(u) =
|G:H|$ and $g(u) = q^{24} = u^{48}/2^{24}$, we use the \textsc{Reduce} command
in Mathematica to find the desired lower and upper bounds.
\end{proof}

In the following lemma, 
we collect some information about conjugacy classes of 
various elements in $\LargeRee{q}$. This is used when we consider the set 
of fixed points of an element later.
\begin{lemma}[{Conjugacy Classes and Centralisers \cite{largereeccls}}]
\label{lem:largereeccls}
Let $G = \LargeRee{q}$ and $H \subgp G$. Then \Cref{tab:ccls} contains 
the centraliser orders of various elements in $G$. 
The representatives chosen can be found in 
\cite[Tables II and IV]{largereeccls}. For each $g$ in \Cref{tab:ccls}, 
we have $g^G \cap H \neq \varnothing$, 
\emph{i.e.}, the conjugacy class of $g$ in $G$ meets $H$. 
Furthermore, taking $h \in g^G \cap H$, we have that 
$h^G \cap H = h^H$, \emph{i.e.}, 
the conjugacy class of $g$ in $G$ does not split in $H$. 
\begin{longtable}{lllll}
\caption{Conjugacy classes of elements and centraliser orders. 
Here, $\epsilon \in \{1,-1\}$ and 
$q_0 = 2^{n_0} > 2$ and $q = q_0^{r_0}$ for some prime $r_0$.
\label{tab:ccls}} \\
\toprule
\thead{$g$} & \thead{$|\Centraliser{G}{h}|$} 
& $H$ & \thead{$|\Centraliser{H}{h}|$} \\
\midrule
$u_2$ & $q^{10}(q^2-1)$ 
& $\LargeRee{q_0}$ & $q_0^{10}(q_0^2-1)$ \\
& & $C_{q^2+q+1+\epsilon\sqrt{2q}(q+1)} : C_{12}$ & $12$ \\
\midrule
$u_5$ & $q^8$ & $\LargeRee{q_0}$ & $q_0^8$ \\
\midrule
$t_4$ & $q^3(q^2-1)(q^3+1)$ 
& $\LargeRee{q_0}$ & $q_0^3(q_0^2-1)(q_0^3+1)$ \\
 & & $C_{q+\epsilon\sqrt{2q}+1}^2 : 4S_4$ & $12$ \\
\bottomrule
\end{longtable}
\end{lemma}
\begin{proof}
Let $G = \LargeRee{q}$ and $H \cong \LargeRee{q_0}$. 
Observe that the unipotent conjugacy class representatives in 
\cite[Table II]{largereeccls} lie in the base field subgroup since they are all written as products of elements of the form $\alpha_i(1)$. 
Hence, $u_2^G \cap H$ and $u_5^G \cap H$ are non-empty sets. Now, we also observe that $u_2^G$ and $u_5^G$ do not split in $H$ since $H$ has the same 
unipotent classes and all of the unipotent classes meet $H$. Finally, the 
centraliser orders can be found in \cite[Table II]{largereeccls} with the convention 
that our $q$ corresponds to their $q^2$. \\

Next, we consider $t_4$. Note that $t_4$ is a representative of the 
unique conjugacy class of elements of order $3$ in $G$ 
(\cite[Section 1.2]{mallef4}). Thus, both 
$G$ and $H$ have a unique conjugacy class of elements of 
order $3$ and so, the conjugacy class of $t_4$ in $G$ has a non-empty 
intersection with $H$ and it does not split in $G$, \emph{i.e.}, 
we have $t_4^G \cap H = h^H$ for any $h \in H$ that is conjugate to 
$t_4$. 
Finally, the centraliser orders can be found in 
\cite[Tables II and IV]{largereeccls} with the convention that our 
$q$ corresponds to their $q^2$. \\

Let $H = K:L 
\cong C_{q^2+q+1+\epsilon\sqrt{2q}(q+1)} : C_{12}$.
The element $u_2$ is an involution in class $(2B)$ 
\cite[Lemma 8]{f4conjcls}. Note that an involution in $L$ 
centralises an element of order $3$ in $L$. Thus, $H$ contains 
involutions from class $(2B)$. 
Next, notice that $q^2+q+1+\epsilon\sqrt{2q}(q+1)$ is odd and so, 
by \Cref{cor:conj}, the number of
conjugacy classes of involutions in $H$ is at most the number of conjugacy
classes in $L$. Since $L$ has a unique conjugacy class of
involutions, we conclude that $H$ also has only one conjugacy class of
involutions. Therefore, $u_2^G$ does not split in $H$. 
Let $h \in u_2^G \cap H$. 
Since $q^2+q+1+\epsilon\sqrt{2q}(q+1)$ is coprime to 
$q^2-1$, it follows that $\Centraliser{H}{h} \cong C_{12}$, and so, 
$|\Centraliser{G_{H}}{h}| = 12$. \\

Let $H = C_{q+\epsilon\sqrt{2q}+1}^2 : M$, where $M \cong 4S_4$ and let 
$K$ be a normal subgroup of $M$ of order $4$. 
We show that $H$ has a unique conjugacy class of elements of order $3$. Since 
$3$ does not divide $q +\epsilon\sqrt{2q} +1$, the Sylow 
$3$-subgroup of $H$ has 
order $3$. Suppose $h \in M$ is of order $3$. Then every element in $H$ 
is conjugate to $h$ or $h^{-1}$ since all Sylow $p$-subgroups are conjugate. 
We now prove that $h$ and $h^{-1}$ are conjugate in $M$ and consequently,
$h$ and $h^{-1}$ are conjugate in $H$. To this end, we first observe the
following: every element in $K$ centralises $h$. This is because $K \nsubgp
M$ and so, $h$ induces an inner automorphism $\theta_h$ of $K$ via
conjugation. However, $\Aut(K) \cong \Aut(C_4) \cong C_2$. Thus, $\theta_h \in
\Aut(K)$ must be the identity map. Therefore, $h$ centralises every element of
$K$. Next, in $M/K \cong S_4$, we find that $Kh$ is conjugate to 
$Kh^{-1}$. This is because $S_4$ has precisely one conjugacy class of elements
of order $3$ (with representative $(1 \ 2 \ 3)$, for instance). Therefore,
there exists an element $x \in M$ such that $(Kh)^{Kx} = Kh^{-1}$,
\emph{i.e.}, $Kh^x = Kh^{-1}$. Thus, $h^x = kh^{-1}$ for some $k \in K$. From
which, we find that 
\[
k^3 = k^3 h^{-3} = (kh^{-1})^3 = (x^{-1} h x)^3 = x^{-1} h^3 x = 1,
\] 
since $h$ centralises all elements of $K$. Thus, $|k| = 1$ or $3$. We cannot
have $|k| = 3$ as $3$ does not divide $|K| = 4$. Thus, $|k| = 1$. In other
words, $h^x = kh^{-1} = h^{-1}$, \emph{i.e.}, $h$ is conjugate in $M$ to
$h^{-1}$. This proves that all elements in $M$ of order $3$ are conjugate.
Consequently, there is a unique conjugacy class of elements of order $3$ in $H$. 
Since $G$ has a unique conjugacy class of elements of order $3$, it follows that 
the conjugacy class of elements of order $3$ in $G$ does not split in
$H$. 
Finally, since $K$ centralises $h$, it follows that $12$ divides $|\Centraliser{H}{h}|$. 
Now, because $|\Centraliser{G}{h}| = q^3(q^2-1)(q^3+1)$ and 
$q+\epsilon\sqrt{2q}+1$ are coprime, it follows that $|\Centraliser{H}{h}| = 12$.
\end{proof}

\subsection{Divisibility Conditions and Coprimality}
The final few arguments in this section make use of the following divisibility 
conditions. 

\begin{lemma}[Polynomial Divisibility Conditions]
\label{lem:groupdivcond}
Let $q = 2^{2m+1}$ for some integer $m \geqslant 1$ and 
$\epsilon \in \{1,-1\}$. 
Then the following three statements hold: 
\begin{enumerate}[(i)]
\item $q+1$ divides $\ds\frac{1}{48} q^{12}(q^6+1)(q^2+1)(q-1)^2(q^2-q+1)-1$; 
\item $q+\epsilon\sqrt{2q}+1$ divides $\ds\frac{1}{96} q^{12}
(q-\epsilon\sqrt{2q}+1)^2(q^4-q^2+1)(q^2-1)(q^3+1)(q-1)-1$; and 
\item $q^2+q+1+\epsilon\sqrt{2q}(q+1)$ divides $\ds
\frac{1}{12}q^{12}(q^2+1)(q^2+q+1-
\epsilon\sqrt{2q}(q+1))(q^4-1)(q^3+1)(q-1)-1$. 
\end{enumerate}
\end{lemma}
\begin{proof}
We use Mathematica to perform various computations and algebraic
manipulations. All the lines in brackets refer to the code in
\Cref{sec:quotient} \Cref{lis:quotient}. 
We show (i). Let $f(q) = q^{12}(q^6+1)(q^2+1)(q-1)^2(q^2-q+1)-48$. Then $f(-1)
= 0$ (see line \ref{line:fminus1}). Hence, $q+1$ divides $f(q)$ as a 
polynomial. Next, we use the
\textsc{PolynomialQuotientRemainder} command in Mathematica to find the 
quotient after dividing $f(q)$ by $q+1$ (line \ref{line:f2}): 
\begin{align*}
g(q) = \frac{f(q)}{q+1} =
q^{23}&-4 q^{22}+9 q^{21}-15 q^{20}+20 q^{19}-23 q^{18}+25
   q^{17}-28 q^{16}+33 q^{15}-39 q^{14} \\ 
   &+44 q^{13}-47 q^{12} +48 q^{11}-48 q^{10}+48 q^9-48 q^8+48 q^7-48 q^6 +48 q^5 \\ 
&-48 q^4+48 q^3-48 q^2+48 q-48.
\end{align*}
Since $q \equiv -1 \pmod 3$, we find that $g(q) \equiv -1-4-20-23-25-28-44-47 
\equiv -192 \equiv 0 \pmod 3$. Moreover, because $q = 2^{2m+1}$, it follows 
that $g(q)$ is divisible by $16$.
Therefore, $g(q)$ is divisible by $48$. This proves that $g(q)/48$ is an
integer and we are done. \\

For (ii), let $u = \sqrt{2q} = 2^{m+1}$. Then $q = u^2/2$ and so, we can
define $f_{\epsilon}(u) = q^{12}(q -\epsilon u+1)^2
(q^4-q^2+1)(q^2-1)(q^3+1)(q-1)-96$ and
$g_{\epsilon}(u) = q+\epsilon u + 1 = u^2/2+\epsilon u +1$. The roots of 
$g_{\epsilon}(u)$ are: $-\epsilon+i$ and $\epsilon - i$ (line \ref{line:f3}). 
Observe that $f_{\epsilon}(-\epsilon + i) = 0$ (line \ref{line:f4}). Since
$f_{\epsilon}(u)$ is a polynomial with real coefficients, 
$f_{\epsilon}(-\epsilon - i) = 0$.
Hence, $g_{\epsilon}(u)$ divides $f_{\epsilon}(u)$ as a polynomial. 
Now, the quotient 
after dividing $f_{\epsilon}(u)$ by $g_{\epsilon}(u)$ is given by (line
\ref{line:f5}): 
\begin{align*}
\frac{f_{\epsilon}(u)}{g_{\epsilon}(u)} = \frac{u^{46}}{2^{23}}& -\epsilon 
\frac{3
   u^{45}}{2^{22}}+\frac{u^{44}}{2^{19}} -\epsilon \frac{5
   u^{43}}{2^{21}}-\frac{3 u^{42}}{2^{21}} +\epsilon \frac{7
   u^{41}}{2^{19}}-\frac{31 u^{40}}{2^{20}} +\epsilon \frac{15
   u^{39}}{2^{19}} \\
   & +\frac{3 u^{38}}{2^{17}} -\epsilon \frac{35
   u^{37}}{2^{18}}+\frac{61 u^{36}}{2^{18}} -\epsilon \frac{5
   u^{35}}{2^{15}}-\frac{35 u^{34}}{2^{17}} +\epsilon \frac{61
   u^{33}}{2^{16}}-\frac{21 u^{32}}{2^{14}} +\epsilon \frac{15
   u^{31}}{2^{15}} \\
   & +\frac{65 u^{30}}{2^{15}} -\epsilon \frac{41
   u^{29}}{2^{13}}+\frac{93 u^{28}}{2^{14}} -\epsilon \frac{5
   u^{27}}{2^{13}}-\frac{11 u^{26}}{2^{10}} +\epsilon \frac{93
   u^{25}}{2^{12}}-\frac{95 u^{24}}{2^{12}}+\frac{3
   u^{22}}{2^6} -\epsilon \frac{3 u^{21}}{2^5} \\ 
   & +\frac{3 u^{20}}{2^5}-\frac{3
   u^{18}}{2^4}+\epsilon \frac{3 u^{17}}{2^3}-\frac{3 u^{16}}{2^3}+\frac{3
   u^{14}}{2^2}-\epsilon \frac{3 u^{13}}{2}+\frac{3 u^{12}}{2}-3 u^{10}
+\epsilon 6
   u^9-6 u^8+12 u^6 \\ 
   &  -\epsilon 24 u^5+24 u^4-48 u^2 +\epsilon 96 u-96.
\end{align*}
Since $u = 2^{m+1}$, the quotient $f_{\epsilon}(u)/g_{\epsilon}(u)$ is an 
integer.
Moreover, we note that $u = 2^{m+1} \equiv (-1)^{m+1} \pmod 3$ and thus,
modulo $3$, we have  
{\small
\begin{align*}
\frac{f_{\epsilon}(u)}{g_{\epsilon}(u)} \equiv{}& \frac{u^{46}}{2^{23}} 
+\frac{u^{44}}{2^{19}} -\epsilon \frac{5
   u^{43}}{2^{21}} +\epsilon \frac{7
   u^{41}}{2^{19}}-\frac{31 u^{40}}{2^{20}} \\
  & \quad \ -\epsilon \frac{35
   u^{37}}{2^{18}}+\frac{61 u^{36}}{2^{18}} -\epsilon \frac{5
   u^{35}}{2^{15}}-\frac{35 u^{34}}{2^{17}} +\epsilon \frac{61
   u^{33}}{2^{16}} \\
  & \quad \ +\frac{65 u^{30}}{2^{15}} -\epsilon \frac{41
   u^{29}}{2^{13}} -\epsilon \frac{5
   u^{27}}{2^{13}}-\frac{11 u^{26}}{2^{10}} -\frac{95 u^{24}}{2^{12}} \\
\equiv{}& (-1)+(-1)-\epsilon (-1)^{m+1} +\epsilon (-1)^{m} - (1) \\ 
& \quad \ -\epsilon (-1)^{m} +(1) -\epsilon (-1)^{m+1} -(1) +\epsilon 
(-1)^{m+1} \\ 
& \quad \ +(1) -\epsilon (-1)^{m+1} -\epsilon (-1)^{m+1} - (-1)-(-1) \\
\equiv{}& \epsilon (-1)^{m} +\epsilon (-1)^{m} \\ 
& \quad \ -\epsilon (-1)^{m} +\epsilon (-1)^{m} -\epsilon (-1)^{m} \\ 
& \quad \ +\epsilon (-1)^{m} +\epsilon (-1)^{m} \\
\equiv{}& 0 \pmod 3.  
\end{align*}
}
Since $u = 2^{m+1}$, it follows that $f_{\epsilon}(u)/g_{\epsilon}(u)$ is 
divisible by $32$.
Therefore, $f_{\epsilon}(u)/(96g_{\epsilon}(u))$ is an integer and we are 
done. \\

Finally, we prove (iii). For this, we consider the subgroup $H = 
C_{q^2+q+1+\epsilon\sqrt{2q}(q+1)} : C_{12}$ in $G = \LargeRee{q}$. Notice that
$k \defi q^2+q+1+\epsilon\sqrt{2q}(q+1)$ is coprime to $12$. Indeed, since 
$q+1 \equiv 0 \pmod 3$, it follows that 
$q^2+q+1+\epsilon\sqrt{2q}(q+1) \equiv (-1)^2+(-1)+1+\epsilon 0 \equiv 1 
\pmod 3$. Hence,
$k$ is coprime to $3$. Furthermore,
$k$ is odd, which shows that it is coprime to $12$.  
Hence, $C_{q^2+q+1+\epsilon\sqrt{2q}(q+1)}$ is characteristic in $H$. Now 
consider
the action of $G$ on the set of right cosets of $G$ in $H$ by right 
multiplication. Denote the set of right cosets by $\Omega$. 
Then the number of points $|\Omega|$ in this action is given by $|G:H|$. Let
$\alpha = H$ and $\beta \in \Omega$ with $\beta \neq \alpha$. Then $G_{\alpha}
= H$ and $G_{\beta} = H^g$ for some $g \in G$. Consider the action of $H$ on
$\Omega$. The suborbit with $\beta$ has length $|H:H_{\beta}| = 
|H:H \cap H^g|$.
We claim that $|H \cap H^g|$ divides $12$. Suppose it does not for the sake of
deriving a contradiction. Then there is a prime $p$ dividing $|H\cap H^g|$
such that $p$ also divides $k$. Hence, there exists an
element $x \in H \cap H^g$ with order $p$. Moreover, the subgroup $\ang{x}$ is
characteristic in $K$ for some $K \cong C_{q^2+q+1+\epsilon\sqrt{2q}(q+1)}$. 
This
means that $\ang{x} \nsubgp H$ because $K \nsubgp H$. Since $H$ is maximal in
$G$ and $G$ is simple, it follows that $\Normaliser{G}{\ang{x}} = H$. By the
same argument, we have that $\Normaliser{G}{\ang{x}} = H^g$, whence, $H =
H^g$, \emph{i.e.}, $|H:H_{\beta}| = |H:H\cap H^g| = 1$. Thus, $H = H_{\beta}$. 
This contradicts the assumption that $\alpha \neq \beta$. Therefore, $|H
\cap H^g|$ divides $12$. Thus, $k$ divides
$|H:H_{\beta}|$ and so, $k$ divides the length of
each suborbit that does not contain $\alpha$. Consequently, 
$k$ divides the disjoint union of these suborbits,
which has total length $|\Omega|-1$. We are done since
\[
|\Omega|-1 = |G:H|-1 =
\frac{1}{12}q^{12}(q^2+1)(q^2+q+1-\epsilon\sqrt{2q}(q+1))(q^4-1)(q^3+1)(q-1)-1. 
\qedhere
\]
\end{proof}

The following number theoretic lemmata are useful in solving for the 
parameters of a generalised quadrangle.
\begin{lemma}
\label{lem:amsans}
Let $a,m$ and $n$ be positive integers and $h \defi \gcd(m,n)$. Write 
$m = hm_0$ and $n = hn_0$. Then the following statements hold.
\begin{enumerate}[(i)]
\item $\gcd(a^m-1,a^n-1) = a^{h}-1$.
\item Suppose (without loss of generality) that $m_0$ is odd. Then 
$\gcd(a^m+1,a^n+1) = \gcd(a^h+1,(-1)^{n_0}+1)$, with the convention that
$\gcd(a^h+1,0) = a^h+1$. 
\item Suppose that $m$ is odd. Then 
\[
\gcd(a^m-1,a^n+1) = \begin{cases}
1 & \textrm{if $a$ is even}, \\
2 & \textrm{if $a$ is odd}. 
\end{cases}
\]
\end{enumerate}
\end{lemma}
\begin{proof}
To show part (i), let $d \defi \gcd(a^m-1,a^n-1)$. 
Then $a^m \equiv 1 \pmod d$ and $a^n \equiv 1
\pmod d$. By Bezout's Lemma, there exist $x,y \in \bZ$ such that
$mx+ny = \gcd(m,n) = h$. Thus, $a^{h} = a^{mx+ny} = (a^m)^x (a^n)^y \equiv 1
\pmod d$. Therefore, $a^h-1 \equiv 0 \pmod d$, \emph{i.e.}, $d$ divides
$a^h-1$. Next, we observe that $a^m-1 = (a^h)^{m_0} - 1$. From which, we find 
that $a^h-1$ divides $a^m-1$. Similarly, $a^h-1$ divides $a^n-1$. Therefore, 
$a^h-1$ divides $\gcd(a^m-1,a^n-1) = d$. This proves that $d = a^h-1$. \\

We move to (ii). Let $A \defi a^h$ and $d \defi \gcd(a^m+1,a^n+1) =
\gcd(A^{m_0}+1,A^{n_0}+1)$. Then $A^{m_0} \equiv -1 \pmod d$ and $A^{n_0}
\equiv -1 \pmod d$. Therefore, $A^{2m_0} \equiv A^{2n_0} \equiv 1 \pmod d$.
Since $m_0$ and $n_0$ are coprime, by Bezout's Lemma, there exist 
$x,y \in \bZ$ such that
$m_0x+n_0y = 1$. From which, we have $A^2 = A^{2(m_0x+n_0y)} = 
(A^{2m_0})^x (A^{2n_0})^y \equiv 1 \pmod d$. Since $m_0$ is odd, we can write
$m_0 = 2m_1+1$ for some integer $m_1$. Hence, $-1 \equiv A^{m_0} \equiv
A^{2m_1+1} \equiv A^{2m_1} A \equiv A \pmod d$. Thus, $A+1 \equiv 0 \pmod d$,
\emph{i.e.}, $d$ divides $A+1$. Now, 
\begin{align*}
d ={}& \gcd(A^{m_0}+1,A^{n_0}+1) \\
={}& \gcd(A+1,A^{n_0}+1), \ \textrm{since $m_0$ is odd and so, 
$A+1$ divides $A^{m_0}+1$.} \\
={}& \gcd(A+1, A^{n_0}+1 \bmod {A+1}) \\
={}& \gcd(A+1, (-1)^{n_0}+1) \\
={}& \gcd(a^h+1, (-1)^{n_0}+1).
\end{align*}

Finally, we show part (iii). Let $d \defi \gcd(a^m-1,a^n+1)$. Then $a^m \equiv
1 \pmod d$ and $a^n \equiv -1 \pmod d$. So, $a^m \equiv a^{2n} \equiv 1 \pmod
d$. Since $m$ is odd, $\gcd(m,2n) = \gcd(m,n)$. By Bezout's Lemma, there 
exist $x,y \in \bZ$ such that $mx+2ny = \gcd(m,2n) = \gcd(m,n) = h$. Thus, 
$a^h = a^{mx+2ny} = a^{mx} a^{2ny} \equiv 1 \pmod d$. Hence, $a^n = a^{hn_0}
\equiv 1 \pmod d$. This means that $a^n-1 \equiv 0 \pmod d$, \emph{i.e.}, $d$
divides $a^n-1$. But we also know that $d$ divides $a^n+1$. Therefore, $d$
divides $a^n+1-(a^n-1) = 2$. Hence, $d = 1$ or $d = 2$. If $a$ is even then
both $a^m-1$ and $a^n+1$ are odd, and so their greatest common divisor is 
$1$. If $a$ is odd then both $a^m-1$ and $a^n+1$ are even and so, their
greatest common divisor is $2$. 
\end{proof}

\begin{lemma}[{\cite[Lemma 3.3]{Leemans}}]
Let $m$ be a positive integer and $q = 2^{2m+1}$. Then $q+\sqrt{2q}+1$ and
$q-\sqrt{2q}+1$ are coprime.
\end{lemma}
\begin{proof}
Let $d \defi \gcd(q+\sqrt{2q}+1,q-\sqrt{2q}+1)$. Then $d$ divides
$q+\sqrt{2q}+1-(q-\sqrt{2q}+1) = 2\sqrt{2q}$. However, $d$ is odd and
$2\sqrt{2q}$ is a power of $2$. Therefore, $d$ must be $1$, as required. 
\end{proof}

The following lemma summarises a list of coprime factors that will be used
later. 
\begin{lemma}[Coprime Factors]
\label{lem:coprime}
Let $m$ be a positive integer and $q = 2^{2m+1}$. Suppose that $r > 3$ is a
prime. Let \[
\xi_{q} \defi \begin{cases}
7 & \textrm{if $q \equiv 1 \pmod 7$}, \\
1 & \textrm{if $q \nequiv 1 \pmod 7$}.
\end{cases}
\]
Then the pairs of numbers $x_i(q)$ and $y_i(q)$ in \Cref{tab:coprimenumbers} 
are coprime. 
\begin{longtable}{lll}
\caption{Coprime numbers}
\label{tab:coprimenumbers}
\endfirsthead
\endhead
\toprule
\thead{$i$} & \thead{$x_i(q)$} & \thead{$y_i(q)$} \\
\midrule
$1$ & $(q^2+1)(q-1)$ & $(q^3+1)(q+1)$ \\
\midrule
$2$ & $(q^2+1)(q-1)$ & $q^2-q+1$ \\
\midrule
$3$ & $(q^6+1)(q^2+1)(q-1)$ & $(q^6-q^3+1)(q^4+q^2+1)$ \\
\midrule
$4$ & $(q-1)(q^2+1)(q^6+1)$ &
$(q^4-q^3+q^2-q+1)(q^4+q^3+q^2+q+1)$ \\
& & $\qquad \qquad (q^8+q^7-q^5-q^4-q^3+q+1)$ \\
\midrule
$5$ & $(q-1)(q^2+1)(q^6+1)/\xi_q$ & 
$(q^6-q^5+q^4-q^3+q^2-q+1)(q^6+q^5+q^4+q^3+q^2+q+1)$ \\
& & $\qquad \qquad (q^{12}+q^{11}-q^9-q^8+q^6-q^4-q^3+q+1)/\xi_q$ \\
\midrule
$6$ & $q+1$ & $q^2+1$ \\
\midrule 
$7$ & $(q+1)(q^2-q+1)$ & $(q^2+q+1)(q^4-q^2+1)$ \\
\midrule
$8$ & $q^2+1$ & $(q^2+q+1)(q^4+q^2+1)$ \\
\midrule
$9$ & $(q^r+1)/3$ & $(q^6+1)(q^2+1)(q-1)(q^2-q+1)/3$ \\
\midrule
$10$ & $q^2-q+1$ & $(q^6+1)(q^2+1)(q-1)$ \\
\midrule
$11$ & $q+1$ & $q^4-q^2+1$ \\
\midrule
$12$ & $q^2+1$ & $q^4-q^2+1$ \\
\bottomrule
\end{longtable}
\end{lemma}
\begin{proof}
We use \Cref{lem:amsans} and the extended Euclidean algorithm for 
polynomials.\footnote{If $f(x)$ and $g(x)$ are polynomials over a field 
$\field$ with greatest common 
divisor $d(x)$, then by Bezout's Lemma for polynomials, there exist 
polynomials $a(x)$ and $b(x)$ over $\field$
such that $a(x) f(x) + b(x) g(x) = d(x)$. The extended Euclidean algorithm
allows us to compute polynomials $a(x)$ and $b(x)$.} We use Mathematica to
perform the extended Euclidean algorithm, see \Cref{sec:eucalg}. 
\begin{enumerate}
\item Follows from \Cref{lem:amsans}. 
\item Follows from \Cref{lem:amsans} since $(q^2-q+1)(q+1) = q^3+1$. 
\item Note that $q^9+1 = (q^3+1)(q^6-q^3+1)$ and $q^9+1$ is coprime to
$(q^6+1)(q^2+1)(q-1)$ by \Cref{lem:amsans}. Next, we observe that $q^6-1 =
(q^2-1)(q^4+q^2+1) = (q^3-1)(q^3+1)$. By \Cref{lem:amsans}, $(q^6+1)(q^2+1)$
is coprime to $(q^3-1)(q^3+1)$ and hence, $(q^6+1)(q^2+1)$ is coprime to
$q^4+q^2+1$. Next, we prove that $q-1$ and $q^4+q^2+1$ are coprime. Suppose
$d = \gcd(q-1,q^4+q^2+1)$. Then $q \equiv 1 \pmod d$. Hence, $q^4+q^2+1 \equiv
3 \pmod d$. However, $d$ divides $q^4+q^2+1$, and so, $3 \equiv 0 \pmod d$,
\emph{i.e.}, $d$ divides $3$. But modulo $3$, we have 
$q-1 = 2^{2m+1}-1 \equiv (-1)^{2m+1}-1 \equiv -2 \nequiv 0 \pmod 3$. This
proves that $d$ must be $1$, whence, $\gcd(q-1,q^4+q^2+1) = 1$. 

\item First, note that since $q^6+1 = (q^2+1)(q^4-q^2+1)$, it suffices to
prove that $(q-1)(q^6+1)$ is coprime to $y_4(q)$. 
Now, observe that $q^5+1 = (q+1)(q^4-q^3+q^2-q+1)$. By \Cref{lem:amsans},
$q-1$ is coprime to $q^5+1$ and hence, $q-1$ is coprime to $q^4-q^3+q^2-q+1$.
From \Cref{lem:amsans}, we also have that $q^6+1$ is coprime to $q^5+1$. Thus,
$q^6+1$ is coprime to $q^4-q^3+q^2-q+1$. Therefore, $(q-1)(q^6+1)$ and
$q^4-q^3+q^2-q+1$ are coprime. Next, let $d = \gcd(q-1,q^4+q^3+q^2+q+1)$. We
show that $d = 1$. Observe that $q \equiv 1 \pmod d$. Thus, $q^4+q^3+q^2+q+1
\equiv 5 \pmod d$. Since $q^4+q^3+q^2+q+1 \equiv 0 \pmod d$, it follows that
$5 \equiv 0 \pmod d$, \emph{i.e.}, $d$ divides $5$. Hence, $d = 1$ or $d = 5$.
If $d = 5$, then $q-1 = 2^{2m+1}-1 = 4^m \cdot 2 - 1 \equiv (-1)^m \cdot 2 - 1
\nequiv 0 \pmod 5$. Hence, $d = 1$. Furthermore, we also have that $q^6+1$ and
$q^5-1$ are coprime by \Cref{lem:amsans}. Since $q^5-1 =
(q-1)(q^4+q^3+q^2+q+1)$, it follows that $q^6+1$ and $q^4+q^3+q^2+q+1$ are
coprime. Finally, we consider $q^8+q^7-q^5-q^4-q^3+q+1$. Let $d =
\gcd(q-1,q^8+q^7-q^5-q^4-q^3+q+1)$. Then $q \equiv 1 \pmod d$ and hence,
$q^8+q^7-q^5-q^4-q^3+q+1 \equiv 1 \pmod d$. Since $d$ divides
$q^8+q^7-q^5-q^4-q^3+q+1$, we find that $1 \equiv 0 \pmod d$. In other words,
$d = 1$. This shows that $\gcd(q-1,q^8+q^7-q^5-q^4-q^3+q+1) = 1$. Next, we 
use the extended Euclidean algorithm for polynomials (\Cref{lis:eucalg} 
line \ref{line:poly1}) to find:
\[
(-q^6+q^3)(q^6+1) + (q^4-q^3+q^2-q+1)(q^8+q^7-q^5-q^4-q^3+q+1) = 1.
\]
Hence, by Bezout's Lemma, $q^6+1$ and $q^8+q^7-q^5-q^4-q^3+q+1$ are coprime.
This proves that $x_4(q)$ and $y_4(q)$ are coprime. 

\item As in the previous case, because $q^6+1 = (q^2+1)(q^4-q^2+1)$, 
it suffices to prove that $(q-1)(q^6+1)/\xi_q$ is coprime to $y_5(q)$. 
Observe that 
$q^7+1 = (q+1)(q^6-q^5+q^4-q^3+q^2-q+1)$. Hence, by \Cref{lem:amsans}, it
follows that $(q-1)(q^6+1)$ and $q^7+1$ are coprime. Next, let $d \defi
\gcd(q-1,q^6+q^5+q^4+q^3+q^2+q+1)$. Then $q \equiv 1 \pmod d$ and so,
$q^6+q^5+q^4+q^3+q^2+q+1 \equiv 7 \pmod d$. This implies that $7 \equiv 0
\pmod d$, \emph{i.e.}, $d = 1$ or $d = 7$. If $d = 7$, then $q \equiv 1 \pmod
7$ and so, $\xi_q = 7$. Consequently, dividing by $7$, we obtain 
$\gcd((q-1)/\xi_q,(q^6+q^5+q^4+q^3+q^2+q+1)/\xi_q) = 1$. If $d = 1$, then $q
\nequiv 1 \pmod 7$ since $q \equiv 1 \pmod 7$ implies that $q-1 \equiv 0 \pmod
7$ and $q^6+q^5+q^4+q^3+q^2+q+1 \equiv 0 \pmod 7$, whence, $7$ divides $d$.
Therefore, in this case, we have $\xi_q = 1$ and
$\gcd((q-1)/\xi_q,(q^6+q^5+q^4+q^3+q^2+q+1)/\xi_q) = 1$. Also, by
\Cref{lem:amsans}, it follows that $q^6+1$ and $q^7-1$ are coprime. Therefore,
$(q-1)(q^6+1)/\xi_q$ and $(q^6+q^5+q^4+q^3+q^2+q+1)/\xi_q$ are coprime.
Next, we let $d \defi \gcd(q-1,q^{12}+q^{11}-q^9-q^8+q^6-q^4-q^3+q+1)$. Then
$q \equiv 1 \pmod d$ and so, $q^{12}+q^{11}-q^9-q^8+q^6-q^4-q^3+q+1 \equiv 1
\pmod d$. Therefore, $d = 1$. Finally, we use the extended Euclidean algorithm
for polynomials (\Cref{lis:eucalg} line \ref{line:poly2}) 
to find:
\[
(-q^{11}+q^9-q^6+q^4)(q^6+1) +
(q^5-q^4+q^2-q+1)(q^{12}+q^{11}-q^9-q^8+q^6-q^4-q^3+q+1) = 1.
\] 
Therefore, by Bezout's Lemma, $q^6+1$ and $q^{12}+q^{11}-q^9-q^8+q^6-q^4-q^3+
q+1$ are coprime. 

\item Follows from \Cref{lem:amsans}.

\item Observe that $q^3+1 = (q+1)(q^2-q+1)$. Also, 
$q^3-1 = (q-1)(q^2+q+1)$ and
$q^6+1 = (q^2+1)(q^4-q^2+1)$. Hence, it suffices to show that $q^3+1$ and
$(q^3-1)(q^6+1)$ are coprime. In fact, this follows precisely from
\Cref{lem:amsans}. 

\item Observe that $q^3-1 = (q-1)(q^2+q+1)$ and $q^6-1 = (q^2-1)(q^4+q^2+1) =
(q^3+1)(q^3-1)$. By \Cref{lem:amsans}, we have that $q^2+1$ and 
$(q^3-1)(q^3+1)$ are coprime. Thus, $q^2+1$ and $(q^2+q+1)(q^4+q^2+1)$ are 
coprime. 

\item Observe that $q^6+1 = (q^2+1)(q^4-q^2+1)$. Therefore, it suffices to
show that $x_8(q)$ is coprime to $(q^6+1)(q-1)(q^2-q+1)/3$. To this end, we
use \Cref{lem:amsans} and see that $q^r+1$ and $(q^6+1)(q-1)$ are coprime. 
Next, observe that $q^3+1 = (q+1)(q^2-q+1)$. Define $d \defi
\gcd(q^r+1,q^2-q+1)$. By \Cref{lem:amsans}, we have
that $k \defi \gcd(q^r+1,q^3+1) = \gcd(q+1,0) = q+1$. Next, we let $\ell \defi
\gcd(q+1,q^2-q+1)$. Note that $3$ divides $\ell$ since $q+1 \equiv (-1)+1
\equiv 0 \pmod 3$ and $q^2-q+1 \equiv (-1)^2-(-1)+1 \equiv 3 \equiv 0 \pmod
3$. We show that $\ell = 3$. Note that $q \equiv -1 \pmod \ell$ and so, 
$q^2-q+1 \equiv (-1)^2-(-1)+1 \equiv 3 \pmod \ell$. So, $3 \equiv 0 \pmod 
\ell$. In 
other words, $\ell$ divides $3$. Now, we know that $d$ divides $k$ and so, $d$
divides $q+1$ and $q^2-q+1$. Therefore, $d$ divides $\ell = 3$. 
Hence, $d = 1$ or $d = 3$. However, $3$ divides $q^r+1$ and $q^2-q+1$. 
Therefore, $d = 3$. 
Dividing by $3$, we obtain $\gcd((q^r+1)/3,(q^2-q+1)/3) = 1$. 

\item Observe that $q^3+1 = (q+1)(q^2-q+1)$ and $q^6+1 = (q^2+1)(q^4-q^2+1)$
Therefore, it suffices to show that $q^3+1$ and $(q^6+1)(q-1)$ are coprime.
Indeed, this follows from \Cref{lem:amsans}.  

\item Observe that $q^6+1 = (q^2+1)(q^4-q^2+1)$. From \Cref{lem:amsans}, we
have that $q+1$ and $q^6+1$ are coprime. Therefore, $q+1$ and $q^4-q^2+1$ are
coprime. 

\item Suppose that $d = \gcd(q^2+1,q^4-q^2+1)$. Then $q^2 \equiv -1 \pmod d$.
Thus, $q^4-q^2+1 \equiv (-1)^2-(-1)+1 \equiv 1+1+1 \equiv 3 \pmod d$. This
means that $d = 1$ or $d = 3$. But we see that $q^2+1 \equiv (-1)^2+1 \equiv 2
\pmod 3$. Therefore, $d = 1$.  \qedhere
\end{enumerate}
\end{proof}

The following lemma provides upper bounds on the greatest common divisor of
various quantities that will be used later to bound $\gcd(s,t)$. 
\begin{lemma}[Bound on Greatest Common Divisor]
\label{lem:boundgcd}
Let $m$ be a positive integer and $q = 2^{2m+1}$. Define the following three
integers 
\[
F_1(q) \defi \frac{1}{2} q^{24} (q^{12}-q^6+1)(q^9+1)(q^3+1)-1, \quad 
F_2(q) \defi \frac{1}{2} q^{24}(q^{18}+1)(q^6-q^3+1)-1 
\]
and 
\[
G(q) \defi q^{24} (q^2-q+1)(q^2+q+1)^2 (q^4-q^2+1) 
(q^6-q^3+1)(q^{12}-q^6+1) -1.
\]
Then we have 
$\gcd(F_i(q),G(q)) \leqslant q$ for $i \in \{1,2\}$.
\end{lemma}
\begin{proof}
We use Mathematica to perform the extended Euclidean algorithm, see 
\Cref{sec:eucalg} \Cref{lis:eucalgbound}. For $i \in \{1,2\}$, we find 
polynomials $u_i(q)$ and $v_i(q)$ with rational coefficients such that 
\[
1 = u_i(q) F_i(q) + v_i(q) G(q).
\]
Clearing denominators (\Cref{lis:eucalgbound} lines \ref{line:F1G1} and
\ref{line:F2G1}), we obtain polynomials $U_i(q)$ and $V_i(q)$ with integer
coefficients such that 
\[
b_1 \defi 4,514,863,691,969,464,490,886,472,694,534 = 
U_1(q) F_1(q) + V_1(q) G(q)
\]
and 
\[
b_2 \defi 24,150,991,970,664,810,056,096 = U_2(q) F_2(q) + V_2(q) G(q).
\]
By Bezout's Lemma, it follows that $\gcd(F_i(q),G(q))$ divides $b_i$. Since
$F_i(q)$ and $G(q)$ are odd, we have that $\gcd(F_1(q), G(q))$ divides $b_1/2$
and $\gcd(F_2(q), G(q))$ divides $b_2/32$, where $2$ and $32$ are the highest
powers of $2$ dividing $b_1$ and $b_2$, respectively (\Cref{lis:eucalgbound}
lines \ref{line:b1} and \ref{line:b2}). Note that for every $q
\geqslant 2^{101}$, we have 
\[
\frac{b_1}{2} = 2,257,431,845,984,732,245,443,236,347,267 \leqslant q
\]
and 
\[
\frac{b_2}{32} = 754,718,499,083,275,314,253 \leqslant q
\]
Thus, for $q \geqslant 2^{101}$, we have that
$\gcd(F_i(q),G(q)) \leqslant q$. Now, for $q \leqslant 2^{99}$, we compute the
$\gcd(F_i(q),G(q))$ numerically (\Cref{lis:eucalgbound} lines 
\ref{line:F1Gfinite}--\ref{line:F1Gfiniteend}) and verify that 
$\gcd(F_i(q),G(q)) \leqslant q$. 
Thus, $\gcd(F_i(q),G(q)) \leqslant q = 2^{2m+1}$ for every positive 
integer $m$. 
\end{proof}

\subsection[Large Ree Groups Acting on GQs]{Large Ree Groups Acting on 
Generalised Quadrangles}
\label{sec:largereeacts}
In this subsection, we prove the following theorem. 
\begin{theorem}
\label{thm:largeree}
Let $q = 2^{2m+1}$ for some positive integer $m$. Then an almost simple group
with socle isomorphic to $\LargeRee{q}$ cannot act primitively on both the set 
of points and the set of lines of a generalised quadrangle.
\end{theorem}
We shall now prove this theorem via the following lemmata.  
\begin{lemma}[Reduction to Simple Group]
\label{cor:reductionf4}
Let $G$ be an almost simple group with socle isomorphic to
$\LargeRee{q}$ and let $\Gamma = (\cP,\cL,\cI)$ be a generalised quadrangle of
order $(s,t)$ where $s,t \geqslant 2$. Suppose $G$ acts primitively on $\cP$
and $\cL$. Then $\soc{G}$ acts primitively on $\cP$ and $\cL$. 
\end{lemma}
\begin{proof}
First, note that the outer automorphism group of $\LargeRee{q}$ is cyclic 
\cite[Theorem 7.1]{reef4}. Second, 
$G$ has no novelty maximal subgroups \cite[Corollary]{mallef4}. Thus, by 
\Cref{cor:reduction}, $\soc{G}$ acts primitively on both $\cP$ and $\cL$. 
\end{proof}
In view of \Cref{cor:reductionf4}, we take $G = \LargeRee{q}$ and consider $G$
acting primitively on the set of points and the set of lines of a generalised
quadrangle $\Gamma = (\cP,\cL,\cI)$ with order $(s,t)$ where $s,t \geqslant
2$. We take a point $\alpha \in \cP$ and a line in $L \in \cL$. Their 
respective stabilisers, $G_{\alpha}$ and $G_L$ are maximal subgroups of $G$. 
Also, note that the dual of $\Gamma$ is a generalised quadrangle 
of order $(t,s)$. Thus, we only need to consider the cases for $G_{\alpha}$ 
and $G_L$ up to duality. In other words, once we have shown that $G_{\alpha}$
cannot be isomorphic to a maximal subgroup, say $H$, we do not need to
consider the case where $G_L \cong H$. With this in mind, we work through the 
list of maximal subgroups from \Cref{thm:largereemaxs}.
\begin{lemma}
Let $q = 8$. 
The subgroup $G_{\alpha}$ (and dually $G_L$) cannot be isomorphic to 
$\PGLn{2}{13}$.
\end{lemma}
\begin{proof}
Suppose for a contradiction that $G_{\alpha} \cong \PGLn{2}{13}$. Then 
$|\cP| = (s+1)(st+1) = |G:G_{\alpha}|$. It can be verified via Mathematica
(see \Cref{sec:finite} \Cref{lis:finite} lines 
\ref{line:pgl231}--\ref{line:pgl232}) that there are no solutions for $s$ 
and $t$. Therefore, $G_{\alpha} \ncong \PGLn{2}{13}$. 
\end{proof}

\begin{lemma}
The subgroup $G_{\alpha}$ (and dually $G_L$) cannot be isomorphic to 
$[q^{10}] : (\Sz{q} \times C_{q-1})$.
\end{lemma}
\begin{proof}
Suppose for a contradiction that $G_{\alpha} \cong 
[q^{10}] : (\Sz{q} \times C_{q-1})$. 
Consider the action of $G_{\alpha}$ on the generalised quadrangle. 
Note that there are $s(t+1)$ points collinear to $\alpha$. Therefore, 
by \Cref{cor:numcol}, there 
are $|\cP| - s(t+1) -1 = s^2 t$ points that are 
not collinear to $\alpha$. Since the action of $G_{\alpha}$ preserves the 
incidence relation, the set of points collinear to $\alpha$ must be a union of
$G_{\alpha}$-orbits. By \Cref{lem:largereesubdegrees}, we can write 
$a \defi s(t+1) = a_1 n_1 + a_2 n_2 + a_3 n_3 + a_4 n_4$, where $a_i \in
\{0,1\}$ and the values of $n_i$ are given by: $n_1 = q(q^2+1)$, 
$n_2 = q^4(q^2+1)$, $n_3 = q^7(q^2+1)$ and $n_4 = q^{10}$. Similarly,
the set of points not collinear to $\alpha$ must also be a union of
$G_{\alpha}$-orbits: $b \defi s^2 t = b_1 n_1 + b_2 n_2 + b_3 n_3 + b_4 n_4$,
where $b_i = 1-a_i$. \\ 

We prove that $a_4 = 0$ and $a_3 = 0$. Indeed, if $a_4 = 1$, then $s(t+1)
\geqslant n_4 \geqslant n_1 + n_2 + n_3 \geqslant s^2 t$, which contradicts 
\Cref{lem:s2t} \ref{lem:s2ti}. 
Now, suppose $a_3 = 1$. Then $s(t+1) \geqslant n_3 =
q^7(q^2+1) \geqslant q^9$. Hence, by \Cref{lem:s2t} \ref{lem:s2tii},  
\[
s^2t \geqslant \paren{\frac{2}{3}}^{4/3} \paren{s(t+1)}^{4/3} \geqslant
\paren{\frac{2}{3}}^{4/3} q^{9 \cdot 4/3} = \paren{\frac{2}{3}}^{4/3} q^{12} =
\paren{\frac{2}{3}}^{4/3} q q^{11} \geqslant q^{11},
\]
since $q = 2^{2m+1} \geqslant 8 > 9/4 = (3/2)^2 > (3/2)^{4/3}$. However, 
$s^2 t \leqslant n_1 + n_2 + n_4 \leqslant 2n_4 = 2q^{10} < q^{11}$, which is
a contradiction. Therefore, $s(t+1) = a_1 n_1 + a_2 n_2$. We prove that
$s(t+1) \neq n_1$. For a contradiction, suppose $s(t+1) = n_1$. Then $s(t+1) =
q(q^2 + 1) \leqslant q(2q^2) \leqslant q^4$. Hence, by \Cref{lem:s2t} 
\ref{lem:s2tii}, 
\[
s^2 t \leqslant \paren{s(t+1)}^{5/3} \leqslant q^{4 \cdot 5/3} = q^{20/3} <
q^7,
\] 
which is a contradiction because $s^2 t \geqslant q^{10}$. \\ 

Next, we focus on 
eliminating the remaining two cases: $s(t+1) = n_2$ and $s(t+1) = n_1 +
n_2$. To this end, we write $a \defi s(t+1)$ and $b \defi s^2t$ as before and
find:  
\[
\frac{s^2(t+1)^2}{s^2 t} = \frac{a^2}{b}.
\]
Hence, $b(t+1)^2 = a^2t$. Expanding and simplifying, we obtain
$bt^2+(2b-a^2)t+b = 0$. Since $t$ is an integer zero to the quadratic:
$bx^2 + (2b-a^2)x+b$, its discriminant is a
square. Computing the discriminant, 
\begin{equation}
\Delta = (2b-a^2)^2 - 4b^2 = 4b^2-4a^2b+a^4-4b^2 = a^2(a^2-4b),
\label{eqn:disc} 
\end{equation}
we deduce that $a^2-4b$ is a square. Now, for the sake of obtaining a
contradiction, suppose $a = s(t+1) = n_2$. Then $b = n_1+n_3+n_4$. Thus, 
\begin{align*}
a^2-4b ={}& n_2^2 - 4(n_1+n_3+n_4) \\
={}& (q^4(q^2+1))^2 - 4(q(q^2+1)+q^7(q^2+1)+q^{10}) \\
={}& q^8(q^4+2q^2+1) - 4(q^3+q+q^9+q^7+q^{10}) \\
={}& 4q\paren{\frac{q^{11}}{4} + \frac{q^9}{2} + \frac{q^7}{4} - q^2 - 1 - q^8
- q^6 - q^9}.
\end{align*}
From which, we see that the $2$-part\footnote{Recall that the 
$2$-part of a number is the highest power of $2$ that divides that 
number.} of $a^2-4b$ is
$4q$, which is $2^2 2^{2m+1} = 2^{2m+3}$. This is an odd power of $2$ and so,
$a^2-4b$ cannot be a square, a contradiction. Thus, $s(t+1) = n_1+n_2$.
Again, we suppose $a = s(t+1) = n_1+n_2$ and so, $b = n_3+n_4$. Computing
$a^2-4b$, 
\begin{align*}
a^2-4b ={}& (n_1+n_2)^2 - 4(n_3+n_4)^2 \\
={}& (q(q^2+1)+q^4(q^2+1))^2 - 4(q^7(q^2+1)+q^{10}) \\
={}& q^2 (q^5-q^3-q^2-1)^2.
\end{align*}
Since $bt^2+(2b-a^2)t+b = 0$ and $t \geqslant 2$, we use the quadratic 
formula and \Cref{eqn:disc} to find that 
\[
t = \frac{a^2-2b + a\sqrt{a^2-4b}}{2b}. 
\]
After some algebraic manipulation (see \Cref{sec:algman}), we obtain the 
following expression for $t$: 
\[
t = \frac{q^5}{1+q^2+q^3},
\]
which is not an integer since $q$ is a power of $2$, a contradiction. 
Thus, we have eliminated all the possibilities for $a$ and
$b$. Hence, $G_{\alpha} \ncong [q^{10}] : (\Sz{q} \times C_{q-1})$ (and
dually, $G_L \ncong [q^{10}] : (\Sz{q} \times C_{q-1})$).
\end{proof}

\begin{lemma}
The subgroup $G_{\alpha}$ (and dually $G_L$) cannot be isomorphic to 
$[q^{11}] : \GLn{2}{q}$.
\end{lemma}
\begin{proof}
Suppose for a contradiction, that $G_{\alpha} \cong 
[q^{11}] : \GLn{2}{q}$. 
By the same argument as before, using 
\Cref{lem:largereesubdegrees}, we can write $a \defi s(t+1) = a_1 n_1 + a_2
n_2 + a_3 n_3 + a_4 n_4$ and $b \defi s^2 t = b_1 n_1 + b_2 n_2 + b_3 n_3 +
b_4 n_4$, where $a_i \in \{0,1\}$ and $b_i = 1-a_i$. The only difference here
is that the values of $n_i$ are given by: $n_1 = q^2(q+1)$, $n_2 = q^5(q+1)$, 
$n_3 = q^8(q+1)$ and $n_4 = q^{11}$. We proceed to eliminate this case in a
similar fashion to the previous case. First, we prove that $a_4 = 0$ and $a_3
= 0$. Indeed, if $a_4 = 1$, then $s(t+1) \geqslant n_4 \geqslant n_1 + n_2 +
n_3 \geqslant s^2 t$, which contradicts \Cref{lem:s2t} \ref{lem:s2ti}. 
Hence, $a_4 = 0$. 
Now, if $a_3 = 1$, then $s(t+1) \geqslant n_3 = q^8(q+1) \geqslant q^9$.
Consequently, using \Cref{lem:s2t} \ref{lem:s2tii},
\[
s^2 \geqslant \paren{\frac{2}{3}}^{4/3} (s(t+1)^{4/3}) \geqslant
\paren{\frac{2}{3}}^{4/3} q^{9 \cdot 4/3} = \paren{\frac{2}{3}}^{4/3} 
q^{12} \geqslant \paren{\paren{\frac{2}{3}}^{4/3} q} q^{11} \geqslant 4q^{11},
\] 
since $q \geqslant 8$. However, $s^2 t \leqslant n_1 + n_2 + n_3 + n_4
\leqslant 2n_4 = 2q^{11}$, which is a contradiction. Hence, $a_3 = 0$.
Therefore, $s(t+1) = a_1 n_1 + a_2 n_2$. We prove that $s(t+1) > n_1 + n_2$.
To this end, suppose $s(t+1) \leqslant n_1+n_2$. Then
\[
s(t+1) \leqslant n_1+n_2 = q^2(q+1) + q^5(q+1) = q^3 + q^2 + q^6 + q^5
\leqslant q^6 + 3q^5 \leqslant 2q^6,
\]
as $q \geqslant 8$. 
Hence, by \Cref{lem:s2t} \ref{lem:s2tii}, 
\[
s^2 t \leqslant (s(t+1))^{5/3} \leqslant (2q^6)^{5/3} = 2^{5/3} q^{6 \cdot
5/3} = 2^{5/3} q^{10} < q^{11} = n_4.
\]
However, $s^2t \geqslant n_4$ and we have a contradiction. Now, we
have proven that $s(t+1) = a_1 n_1 + a_2 n_2$ and $s(t+1) > n_1+n_2$, both of
which cannot hold true. Therefore, $G_{\alpha} \ncong [q^{11}] : \GLn{2}{q}$
(and dually, $G_L \ncong [q^{11}] : \GLn{2}{q}$).
\end{proof}

\begin{lemma}
The subgroups $G_{\alpha}$ and $G_L$ have different orders. In other words,
the number of points and the number of lines are different. 
\end{lemma}
\begin{proof}
Suppose that $|G_{\alpha}| = |G_L|$ for a contradiction. Then 
$|\cP| = |\cL|$ and so, $s = t$,
thereby, $|\cP| = (s+1)(st+1) = (s+1)(s^2+1)$. Let $2^a$ be the highest power
of $2$ dividing $|\cP|$. Then by \Cref{lem:bound} \ref{item:s}, we find that
$|G:G_{\alpha}| = |\cP| > 2^{a-1} (2^{2a-4}+1) > 2^{3a-5}$. On the other hand,
we have an upper bound from \Cref{tab:bounds} of the form: 
$|G:G_{\alpha}| \leqslant f(q)$. Thus, $2^{3a-5} < f(q)$. We go through each
of the remaining cases for $G_{\alpha}$ and compute $2^{3a-5}$ and $f(q)$. The
results are in \Cref{tab:sist}.
\begin{longtable}{llll}
\caption{Lower and upper bounds for $|G:H|$ in the case $s = t$.}
\label{tab:sist}
\endfirsthead
\endhead
\toprule
\thead{$G_{\alpha}$} & \thead{$2^a$} & \thead{$2^{3a-5}$} & 
\thead{$f(q)$} \\ 
\midrule
\makecell{$\SUn{3}{q}:C_2$ \\ or $\PGUn{3}{q}:C_2$} & 
$\ds \frac{1}{2} q^9 = 2^{18m+8}$  
& $2^{54m+19}$  & $\ds \frac{q^{18}}{2} = 2^{36m+17}$ \\
\midrule
$\Sz{q} \wr C_2$ & $\ds \frac{1}{2} q^8 = 2^{16m+7}$  & $2^{48m+16}$ & 
$q^{16} = 2^{32m+16}$ \\
\midrule
$\SPn{4}{q} : C_2$ & $\ds \frac{1}{2} q^8 = 2^{16m+7}$ & $2^{48m+16}$ & $\ds
\frac{q^{16}}{2} = 2^{32m+15}$ \\
\midrule
$\LargeRee{q_0}$ & $\ds \paren{\frac{q}{q_0}}^{12} = 2^{12n_0(r_0-1)}$ 
& $2^{36n_0(r_0-1)-5}$ & $\ds 2\paren{\frac{q}{q_0}}^{26} = 
2^{26n_0(r_0-1)+1}$ \\
\midrule
$C_{q+1}^2:\GLn{2}{3}$ & $\ds \frac{1}{16}
q^{12} = 2^{24m+8}$ & $2^{72m+19}$  & 
$\ds \frac{q^{24}}{48} \leqslant \frac{q^{24}}{32} = 2^{48m+19}$ \\
\midrule
$C_{q+\epsilon\sqrt{2q}+1}^2 : 4S_4$ & $\ds \frac{1}{32} q^{12} = 2^{24m+7}$ 
& $2^{72m+16}$ & $\ds \frac{q^{24}}{48} \leqslant \frac{q^{24}}{32} = 
2^{48m+19}$ \\
\midrule
$C_{q^2+q+1+\epsilon\sqrt{2q}(q+1)} : C_{12}$ & $\ds \frac{1}{4} q^{12} =
2^{24m+10}$ & $2^{72m+25}$ & $\ds \frac{q^{24}}{6} \leqslant \frac{q^{24}}{4}
= 2^{48m+22}$ \\
\bottomrule
\end{longtable}
In each of these cases, we find that $2^{3a-5} > f(q)$, which contradicts
$2^{3a-5} < |\cP| \leqslant f(q)$. Therefore, $|G_{\alpha}| \neq |G_L|$. 
\end{proof}

We now consider the remaining cases where $|G_{\alpha}| \neq |G_L|$. \\

In Lemmata \ref{lem:su3first}--\ref{lem:su3last} we prove, 
by contradiction, that $G_{\alpha}$ cannot be 
isomorphic to $\SUn{3}{q} : C_2$ or $\PGUn{3}{q} : C_2$. We show this by
assuming $G_{\alpha}$ is isomorphic to $\SUn{3}{q} : C_2$ or $\PGUn{3}{q} :
C_2$ and study the various possibilities for $G_L$. 
\begin{lemma}
\label{lem:su3first}
Suppose $G_{\alpha} \cong \SUn{3}{q} : C_2$ or $\PGUn{3}{q} : C_2$. Then $G_L$
cannot be isomorphic to $\Sz{q} \wr C_2$. 
\end{lemma}
\begin{proof}
Suppose for a contradiction, $G_L \cong \Sz{q} \wr C_2$. Then 
we have 
\[
|\cP| = (s+1)(st+1) = \frac{1}{2} q^9 (q^6+1)(q^2+1)(q-1) \ \textrm{and} \
|\cL| = (t+1)(st+1) = \frac{1}{2} q^8 (q^4-q^2+1)(q^3+1)(q+1).
\]
Hence,
\[
\frac{|\cP|}{|\cL|} = \frac{s+1}{t+1} = \frac{q(q^2+1)^2(q-1)}{(q^3+1)(q+1)},
\]
since $q^6+1 = (q^2+1)(q^4-q^2+1)$. Cross multiplying, we obtain 
\[
(s+1)(q^3+1)(q+1) = (t+1)q(q^2+1)^2(q-1).
\]
By \Cref{lem:coprime} ($i=1$), 
we deduce that $q(q^2+1)^2(q-1)$ divides $s+1$ and
$(q^3+1)(q+1)$ divides $t+1$. Hence, there exists an integer $k$ such that 
\[
s+1 = kq(q^2+1)^2(q-1) \ \textrm{and} \ t+1 = k(q^3+1)(q+1).
\]
First, notice that $k$ is even. For if $k$ were to be odd, then $t$ is even.
Thus, $|\cL| = (t+1)(st+1)$ is odd. However, this is not the case as $q$
divides $|\cL|$. Thus, $k \equiv 0$ or $2 \pmod 4$. Suppose first that 
$k \equiv 
0 \pmod 4$. Then $s \equiv -1 \pmod 4$ and $t \equiv -1 \pmod 4$. Whence, 
$st+1 \equiv 2 \pmod 4$. Thus, the $2$-part of $st+1$ is $2$ 
and consequently, the $2$-part of $s+1$ is $q^9/4$. 
Since $s+1 = kq(q^2+1)^2(q-1)$, the $2$-part of $k$ is $q^8/4$. Therefore,
\[
s+1 = kq(q^2+1)^2(q-1) = kq(q^5-q^4+2q^3-2q^2+q-1) > k \cdot q \cdot 
\frac{q^5}{2} \geqslant
\frac{q^8}{4} \cdot q \cdot \frac{q^5}{2} = \frac{q^{14}}{8} 
\]
and 
\[
t+1 = k(q^3+1)(q+1) = k(q^4+q^3+q+1) > kq^4 \geqslant \frac{q^8}{4} \cdot 
q^4 = \frac{q^{12}}{4}. 
\]
Applying \Cref{lem:modhigman} \ref{item:modhigman1}, we find that 
\[
|\cP| > \frac{(s+1)^2(t+1)}{2} > \frac{1}{2} \paren{\frac{q^{14}}{8}}^2
\frac{q^{12}}{4} = \frac{q^{40}}{2^9}.
\]
However, this contradicts the upper bound from \Cref{tab:bounds}, namely that 
$|\cP| \leqslant q^{18}/2$. 
So, in fact, $k \equiv 2 \pmod 4$. In this case, the
$2$-part of $s+1$ is $2q$ and so, the $2$-part of $st+1$ must be $q^8/4$. Note
that modulo $q$, we have $s \equiv -1 \pmod q$ and $t \equiv k-1 \pmod q$. 
From this, we see that  
\[
st+1 \equiv -1 \cdot (k-1) + 1 \equiv -k+2 \equiv 0 \pmod q.
\]
Hence, $k = 2+aq$ for some non-negative integer $a$. Now we have the following
two cases: $a = 0$ or $a > 0$. Suppose $a = 0$. Then $k = 2$ and so, 
\[
s+1 = 2q(q^2+1)^2(q+1) = 2(q^6-q^5+2q^4-2q^3+q^2-q) < 2q^6
\]
and 
\[
t+1 = 2(q^3+1)(q+1) = 2(q^4+q^3+q+1) < 4q^4.
\]
Thus, 
\[
|\cP| = (s+1)(st+1) < (s+1)(s(t+1)) < 2q^6 \cdot 2q^6 \cdot 4q^4 = 16q^{16}.
\]
However, from \Cref{tab:bounds}, we see that  
\[
\frac{q^{18}}{4} \leqslant |\cP| < 16q^{16}.
\]
This implies that $q^2 < 32$, which contradicts the assumption that $q
\geqslant 8$. Therefore, $a \neq 0$. So suppose that $a > 0$. In this
scenario, we have $k = 2+aq > q+1$. Therefore, 
\[
s+1 = kq(q^2+1)^2(q-1) > (q+1)q(q^2+1)^2(q-1) = q^7+q^5-q^3-q > q^7
\] 
and  
\[
t+1 = k(q^3+1)(q+1) > (q+1)(q^3+1)(q+1) = q^5+2q^4+q^3+q^2+2q+1 >  q^5.
\]

Applying \Cref{lem:modhigman} \ref{item:modhigman1},  
\[
|\cP| > \frac{(s+1)^2(t+1)}{2} > \frac{q^{14} q^5}{2} = \frac{q^{19}}{2}. 
\]
However, from \Cref{tab:bounds}, we have 
\[
\frac{q^{19}}{2} < |\cP| \leqslant \frac{q^{18}}{2},
\]
which is a contradiction. Thus, $G_L \ncong \Sz{q} \wr C_2$. 
\end{proof}

\begin{lemma}
Suppose $G_{\alpha} \cong \SUn{3}{q} : C_2$ or $\PGUn{3}{q} : C_2$. Then $G_L$
cannot be isomorphic to $\SPn{4}{q} : C_2$. 
\end{lemma}
\begin{proof}
Suppose for a contradiction, $G_L \cong \SPn{4}{q} : C_2$. Then 
we have 
\[
|\cP| = (s+1)(st+1) = \frac{1}{2}q^9(q^6+1)(q^2+1)(q-1) \ \textrm{and} \ |\cL|
= (t+1)(st+1) = \frac{1}{2} q^8(q^6+1)(q^2-q+1).
\]
Hence, 
\[
\frac{|\cP|}{|\cL|} = \frac{s+1}{t+1} = \frac{q(q^2+1)(q-1)}{q^2-q+1}. 
\]
Cross multiplying yields
\[
(s+1)(q^2-q+1) = (t+1)q(q^2+1)(q-1).
\]
By \Cref{lem:coprime} ($i=2$), $q^2-q+1$ is coprime to
$q(q^2+1)(q-1)$. Thus, there exists an integer $k$ such that 
\[
s+1 = kq(q^2+1)(q-1) \ \textrm{and} \ t+1 = k(q^2-q+1).
\]
As before, we deduce that $k$ is even. For if $k$ were to be odd, then $t$ is
even, which means that $|\cL| = (t+1)(st+1)$ is odd. This is not the case as
$q$ divides $|\cL|$. Therefore, $k \equiv 0$ or $2 \pmod 4$. Suppose first
that $k \equiv 0
\pmod 4$. Then $s \equiv -1 \pmod 4$ and $t \equiv -1 \pmod 4$. Whence, $st+1
\equiv 2 \pmod 4$. Thus, the $2$-part of $st+1$ is $2$ and consequently, the
$2$-part of $s+1$ is $q^9/4$. Since $s+1 = kq(q^2-q+1)$, the $2$-part of $k$
is $q^8/4$. Therefore, 
\[
s+1 = kq(q^2+1)(q-1) = kq(q^3-q^2+q-1) > \frac{q^8}{4} \cdot q \cdot
\frac{q^3}{2} = \frac{q^{12}}{8} 
\]
and 
\[
t+1 = k(q^2-q+1) > \frac{q^8}{4} \cdot \frac{q^2}{2} = \frac{q^{10}}{8}. 
\]
Applying \Cref{lem:modhigman} \ref{item:modhigman1}, we find that 
\[
|\cP| > \frac{(s+1)^2(t+1)}{2} > \frac{1}{2} \paren{\frac{q^{12}}{8}}^2
\frac{q^{10}}{8} = \frac{q^{34}}{2^{10}}.
\]
However, this contradicts the upper bound from \Cref{tab:bounds}, namely that 
$|\cP|
\leqslant q^{18}/2$. So, in fact, $k \equiv 2 \pmod 4$. In this case, the
$2$-part of $s+1$ is $2q$ and so, the $2$-part of $st+1$ must be $q^8/4$. 
Now, we write $k$ in $q$-ary, \emph{i.e.}, $k = a_0 + a_1 q + \cdots + 
a_n q^n$ where $a_0,a_1,\dots,a_n$ are integers such that 
$0 \leqslant a_0,a_1,\dots,a_n < q$. 
Next, we compute $st+1$ modulo powers of $q$. First, $s$ and $t$ are given by:
\[
s = (a_0+a_1q+\cdots+a_nq^n)(q^4-q^3+q^2-q) - 1
\]
and 
\[
t = (a_0+a_1q+\cdots+a_nq^n)(q^2-q+1) - 1.
\]
Plugging these expressions into $st+1$, we find that $st+1 \equiv 2-a_0 \pmod
q$. Hence, $a_0 = 2$. Next, we find that $st+1 \equiv -a_1 q 
\equiv 0 \pmod {q^2}$. Hence, $a_1 = 0$. Next, $st+1 \equiv (4-a_2)q^2 \equiv 
0 \pmod {q^3}$. Hence, $a_2 = 4$. Finally, $st+1 \equiv (-18-a_3)q^3 \equiv 0
\pmod {q^4}$. Thus, 
\[
a_3 = \begin{cases} 
q-18 & \textrm{if $q \geqslant 32$}, \\
q-2 & \textrm{if $q = 8$}. 
\end{cases}
\]
Therefore, \[
k = a_0 + a_1 q + \cdots + a_n q^n = \begin{cases}
2+4q^2+(q-18)q^3  & \textrm{if $q \geqslant 32$}, \\
3330 & \textrm{if $q = 8$}. 
\end{cases}
\]
In both cases, we have $k \geqslant q^4/4$. Hence, 
\[
s+1 = kq(q^2+1)(q-1) = kq(q^3-q^2+q-1) > \frac{q^4}{4} \cdot q \cdot
\frac{q^3}{2} = \frac{q^{8}}{8}
\]
and 
\[
t+1 = k(q^2-q+1) > \frac{q^4}{4} \cdot \frac{q^2}{2} = \frac{q^6}{8}.
\]
Applying \Cref{lem:modhigman} \ref{item:modhigman1},  
\[
|\cP| > \frac{(s+1)^2(t+1)}{2} > \frac{1}{2} \paren{\frac{q^8}{8}}^2
\frac{q^6}{8} = \frac{q^{22}}{2^9}.
\]
However, from \Cref{tab:bounds}, we have 
\[
\frac{q^{22}}{2^9} < |\cP| \leqslant \frac{q^{18}}{2},
\]
from which, we find that $q < 4$ and this is a contradiction since 
$q \geqslant 8$. Thus, $G_L \ncong \SPn{4}{q} : C_2$.  
\end{proof}

\begin{lemma}
Suppose $G_{\alpha} \cong \SUn{3}{q} : C_2$ or $\PGUn{3}{q} : C_2$. Then $G_L$
cannot be isomorphic to $\LargeRee{q_0}$, where $q_0 = 2^{n_0} > 2$ and $q =
q_0^{r_0}$ for some prime $r_0$. 
\end{lemma}
\begin{proof}
Suppose for a contradiction, $G_L \cong \LargeRee{q_0}$. Then by
\Cref{tab:bounds}, we have 
\[
|\cP| = \frac{1}{2} q^9 (q^6+1)(q^2+1)(q-1) \leqslant \frac{q^{18}}{2} \ 
\textrm{and}
\ |\cL| = \frac{q^{12}(q^6+1)(q^4-1)(q^3+1)(q-1)}{q_0^{12} 
(q_0^6+1)(q_0^4-1)(q_0^3+1)(q_0-1)} \geqslant 
\frac{1}{2} \paren{\frac{q}{q_0}}^{26}.
\]
By \Cref{lem:modhigman} \ref{item:modhighman2}, we have 
\[
\paren{\frac{1}{2} \paren{\frac{q}{q_0}}^{26}}^4 \leqslant |\cL|^4 < |\cP|^5 
\leqslant \paren{\frac{q^{18}}{2}}^5.
\]
Hence,
\[
\frac{1}{2^4} \paren{\frac{q}{q_0}}^{104} < \frac{q^{90}}{2^5}.
\]
Simplifying further, we obtain $q^{14} < q_0^{104}/2$. Therefore, $14n_0 r_0 <
104 n_0-1 < 104n_0$. 
Hence, $r_0 < 104/14 = 52/7 < 8$. Since $r_0$ is an odd prime, it
follows that $r_0 \in \{3,5,7\}$. \\

\underline{Suppose $r_0 = 3$.} Then 
\begin{align*}
\frac{s+1}{t+1} ={}&
\frac{q_0^{12}(q_0^6+1)(q_0^4-1)(q_0^3+1)(q_0-1)}{2q^3(q^2-1)(q^3+1)} \\
={}& 
\frac{q_0^{12}(q_0^6+1)(q_0^4-1)(q_0^3+1)(q_0-1)}{2q_0^9(q_0^6-1)(q_0^9+1)} \\
={}&
\frac{q_0^3(q_0^6+1)(q_0^2+1)(q_0-1)}{2(q_0^6-q_0^3+1)(q_0^4+q_0^2+1)}.
\end{align*}
Cross multiplying,
\[
(q_0^6-q_0^3+1)(q_0^4+q_0^2+1)(s+1) = \frac{q_0^3}{2}
(q_0^6+1)(q_0^2+1)(q_0-1) (t+1). 
\]

By \Cref{lem:coprime} ($i=3$), 
we find that $q_0^3(q_0^6+1)(q_0^2+1)(q_0-1)/2$ divides
$s+1$ and $(q_0^6-q_0^3+1)(q_0^4+q_0^2+1)$ divides $t+1$. Hence, 
\[
s+1 = k \frac{q_0^3}{2} (q_0^6+1)(q_0^2+1)(q_0-1) \ \textrm{and} \ t+1 =
k(q_0^6-q_0^3+1)(q_0^4+q_0^2+1), 
\]
where $k$ is some positive integer. 
As before, we deduce that $k$ is even. For if $k$ were to be odd, then $t$ is
even, which means that $|\cL| = (t+1)(st+1)$ is odd. This is not the case as
$q/q_0$ divides $|\cL|$. 
Therefore, $k \equiv 0$ or $2 \pmod 4$. \\ 

Suppose first that $k \equiv 0
\pmod 4$. Then $s \equiv -1 \pmod 4$ and $t \equiv -1 \pmod 4$. Whence, $st+1
\equiv 2 \pmod 4$. Thus, the $2$-part of $st+1$ is $2$ and consequently, the
$2$-part of $s+1$ is $q^9/4$. 
Since $s+1 = k \frac{q_0^3}{2} (q_0^6+1)(q_0^2+1)(q_0-1)$, the $2$-part of $k$
is $q^9/(2q_0^3) = q_0^{24}/2$. Therefore, 
\[
s+1 = k \frac{q_0^3}{2} (q_0^6+1)(q_0^2+1)(q_0-1) > k\frac{q_0^3}{2} \cdot
q_0^9 \geqslant \frac{q_0^{24}}{2} \cdot \frac{q_0^3}{2} \cdot q_0^9 =
\frac{q_0^{36}}{4}
\]
and 
\[
t+1 = k(q_0^6-q_0^3+1)(q_0^4+q_0^2+1) > k \cdot q_0^{10} \geqslant 
\frac{q_0^{24}}{2} \cdot q_0^{10} = \frac{q_0^{34}}{2}. 
\]
Applying \Cref{lem:modhigman} \ref{item:modhigman1}, we find that 
\[
|\cP| > \frac{(s+1)^2(t+1)}{2} > \frac{1}{2} \paren{\frac{q_0^{36}}{4}}^2
\frac{q_0^{34}}{2} = \frac{q_0^{106}}{2^6}.
\]
However, this contradicts the upper bound from \Cref{tab:bounds}, namely that 
$|\cP| \leqslant q^{18}/2 = q_0^{54}/2$. 
Thus, $k \equiv 2 \pmod 4$. \\ 

When $k \equiv 2 \pmod 4$, the
$2$-part of $s+1$ is $q_0^3$ and so, the $2$-part of $st+1$ must be
$q_0^{24}/2$. 
Now, we write $k$ in $q_0$-ary, \emph{i.e.}, $k = a_0 + a_1 q_0 + \cdots + 
a_n q_0^n$, where $a_0,a_1,\dots,a_n$ are integers such that 
$0 \leqslant a_0,a_1,\dots,a_n < q_0$. Using this expression for $k$, we 
compute $st+1$ modulo powers of $q_0$ and solve for $a_0,\dots,a_n$. This
computation is performed using Mathematica (see \Cref{sec:solving} 
\Cref{lis:stplus1r3}) and it 
is summarised in \Cref{tab:stplus1r3}. 
\begin{longtable}{lll}
\caption{$st+1$ modulo powers of $q_0$ for $G_{\alpha} \cong
\SUn{3}{q} : C_2$ or $\PGUn{3}{q} : C_2$ and $G_L \cong \LargeRee{q_0}$, where
$q = q_0^3$. \label{tab:stplus1r3}} \\
\toprule
\thead{$c$} & \thead{$st+1 \pmod {q_0^c}$} & \thead{Solution} \\
\midrule
$1$ & $2-a_0$ & $a_0 = 2$ \\
\midrule
$2$ & $-a_1q_0$ & $a_1 = 0$ \\
\midrule
$3$ & $(-2-a_2)q_0^2$ & $a_2 = q_0-2$ \\
\midrule
$4$ & $-a_3 q_0^3$ & $a_3 = 0$ \\
\midrule
$5$ & $(1-a_4)q_0^4$ & $a_4 = 1$ \\
\midrule
$6$ & $(-1-a_5+q_0/2)q_0^5$ & $a_5 = q_0/2-1$ \\
\midrule
$7$ & $-a_6 q_0^6$ & $a_6 = 0$ \\
\midrule
$8$ & $-a_7 q_0^7$ & $a_7 = 0$ \\
\midrule
$9$ & $(-1-a_8-9q_0/4)q_0^8$ & $a_8 = 3q_0/4-1$ \\
\bottomrule
\end{longtable}

Therefore, $k \geqslant q_0^8$. Hence,
\[
s+1 = k \frac{q_0^3}{2}(q_0^6+1)(q_0^2+1)(q_0-1) > k \frac{q_0^{12}}{2}
\geqslant \frac{q_0^{20}}{2} 
\]
and 
\[
t+1 = k(q_0^6-q_0^3+1)(q_0^4+q_0^2+1) > k \cdot q_0^{10} \geqslant q_0^{18}.
\]
Applying \Cref{lem:modhigman} \ref{item:modhigman1}, we find that 
\[
|\cP| > \frac{(s+1)^2(t+1)}{2} > \frac{1}{2} \paren{\frac{q_0^{20}}{2}}^2
q_0^{18} = \frac{q_0^{58}}{2^3}.
\]
However, this contradicts the upper bound from \Cref{tab:bounds}, namely that 
$|\cP| \leqslant q^{18}/2 = q_0^{54}/2$. Therefore, $r_0 \neq 3$. \\ 

\underline{Suppose that $r_0 = 5$.} Then 
\begin{align*}
\frac{s+1}{t+1} ={}&
\frac{q_0^{12}(q_0^6+1)(q_0^4-1)(q_0^3+1)(q_0-1)}{2q^3(q^2-1)(q^3+1)} \\
={}& 
\frac{q_0^{12}(q_0^6+1)(q_0^4-1)(q_0^3+1)(q_0-1)}{2q_0^{15}(q_0^{10}-1)
(q_0^{15}+1)} \\
={}&
\frac{(q_0-1)(q_0^2+1)(q_0^6+1)}{2q_0^3(q_0^4-q_0^3+q_0^2-q_0+1)^2 
(q_0^4+q_0^3+q_0^2+q_0+1)(q_0^8+q_0^7-q_0^5-q_0^4-q_0^3+q_0+1)}.
\end{align*}

Cross multiplying, we obtain
\[
2q_0^3(q_0^4-q_0^3+q_0^2-q_0+1)^2 
(q_0^4+q_0^3+q_0^2+q_0+1)(q_0^8+q_0^7-q_0^5-q_0^4-q_0^3+q_0+1)(s+1) 
= (q_0-1)(q_0^2+1)(q_0^6+1)(t+1).
\]

By \Cref{lem:coprime} ($i=4$), we find that $(q_0-1)(q_0^2+1)(q_0^6+1)$ 
divides $s+1$ and $2q_0^3(q_0^4-q_0^3+q_0^2-q_0+1)^2 
(q_0^4+q_0^3+q_0^2+q_0+1)(q_0^8+q_0^7-q_0^5-q_0^4-q_0^3+q_0+1)$ 
divides $t+1$. Hence, 
\[
s+1 = k (q_0-1)(q_0^2+1)(q_0^6+1)
\]
and 
\[
t+1 = 2kq_0^3(q_0^4-q_0^3+q_0^2-q_0+1)^2 
(q_0^4+q_0^3+q_0^2+q_0+1)(q_0^8+q_0^7-q_0^5-q_0^4-q_0^3+q_0+1),
\]
where $k$ is some positive integer. 
As before, we deduce that $k$ is even. For if $k$ were to be odd, then $s$ is
even, which means that $|\cP| = (s+1)(st+1)$ is odd. This is not the case as
$q$ divides $|\cP|$. 
Therefore, $k \equiv 0$ or $2 \pmod 4$. \\ 

Suppose first that $k \equiv 0
\pmod 4$. Then $s \equiv -1 \pmod 4$ and $t \equiv -1 \pmod 4$. Whence, $st+1
\equiv 2 \pmod 4$. Thus, the $2$-part of $st+1$ is $2$ and consequently, the
$2$-part of $s+1$ is $q^9/4$. Since $s+1 = k(q_0-1)(q_0^2+1)(q_0^6+1)$, the
$2$-part of $k$ is $q^9/4$.
From this, we obtain the following bounds: 
\[
s+1 = k (q_0-1)(q_0^2+1)(q_0^6+1) > k \frac{q_0^9}{2} \geqslant 
\frac{q^9}{2} \cdot \frac{q_0^9}{2} = \frac{q_0^{54}}{4} 
\]
and 
\begin{align*}
t+1 ={}& 2kq_0^3(q_0^4-q_0^3+q_0^2-q_0+1)^2 
(q_0^4+q_0^3+q_0^2+q_0+1)(q_0^8+q_0^7-q_0^5-q_0^4-q_0^3+q_0+1) \\ 
>{}& 2kq_0^3 \cdot
q_0^{20} \geqslant 2 \cdot \frac{q^9}{4} \cdot q_0^3 \cdot q_0^{20} =
\frac{q_0^{68}}{2}.
\end{align*}

Applying \Cref{lem:modhigman} \ref{item:modhigman1}, we find that 
\[
|\cP| > \frac{(s+1)^2(t+1)}{2} > \frac{1}{2} \paren{\frac{q_0^{54}}{4}}^2
\frac{q_0^{68}}{2} = \frac{q_0^{176}}{2^6}.
\]
However, this contradicts the upper bound from \Cref{tab:bounds}, namely that 
$|\cP| \leqslant q^{18}/2 = q_0^{90}/2$. 
So, in fact, $k \equiv 2 \pmod 4$. In this case, the
$2$-part of $s+1$ is $2$ and the $2$-part of $st+1$ must be
$q^9/4 = q_0^{45}/4$. 
Now, we write $k$ in $q_0$-ary, \emph{i.e.}, $k = a_0 + a_1 q_0 + \cdots + 
a_n q_0^n$, where $a_0,a_1,\dots,a_n$ are integers such that 
$0 \leqslant a_0,a_1,\dots,a_n < q_0$. Using this expression for $k$, we 
compute $st+1$ modulo powers of $q_0$ and solve for $a_0,\dots,a_n$. 
This computation is performed using Mathematica (see
\Cref{sec:solving} \Cref{lis:stplus1r5} for $q_0 > 174776$ and 
\Cref{lis:finiteq} lines \ref{line:stplus1r51}--\ref{line:stplus1r52} 
for $q_0 \leqslant 174776$). This data is summarised in
\Cref{tab:stplus1r5} for $q_0 > 174776$ and in \Cref{tab:stplus1r5finite} 
for $q_0 \leqslant 174776$. 

\begin{longtable}{lll}
\caption{$st+1$ modulo powers of $q_0$ for $G_{\alpha} \cong
\SUn{3}{q} : C_2$ or $\PGUn{3}{q} : C_2$ and $G_L \cong \LargeRee{q_0}$, where
$q = q_0^5$.}
\label{tab:stplus1r5}
\endfirsthead
\endhead
\toprule
\thead{$c$} & \thead{$st+1 \pmod {q_0^c}$} & \thead{Solution} \\
\midrule
$1$ & $2+a_0$ & $a_0 = q_0-2$ \\
\midrule
$2$ & $(3+a_1)q_0$ & $a_1 = q_0-3$ \\
\midrule
$3$ & $(1+a_2)q_0^2$ & $a_2 = q_0-1$ \\
\midrule
$4$ & $(-3+a_3)q_0^3$ & $a_3 = 3$ \\
\midrule
$5$ & $(-6+a_4)q_0^4$ & $a_4 = 6$ \\
\midrule
$6$ & $(-6+a_5)q_0^5$ & $a_5 = 6$ \\
\midrule
$7$ & $(18+a_6)q_0^6$ & $a_6 = q_0-18$ \\
\midrule
$8$ & $(63+a_7)q_0^7$ & $a_7 = q_0-63$ \\ 
\midrule
$9$ & $(107+a_8)q_0^8$ & $a_8 = q_0-107$ \\
\midrule
$10$ & $(-65+a_9)q_0^9$ & $a_9 = 65$ \\
\midrule
$11$ & $(-562+a_{10})q_0^{10}$ & $a_{10} = 562$ \\
\midrule
$12$ & $(-1350+a_{11})q_0^{11}$ & $a_{11} = 1350$ \\
\midrule
$13$ & $(-528+a_{12})q_0^{12}$ & $a_{12} = 528$ \\
\midrule
$14$ & $(4488+a_{13})q_0^{13}$ & $a_{13} = q_0-4488$ \\
\midrule
$15$ & $(16011+a_{14})q_0^{14}$ & $a_{14} = q_0-16011$ \\
\midrule
$16$ & $(18351+a_{15})q_0^{15}$ & $a_{15} = q_0-18351$ \\
\midrule
$17$ & $(-26175+a_{16})q_0^{16}$ & $a_{16} = 26175$ \\
\midrule
$18$ & $(-174776+a_{17})q_0^{17}$ & $a_{17} = 174776$ \\
\bottomrule
\end{longtable}

\begin{longtable}{ll}
\caption{$a_{17}$ for $8 < q_0 \leqslant 174776$.
\label{tab:stplus1r5finite}} \\
\toprule
\thead{$q_0$} & \thead{$a_{17}$} \\ 
\midrule
32 & 23 \\
\midrule
128 & 3 \\
\midrule
512 & 235 \\
\midrule
2048 & 708 \\
\midrule
8192 & 2747 \\
\midrule
32768 & 10936 \\
\midrule
131072 & 43704 \\
\bottomrule
\end{longtable}

For $q_0 = 8$, we have that $k \geqslant q_0^{20}$ and 
from \Cref{tab:stplus1r5,tab:stplus1r5finite}, we see that 
$k \geqslant q_0^{17}$ for all $q_0 > 8$. So, $k \geqslant q_0^{17}$ for every
$q_0 \geqslant 8$. Thus, we obtain the bounds: 
\[
s+1 = k (q_0-1)(q_0^2+1)(q_0^6+1) > k \frac{q_0^9}{2} \geqslant 
q_0^{17} \cdot \frac{q_0^9}{2} = \frac{q_0^{26}}{2}
\]
and 
\begin{align*}
t+1 ={}& 2kq_0^3(q_0^4-q_0^3+q_0^2-q_0+1)^2 
(q_0^4+q_0^3+q_0^2+q_0+1)(q_0^8+q_0^7-q_0^5-q_0^4-q_0^3+q_0+1) \\ 
>{}& 2kq_0^3 \cdot
q_0^{20} \geqslant 
2 \cdot q_0^{17} \cdot q_0^3 \cdot q_0^{20} = 2 q_0^{40}.
\end{align*}

Applying \Cref{lem:modhigman} \ref{item:modhigman1}, we find that 
\[
|\cP| > \frac{(s+1)^2(t+1)}{2} > \frac{1}{2} \paren{\frac{q_0^{26}}{2}}^2
\cdot 2q_0^{40} > \frac{q_0^{92}}{4}.
\]
However, this contradicts the upper bound from \Cref{tab:bounds}, namely that 
$|\cP| \leqslant q^{18}/2 = q_0^{90}/2$. Therefore, $r_0 \neq 5$. \\

\underline{Finally, suppose that $r_0 = 7$.} We have 
\begin{align*}
\frac{s+1}{t+1} ={}&
\frac{q_0^{12}(q_0^6+1)(q_0^4-1)(q_0^3+1)(q_0-1)}{2q^3(q^2-1)(q^3+1)} \\
={}& 
\frac{q_0^{12}(q_0^6+1)(q_0^4-1)(q_0^3+1)(q_0-1)}{2q_0^{21}(q_0^{14}-1)
(q_0^{21}+1)} \\
={}&
\frac{(q_0-1)(q_0^2+1)(q_0^6+1)}{2q_0^9
(q_0^6-q_0^5+q_0^4-q_0^3+q_0^2-q_0+1)^2 
(q_0^6+q_0^5+q_0^4+q_0^3+q_0^2+q_0+1)} \\
& \quad \ \ 
\frac{1}{(q_0^{12}+q_0^{11}-q_0^9-q_0^8+q_0^6-q_0^4-q_0^3+q_0+1)}.
\end{align*}

Cross multiplying, we obtain
\begin{align*}
2q_0^9
(q_0^6-q_0^5+q_0^4-q_0^3+q_0^2-q_0+1)^2 
(q_0^6+q_0^5+q_0^4+q_0^3+q_0^2+q_0+1) & \\
(q_0^{12}+q_0^{11}-q_0^9-q_0^8+q_0^6-q_0^4-q_0^3+q_0+1)(s+1) 
={}& (q_0-1)(q_0^2+1)(q_0^6+1)(t+1).
\end{align*}

By \Cref{lem:coprime} ($i=5$), 
we find that $(q_0-1)(q_0^2+1)(q_0^6+1)/\xi_{q_0}$ divides $s+1$ and 
$2q_0^9
(q_0^6-q_0^5+q_0^4-q_0^3+q_0^2-q_0+1)^2 (q_0^6+q_0^5+q_0^4+q_0^3+q_0^2+q_0+1)
(q_0^{12}+q_0^{11}-q_0^9-q_0^8+q_0^6-q_0^4-q_0^3+q_0+1)/\xi_{q_0}$ 
divides $t+1$, where 
\[
\xi_{q_0} \defi \begin{cases}
7 & \textrm{if $q_0 \equiv 1 \pmod 7$}, \\ 
1 & \textrm{if $q_0 \nequiv 1 \pmod 7$}. 
\end{cases}
\]
Hence, 
\[
s+1 = k (q_0-1)(q_0^2+1)(q_0^6+1)/\xi_{q_0}
\]
and 
\[
t+1 = 
2kq_0^9
(q_0^6-q_0^5+q_0^4-q_0^3+q_0^2-q_0+1)^2 (q_0^6+q_0^5+q_0^4+q_0^3+q_0^2+q_0+1)
(q_0^{12}+q_0^{11}-q_0^9-q_0^8+q_0^6-q_0^4-q_0^3+q_0+1)/\xi_{q_0}, 
\]
where $k$ is some positive integer. 
As before, we deduce that $k$ is even. For if $k$ were to be odd, then $s$ is
even, which means that $|\cP| = (s+1)(st+1)$ is odd. This is not the case as
$q$ divides $|\cP|$. Therefore, $k \equiv 0$ or $2 \pmod 4$. \\ 

Suppose first
that $k \equiv 0
\pmod 4$. Then $s \equiv -1 \pmod 4$ and $t \equiv -1 \pmod 4$. Whence, $st+1
\equiv 2 \pmod 4$. Thus, the $2$-part of $st+1$ is $2$ and consequently, the
$2$-part of $s+1$ is $q^9/4$. Since 
$s+1 = k(q_0-1)(q_0^2+1)(q_0^6+1)/\xi_{q_0}$, the $2$-part of $k$ is $q^9/4$.
From this, we obtain the following bounds: 
\[
s+1 = k (q_0-1)(q_0^2+1)(q_0^6+1)/\xi_{q_0} > k \frac{q_0^9}{2\xi_{q_0}} 
\geqslant 
\frac{q^9}{2} \cdot \frac{q_0^9}{2\xi_{q_0}} =
\frac{q_0^{72}}{4\xi_{q_0}} 
\]
and 
\begin{align*}
t+1 ={}& 2kq_0^9
(q_0^6-q_0^5+q_0^4-q_0^3+q_0^2-q_0+1)^2 (q_0^6+q_0^5+q_0^4+q_0^3+q_0^2+q_0+1)
\\
& \qquad (q_0^{12}+q_0^{11}-q_0^9-q_0^8+q_0^6-q_0^4-q_0^3+q_0+1)/\xi_{q_0} \\
>{}& 2k q_0^9 \cdot q_0^{30}/\xi_{q_0} \geqslant 
2 \cdot \frac{q^9}{4} \cdot q_0^3 \cdot q_0^{30}/\xi_{q_0} = 
\frac{q_0^{96}}{2\xi_{q_0}}.
\end{align*}
Applying \Cref{lem:modhigman} \ref{item:modhigman1}, we find that 
\[
|\cP| > \frac{(s+1)^2(t+1)}{2} > \frac{1}{2} \paren{\frac{q_0^{72}}{4
\xi_{q_0}}}^2
\frac{q_0^{96}}{2\xi_{q_0}} = \frac{q_0^{240}}{2^6\xi_{q_0}^3}.
\]
However, this contradicts the upper bound from \Cref{tab:bounds}, namely that 
$|\cP| \leqslant q^{18}/2 = q_0^{126}/2$. 
Thus, $k \equiv 2 \pmod 4$. \\ 

When $k \equiv 2 \pmod 4$, the
$2$-part of $s+1$ is $2$ and the $2$-part of $st+1$ must be
$q^9/4 = q_0^{63}/4$. 
Now, we write $k$ in $q_0$-ary, \emph{i.e.}, $k = a_0 + a_1 q_0 + \cdots + 
a_n q_0^n$, where $a_0,a_1,\dots,a_n$ are integers such that 
$0 \leqslant a_0,a_1,\dots,a_n < q_0$. Using this expression for $k$, we 
compute $\xi_{q_0}(st+1)$ modulo powers of $q_0$ and solve for 
$a_0,\dots,a_n$.\footnote{The reason for considering $\xi_{q_0}(st+1)$ instead
of simply $st+1$ is so that we don't run across conditions where we
have to use the multiplicative inverse of $\xi_{q_0}$ modulo powers of $q_0$.
This problem occurs because $s+1$ and $t+1$ contain $\xi_{q_0}$ in the
denominator.} This 
computation is performed using Mathematica (see
\Cref{sec:solving} \Cref{lis:stplus1r7}) and it is summarised in
\Cref{tab:stplus1r7}. For the purposes of solving for some 
of the $a_i$ in \Cref{tab:stplus1r7}, 
we shall assume that $q_0 > 2451 = 350 \cdot 7 + 1 
\geqslant 350\xi_{q_0}+1$. 
For $q_0 \leqslant 2451$, we find that $q_0 \in \{8,32,128,512,2048\}$. We
compute $a_{24}$ for these values of $q_0$ separately using Mathematica, see
\Cref{sec:solving} \Cref{lis:finiteq} lines
\ref{line:stplus1r71}--\ref{line:stplus1r72}. These values are recorded in
\Cref{tab:stplus1r7finite}. 

\afterpage{
\begin{longtable}{lll}
\caption{$st+1$ modulo powers of $q_0$ for $G_{\alpha} \cong \SUn{3}{q} : C_2$
or $\PGUn{3}{q} : C_2$ and $G_L \cong \LargeRee{q_0}$ where $q = q_0^7$.
\label{tab:stplus1r7}} \\
\toprule
\thead{$c$} & \thead{$st+1 \pmod {q_0^c}$} & \thead{Solution} \\
\midrule
$1$ & $2\xi_{q_0}+a_0$ & $a_0 = q_0-2\xi_{q_0}$ \\
\midrule
$2$ & $(1+2\xi_{q_0}+a_1)q_0$ & $a_1 = q_0-2\xi_{q_0}-1$ \\
\midrule
$3$ & $(1+a_2)q_0^2$ & $a_2 = q_0-1$ \\
\midrule
$4$ & $(1+a_3)q_0^3$ & $a_3 = q_0-1$ \\
\midrule
$5$ & $(1+2\xi_{q_0}+a_4)q_0^4$ & $a_4 = q_0-2\xi_{q_0}-1$ \\
\midrule
$6$ & $(1+2\xi_{q_0}+a_5)q_0^5$ & $a_5 = q_0-2\xi_{q_0}-1$ \\
\midrule
$7$ & $(1-2\xi_{q_0}+a_6)q_0^6$ & $a_6 = 2\xi_{q_0}-1$ \\
\midrule
$8$ & $(-2\xi_{q_0}+a_7)q_0^7$ & $a_7 = 2\xi_{q_0}$ \\
\midrule
$9$ & $(2\xi_{q_0}+a_8)q_0^8$ & $a_8 = q_0-2\xi_{q_0}$ \\
\midrule
$10$ & $(1-2\xi_{q_0}+a_9)q_0^9$ & $a_9 = 2\xi_{q_0}-1$ \\
\midrule
$11$ & $(-10\xi_{q_0}+a_{10})q_0^{10}$ & $a_{10} = 10\xi_{q_0}$ \\
\midrule
$12$ & $(-10\xi_{q_0}+a_{11})q_0^{11}$ & $a_{11} = 10\xi_{q_0}$ \\
\midrule
$13$ & $a_{13}q_0^{12}$ & $a_{12} = 0$ \\
\midrule
$14$ & $(-4\xi_{q_0}+a_{13})q_0^{13}$ & $a_{13} = 4\xi_{q_0}$ \\
\midrule
$15$ & $(-18\xi_{q_0}+a_{14})q_0^{14}$ & $a_{14} = 18\xi_{q_0}$ \\
\midrule
$16$ & $(-14\xi_{q_0}+a_{15})q_0^{15}$ & $a_{15} = 14\xi_{q_0}$ \\
\midrule
$17$ & $(4\xi_{q_0}+a_{16})q_0^{16}$ & $a_{16} = q_0-4\xi_{q_0}$ \\
\midrule
$18$ & $(1+a_{17})q_0^{17}$ & $a_{17} = q_0-1$ \\
\midrule
$19$ & $(1+a_{18})q_0^{18}$ & $a_{18} = q_0-1$ \\
\midrule
$20$ & $(1+52\xi_{q_0}+a_{19})q_0^{19}$ & $a_{19} = q_0-52\xi_{q_0}-1$ \\
\midrule
$21$ & $(1+128\xi_{q_0}+a_{20})q_0^{20}$ & $a_{20} = q_0-128\xi_{q_0}-1$ \\
\midrule
$22$ & $(1+116\xi_{q_0}+a_{21})q_0^{21}$ & $a_{21} = q_0-116\xi_{q_0}-1$ \\
\midrule
$23$ & $(1+100\xi_{q_0}+a_{22})q_0^{22}$ & $a_{22} = q_0-100\xi_{q_0}-1$ \\
\midrule
$24$ & $(1+204\xi_{q_0}+a_{23})q_0^{23}$ & $a_{23} = q_0-204\xi_{q_0}-1$ \\
\midrule
$25$ & $(1+350\xi_{q_0}+a_{24})q_0^{24}$ & $a_{24} = q_0-350\xi_{q_0}-1$ \\
\bottomrule
\end{longtable}
}
\begin{longtable}{lll}
\caption{$a_{24}$ for $8 \leqslant q_0 \leqslant 2451$. 
\label{tab:stplus1r7finite}} \\ 
\toprule
\thead{$q_0$} & \thead{$a_{17}$} \\ 
\midrule
8 & 6 \\
\midrule
32 & 27 \\
\midrule
128 & 32 \\
\midrule
512 & 161 \\
\midrule
2048 & 1697 \\
\bottomrule
\end{longtable}
From \Cref{tab:stplus1r7,tab:stplus1r7finite}, we see that $k \geqslant
a_{24} q_0^{24} \geqslant q_0^{24}$. From which, we obtain the bounds: 
\[
s+1 = k (q_0-1)(q_0^2+1)(q_0^6+1)/\xi_{q_0} > k \frac{q_0^9}{2\xi_{q_0}} 
\geqslant q_0^{24} \cdot \frac{q_0^9}{2\xi_{q_0}} =
\frac{q_0^{33}}{2\xi_{q_0}}
\]
and 
\begin{align*}
t+1 ={}& 2kq_0^9
(q_0^6-q_0^5+q_0^4-q_0^3+q_0^2-q_0+1)^2 (q_0^6+q_0^5+q_0^4+q_0^3+q_0^2+q_0+1)
\\ 
& \quad (q_0^{12}+q_0^{11}-q_0^9-q_0^8+q_0^6-q_0^4-q_0^3+q_0+1)/\xi_{q_0} \\
>{}& 2kq_0^9 \cdot q_0^{30}/\xi_{q_0} \geqslant 2 \cdot q_0^{24} \cdot 
q_0^{9} \cdot q_0^{30}/\xi_{q_0} = 2 q_0^{63}/\xi_{q_0}.
\end{align*}

Applying \Cref{lem:modhigman} \ref{item:modhigman1}, we find that 
\[
|\cP| > \frac{(s+1)^2(t+1)}{2} > \frac{1}{2} \paren{q_0^{33}/\xi_{q_0}}^2
\cdot 2 q_0^{63}/\xi_{q_0} \geqslant q_0^{129}/\xi_{q_0}^2.
\]
However, this contradicts the upper bound from \Cref{tab:bounds}, namely that 
$|\cP| 
\leqslant q^{18}/2 = q_0^{126}/2$. Thus, $r_0 \neq 7$ and so $G_L \ncong
\LargeRee{q_0}$. 
\end{proof}

\begin{lemma}
\label{lem:su3last}
Suppose $G_{\alpha} \cong \SUn{3}{q} : C_2$ or $\PGUn{3}{q} : C_2$. Then $G_L$
cannot be isomorphic to $C_{q+1}:\GLn{2}{3}$, $C_{q+\epsilon\sqrt{2q}+1}:
4S_4$ or $C_{q^2+q+1+\epsilon\sqrt{2q}(q+1)}:C_{12}$, where $\epsilon \in
\{1,-1\}$. 
\end{lemma}
\begin{proof}
Suppose for a contradiction that $G_L$ is isomorphic to 
$C_{q+1}:\GLn{2}{3}$, $C_{q+\epsilon\sqrt{2q}+1}:4S_4$ or 
$C_{q^2+q+1\epsilon\sqrt{2q}(q+1)}:C_{12}$. Then 
from \Cref{tab:bounds}, we have 
\[
|\cP| = \frac{1}{2}q^9(q^6+1)(q^2+1)(q-1) \leqslant \frac{q^{18}}{2}
\]
and 
\[
|\cL| \geqslant \frac{q^{24}}{384}.
\]
Applying \Cref{lem:modhigman} \ref{item:modhighman2}, we find that 
\[
\paren{\frac{q^{24}}{384}}^4 \leqslant |\cL|^4 < |\cP|^5 \leqslant
\paren{\frac{q^{18}}{2}}^5,
\]
whence, 
\[
\frac{q^{96}}{384^4} < \frac{q^{90}}{2^5}.
\] 
From this, we see that $q^6 < 384^4/2^5$, that is, $q < 29$. Furthermore, 
it can be verified by computer that there are no solutions for $s$ and $t$
when $q < 29$ (see \Cref{sec:finite} \Cref{lis:finite} lines
\ref{line:small1}--\ref{line:small2}). Therefore, $G_L$
cannot be isomorphic to $C_{q+1}:\GLn{2}{3}$, $C_{q+\epsilon\sqrt{2q}+1}:
4S_4$ or $C_{q^2+q+1+\epsilon\sqrt{2q}(q+1)}:C_{12}$.
\end{proof}

We have now established that $G_{\alpha}$ is not isomorphic to 
$\SUn{3}{q} : C_2$ or $\PGUn{3}{q} : C_2$ and as mentioned before, by duality,
$G_L$ is not isomorphic to $\SUn{3}{q} : C_2$ or $\PGUn{3}{q}$. 
Next, we show in Lemmata \ref{lem:szfirst}--\ref{lem:szlast}, 
that $G_{\alpha}$ cannot be isomorphic to $\Sz{q} \wr C_2$. As before, we 
show this by assuming
$G_{\alpha} \cong \Sz{q} \wr C_2$ and study the various possibilities for
$G_L$. 

\begin{lemma}
\label{lem:szfirst}
Suppose $G_{\alpha} \cong \Sz{q} \wr C_2$. Then $G_L$ cannot be isomorphic to
$\SPn{4}{q} : C_2$. 
\end{lemma}
\begin{proof}
Suppose for a contradiction, $G_L \cong \SPn{4}{q} : C_2$. Then 
\[
|\cP| = (s+1)(st+1) = \frac{1}{2} q^8(q^4-q^2+1)(q^3+1)(q+1)
\ \textrm{and} \ |\cL| = \frac{1}{2} q^8 (q^6+1)(q^2-q+1).
\]
Hence,
\[
\frac{|\cP|}{|\cL|} = \frac{s+1}{t+1} = \frac{(q+1)^2}{q^2+1}. 
\] 
Cross multiplying, we obtain 
\[
(s+1)(q^2+1) = (t+1)(q+1)^2.
\]
By \Cref{lem:coprime} ($i=6$), we deduce that $(q+1)^2$ divides $s+1$ and
$q^2+1$ divides $t+1$. Hence, there exists an integer $k$ such that 
\[
s+1 = k(q+1)^2 \ \textrm{and} \ t+1 = k(q^2+1).
\]
First, notice that $k$ is even. For if $k$ were to be odd, then $s$ is even.
Thus, $|\cP| = (s+1)(st+1)$ is odd. However, this is not the case as $q$
divides $|\cP|$. Thus, $k \equiv 0$ or $2 \pmod 4$. In either of these cases,
we claim that $st+1 \nequiv 0 \pmod 4$. Indeed, note that $s \equiv k-1 \pmod 
4$ and $t \equiv k-1 \pmod 4$. Thus, 
\[
st+1 \equiv (k-1)^2+1 \pmod 4.
\]
Plugging in $k \equiv 0$ and $k \equiv 2 \pmod 4$, we verify that $st+1 \equiv
2 \pmod 4$. Hence, the $2$-part of $st+1$ is $2$ and consequently, the 
$2$-part of $s+1$ is $q^8/4$. 
Since $s+1 = k(q+1)^2$, the $2$-part of $k$ is $q^8/4$. Therefore,
\[
s+1 = k(q+1)^2 > kq^2 \geqslant \frac{q^8}{4} \cdot q^2 = \frac{q^{10}}{4}
\]
and 
\[
t+1 = k(q^2+1) > kq^2 \geqslant \frac{q^8}{4} \cdot q^2 = \frac{q^{10}}{4}.
\]
Applying \Cref{lem:modhigman} \ref{item:modhigman1}, we find that 
\[
|\cP| > \frac{(s+1)^2(t+1)}{2} > \frac{1}{2} \paren{\frac{q^{10}}{4}}^2
\frac{q^{10}}{4} = \frac{q^{30}}{2^7}.
\]
However, this contradicts the upper bound from \Cref{tab:bounds}, namely that 
$|\cP| \leqslant q^{18}/2$. Thus, $G_L \ncong \SPn{4}{q} : C_2$. 
\end{proof}

\begin{lemma}
Suppose $G_{\alpha} \cong \Sz{q} \wr C_2$. Then $G_L$ cannot be isomorphic to
$\LargeRee{q_0}$ where $q_0 = 2^{n_0} > 2$ and $q = q_0^{r_0}$ for some prime
$r_0$. 
\end{lemma}
\begin{proof}
Suppose for a contradiction that $G_L \cong \LargeRee{q_0}$. 
Then by \Cref{tab:bounds}, we have  
\[
|\cP| = \frac{1}{2} q^8 (q^4-q^2+1)(q^3+1)(q+1) \leqslant q^{16} \ 
\textrm{and}
\ |\cL| = \frac{q^{12}(q^6+1)(q^4-1)(q^3+1)(q-1)}{q_0^{12} 
(q_0^6+1)(q_0^4-1)(q_0^3+1)(q_0-1)} \geqslant 
\frac{1}{2} \paren{\frac{q}{q_0}}^{26}.
\]
By \Cref{lem:modhigman} \ref{item:modhighman2}, we have 
\[
\paren{\frac{1}{2} \paren{\frac{q}{q_0}}^{26}}^4 \leqslant |\cL|^4 < |\cP|^5 
\leqslant \paren{q^{16}}^5.
\]
Hence,
\[
\paren{\frac{q}{q_0}}^{104} < 2^4 q^{80}
\]
Simplifying further, we obtain $q^{24} < 2^4 q_0^{104}$. 
Therefore, $24n_0 r_0 < 104 n_0+4$. Hence, $(24r_0-104) n_0 < 4$. Since $n_0
\geqslant 3$, 
it follows that $3(24r_0-104) < 4$. Thus, $24r_0-104 < 4/3 < 2$ and
because $24r_0-104$ is an integer, we have $24r_0-104 \leqslant 1$. Solving
for $r_0$ yields $r_0 \leqslant 105/24 < 5$. Since $r_0$ is an odd prime, it
follows that $r_0 = 3$. In which case, we have 
\begin{align*}
\frac{s+1}{t+1} ={}&
\frac{q_0^{12}(q_0^6+1)(q_0^4-1)(q_0^3+1)(q_0-1)}{2q^4(q^2+1)^2(q-1)^2} \\
={}& \frac{q_0^{12}(q_0^6+1)(q_0^4-1)(q_0^3+1)(q_0-1)}{2q_0^{12}
(q_0^6+1)^2(q_0^3-1)^2} \\
={}& \frac{(q_0+1)^2(q_0^2-q_0+1)}{2(q_0^2+q_0+1)^2(q_0^4-q_0^2+1)}. 
\end{align*}
Cross multiplying,
\[
2(q_0^2+q_0+1)^2(q_0^4-q_0^2+1)(s+1)
= (q_0+1)^2(q_0^2-q_0+1)(t+1). 
\]

By \Cref{lem:coprime} ($i=7$), 
we find that $(q_0+1)^2(q_0^2-q_0+1)$ divides $s+1$ and 
$2(q_0^2+q_0+1)^2(q_0^4-q_0^2+1)$ divides $t+1$. Hence, 
\[
s+1 = k (q_0+1)^2(q_0^2-q_0+1) \ \textrm{and} \ t+1 =
2k(q_0^2+q_0+1)^2(q_0^4-q_0^2+1), 
\]
where $k$ is some positive integer. 
As before, we deduce that $k$ is even. For if $k$ were to be odd, then $s$ is
even, which means that $|\cP| = (s+1)(st+1)$ is odd. This is not the case as
$q$ divides $|\cP|$. Therefore, $k \equiv 0$ or $2 \pmod 4$. Suppose first
that $k \equiv 0
\pmod 4$. Then $s \equiv -1 \pmod 4$ and $t \equiv -1 \pmod 4$. Whence, $st+1
\equiv 2 \pmod 4$. Thus, the $2$-part of $st+1$ is $2$ and consequently, the
$2$-part of $s+1$ is $q^8/4$. Since $s+1 = k(q_0+1)^2(q_0^2-q_0+1)$, the
$2$-part of $k$ is $q^8/4$. From this, we obtain the following bounds: 
\[
s+1 = k (q_0+1)^2(q_0^2-q_0+1) > k q_0^4 \geqslant \frac{q^8}{4} \cdot q_0^4 
= \frac{q_0^{24}}{4} \cdot q_0^4 = \frac{q_0^{28}}{4}
\]
and 
\begin{align*}
t+1 ={}& 2k(q_0^2+q_0+1)^2(q_0^4-q_0^2+1) \geqslant 2k q_0^8 \geqslant 2
\frac{q^8}{4} \cdot q_0^8 = \frac{q_0^{24}}{2} q_0^8 = \frac{q_0^{32}}{2}. 
\end{align*}
Applying \Cref{lem:modhigman} \ref{item:modhigman1}, we find that 
\[
|\cP| > \frac{(s+1)^2(t+1)}{2} > \frac{1}{2} \paren{\frac{q_0^{28}}{4}}^2
\frac{q_0^{32}}{2} = \frac{q_0^{88}}{2^5}.
\]
However, this contradicts the upper bound from \Cref{tab:bounds}, namely that 
$|\cP| \leqslant q^{16} = q_0^{48}$. So, in fact, $k \equiv 2 \pmod 4$. \\

Now, define 
\[
x \defi (q_0+1)^2 (q_0^2-q_0+1) \ \textrm{and} \ y \defi
2(q_0^2+q_0+1)^2(q_0^4-q_0^2+1). 
\]
Then, $s+1 = kx$ and $t+1 = ky$ where $k = \gcd(s+1,t+1)$. By
\Cref{lem:moddivcond2}, $s+t$ divides $4xy d^2$, where $d = \gcd(s,t)$. 
Note that by \Cref{cor:numcol}, $s$ divides $|\cP|-1$ and $t$
divides $|\cL|-1$. Thus, $d$ divides $\gcd(|\cP|-1,|\cL|-1)$. Finally, we can
apply \Cref{lem:boundgcd} to obtain $\gcd(|\cP|-1,|\cL|-1) =
\gcd(F_1(q_0),G(q_0)) \leqslant q_0$. Therefore, $d \leqslant q_0$. 
Consequently,
$s+t \leqslant 4xyd^2 \leqslant 4xy q_0^2$. Putting all this together, 
we have the following chain of inequalities: 
\[
ky = t+1 \leqslant s+t \leqslant 4xy q_0^2 = 4 (q_0+1)^2 (q_0^2-q_0+1) y 
q_0^2 \leqslant 4 \cdot 2q_0^2 q_0^2 y q_0^2 = 8y q_0^6.
\]  
Simplifying, we obtain 
\[
k \leqslant 8q_0^6.
\]
From this, we see that 
\[
|\cP| < (s+1)^2(t+1) = k^3 x^2 y \leqslant
2^{9} q_0^{18} x^2 \cdot y \leqslant 2^9 q_0^{18} \cdot ((2q_0^2)(q_0^2))^2 
\cdot 2(4q_0^4)(q_0^4) = 2^{14} q_0^{34}.
\]
However, this contradicts the lower bound from \Cref{tab:bounds}, namely that
$|\cP| \geqslant q^{16}/2 = q_0^{48}/2$. Therefore, $G_L \ncong
\LargeRee{q_0}$. 
\end{proof}
\begin{lemma}
\label{lem:szlast}
Suppose $G_{\alpha} \cong \Sz{q} \wr C_2$. Then $G_L$ cannot be isomorphic to 
$C_{q+1}^2 : \GLn{2}{3}$, $C_{q+\epsilon\sqrt{2q}+1}^2 : 4S_4$ or 
$C_{q^2+q+1+\epsilon\sqrt{2q}(q+1)} : C_{12}$ where $\epsilon \in \{1,-1\}$. 
\end{lemma}

\begin{proof}
Suppose for a contradiction that 
$G_L \cong C_{q+1}^2 : \GLn{2}{3}$, $C_{q+\epsilon\sqrt{2q}+1}^2 : 4S_4$ or
$C_{q^2+q+1+\epsilon\sqrt{2q}(q+1)} : C_{12}$. 
Then from \Cref{tab:bounds}, we have 
\[
|\cP| = \frac{1}{2}q^8 (q^4-q^2+1)(q^3+1)(q+1) \leqslant q^{16}
\]
and 
\[
|\cL| \geqslant \frac{q^{24}}{384}.
\]
Applying \Cref{lem:modhigman} \ref{item:modhighman2}, we find that 
\[
\paren{\frac{q^{24}}{384}}^4 \leqslant |\cL|^4 < |\cP|^5 \leqslant
\paren{q^{16}}^5.
\]
From this, we see that $q^{16} < 384^4$, \emph{i.e.}, $q < 5$. 
This is a contradiction as $q \geqslant 8$. Therefore, $G_L$
cannot be isomorphic to $C_{q+1}:\GLn{2}{3}$, 
$C_{q+\epsilon\sqrt{2q}+1}:4S_4$ or 
$C_{q^2+q+1+\epsilon\sqrt{2q}(q+1)}:C_{12}$.
\end{proof}

We have now established that $G_{\alpha}$ is not isomorphic to 
$\Sz{q} \wr C_2$. Next, we show in Lemmata 
\ref{lem:spnfirst}--\ref{lem:spnlast}, that $G_{\alpha}$
cannot be isomorphic to $\SPn{4}{q} : C_2$. As before, we show this by 
assuming $G_{\alpha} \cong \SPn{4}{q} : C_2$ and study the various 
possibilities for $G_L$. 
\begin{lemma}
\label{lem:spnfirst}
Suppose $G_{\alpha} \cong \SPn{4}{q} : C_2$. Then $G_L$ cannot be isomorphic 
to $\LargeRee{q_0}$ where $q_0 = 2^{n_0} > 2$ and $q = q_0^{r_0}$ for some
prime $r_0$. 
\end{lemma}
\begin{proof}
Let $G_{\alpha} \cong \SPn{4}{q} : C_2$ and $G_L \cong \LargeRee{q_0}$ where 
$q_0 = 2^{n_0} > 2$ and $q = q_0^{r_0}$ for some prime $r_0$. 
Then 
\[
|\cP| = \frac{1}{2} q^8 (q^6+1)(q^2-q+1) \leqslant \frac{q^{16}}{2} \ 
\textrm{and}
\ |\cL| = \frac{q^{12}(q^6+1)(q^4-1)(q^3+1)(q-1)}{q_0^{12} 
(q_0^6+1)(q_0^4-1)(q_0^3+1)(q_0-1)} \geqslant 
\frac{1}{2} \paren{\frac{q}{q_0}}^{26}.
\]
By \Cref{lem:modhigman} \ref{item:modhighman2}, we have 
\[
\paren{\frac{1}{2} \paren{\frac{q}{q_0}}^{26}}^4 \leqslant |\cL|^4 < |\cP|^5 
\leqslant \paren{\frac{q^{16}}{2}}^5.
\]
Simplifying further, we obtain $q^{24} < q_0^{104}$. 
Therefore, $24n_0 r_0 < 104 n_0$ and so, $24 r_0 < 104$. From which, we have
$r_0 < 104/24 < 5$. Since $r_0$ is an odd prime, it
follows that $r_0 = 3$. In which case, we have 

\begin{align*}
\frac{s+1}{t+1} ={}&
\frac{q_0^{12}(q_0^6+1)(q_0^4-1)(q_0^3+1)(q_0-1)}{2q^4(q^4-1)(q+1)(q-1)} \\
={}& 
\frac{q_0^{12}(q_0^6+1)(q_0^4-1)(q_0^3+1)(q_0-1)}{2q_0^{12}(q_0^{12}-1)
(q_0^3+1)(q_0^3-1)} \\
={}& \frac{q_0^2+1}{2(q_0^2+q_0+1)(q_0^4+q_0^2+1)}. 
\end{align*}
Cross multiplying,
\[
2(q_0^2+q_0+1)(q_0^4+q_0^2+1)(s+1)
= (q_0^2+1)(t+1). 
\]
By \Cref{lem:coprime} ($i=8$), we find that $q_0^2+1$ divides $s+1$ and 
$2(q_0^2+q_0+1)(q_0^4+q_0^2+1)$ divides $t+1$. Hence, 
\[
s+1 = k (q_0^2+1) \ \textrm{and} \ t+1 =
2k(q_0^2+q_0+1)(q_0^4+q_0^2+1), 
\]
where $k$ is some positive integer. 
As before, we deduce that $k$ is even. For if $k$ were to be odd, then $s$ is
even, which means that $|\cP| = (s+1)(st+1)$ is odd. This is not the case as
$q$ divides $|\cP|$. Therefore, $k \equiv 0$ or $2 \pmod 4$. \\ 

Suppose first
that $k \equiv 0
\pmod 4$. Then $s \equiv -1 \pmod 4$ and $t \equiv -1 \pmod 4$. Whence, $st+1
\equiv 2 \pmod 4$. Thus, the $2$-part of $st+1$ is $2$ and consequently, the
$2$-part of $s+1$ is $q^8/4$. Since $s+1 = k(q_0^2+1)$, the
$2$-part of $k$ is $q^8/4$. From this, we obtain the following bounds: 
\[
s+1 = k (q_0^2+1) > k q_0^2 \geqslant \frac{q^8}{4} \cdot q_0^2
= \frac{q_0^{24}}{4} \cdot q_0^2 = \frac{q_0^{26}}{4}
\]
and 
\begin{align*}
t+1 ={}& 2k(q_0^2+q_0+1)(q_0^4+q_0^2+1) \geqslant 2k q_0^6 \geqslant 2
\frac{q^8}{4} \cdot q_0^6 = \frac{q_0^{24}}{2} q_0^6 = \frac{q_0^{30}}{2}. 
\end{align*}
Applying \Cref{lem:modhigman} \ref{item:modhigman1}, we find that 
\[
|\cP| > \frac{(s+1)^2(t+1)}{2} > \frac{1}{2} \paren{\frac{q_0^{26}}{4}}^2
\frac{q_0^{30}}{2} = \frac{q_0^{82}}{2^5}.
\]
However, this contradicts the upper bound from \Cref{tab:bounds}, namely that 
$|\cP|
\leqslant q^{16}/2 = q_0^{48}/2$. Thus, $k \equiv 2 \pmod 4$. \\

Suppose $k \equiv 2 \pmod 4$. Now, define 
\[
x \defi (q_0^2+1) \ \textrm{and} \ y \defi
2(q_0^2+q_0+1)(q_0^4+q_0^2+1). 
\]
Then, $s+1 = kx$ and $t+1 = ky$ where $k = \gcd(s+1,t+1)$. By
\Cref{lem:moddivcond2}, $s+t$ divides $4xy d^2$, where $d = \gcd(s,t)$. 
Note that by \Cref{cor:numcol}, $s$ divides $|\cP|-1$ and $t$
divides $|\cL|-1$. Thus, $d$ divides $\gcd(|\cP|-1,|\cL|-1)$. Finally, we can
apply \Cref{lem:boundgcd} to obtain $\gcd(|\cP|-1,|\cL|-1) =
\gcd(F_2(q_0),G(q_0)) \leqslant q_0$. Therefore, $d \leqslant q_0$. 
Consequently,
$s+t \leqslant 4xyd^2 \leqslant 4xy q_0^2$. Putting all this together, 
we have the following chain of inequalities:  
\[
ky = t+1 \leqslant s+t \leqslant 4xy q_0^2 = 4 (q_0^2+1)  y 
q_0^2 \leqslant 4 \cdot 2q_0^2 y q_0^2 = 8y q_0^4.
\]  
Simplifying, we obtain $k \leqslant 8q_0^4$. From this, we see that 
\[
|\cP| < (s+1)^2(t+1) = k^3 x^2 y \leqslant
2^{9} q_0^{12} x^2 \cdot y \leqslant 2^9 q_0^{18} \cdot (2q_0^2)^2 
\cdot 2(2q_0^2)(2q_0^4) = 2^{14} q_0^{28}.
\]
However, this contradicts the lower bound from \Cref{tab:bounds}, namely that
$|\cP| \geqslant q^{16}/4 = q_0^{48}/4$. Therefore, $G_L \ncong
\LargeRee{q_0}$. 
\end{proof}

\begin{lemma}
\label{lem:spnlast}
Suppose $G_{\alpha} \cong \SPn{4}{q} : C_2$. 
Then $G_L$ cannot be isomorphic to 
$C_{q+1}^2 : \GLn{2}{3}$, $C_{q+\epsilon\sqrt{2q}+1}^2 : 4S_4$ or 
$C_{q^2+q+1+\epsilon\sqrt{2q}(q+1)} : C_{12}$, where $\epsilon \in \{1,-1\}$.
\end{lemma}
\begin{proof}
Suppose for a contradiction that 
$G_L \cong C_{q+1}^2 : \GLn{2}{3}$, $C_{q+\epsilon\sqrt{2q}+1}^2 : 4S_4$ or
$C_{q^2+q+1+\epsilon\sqrt{2q}(q+1)} : C_{12}$. 
Then from \Cref{tab:bounds}, we have 
\[
|\cP| = \frac{1}{2}q^8 (q^6+1)(q^2-q+1) \leqslant \frac{q^{16}}{2}
\]
and 
\[
|\cL| \geqslant \frac{q^{24}}{384}.
\]
Applying \Cref{lem:modhigman} \ref{item:modhighman2}, we find that 
\[
\paren{\frac{q^{24}}{384}}^4 \leqslant |\cL|^4 < |\cP|^5 \leqslant
\paren{\frac{q^{16}}{2}}^5. 
\]
From this, we see that $q^{16} < 192^4/2$, \emph{i.e.}, $q < 4$. 
This is a contradiction as $q \geqslant 8$. Therefore, $G_L$ cannot be 
isomorphic 
to $C_{q+1}^2 : \GLn{2}{3}$, $C_{q+\epsilon\sqrt{2q}+1}^2 : 4S_4$ or
$C_{q^2+q+1+\epsilon\sqrt{2q}(q+1)} : C_{12}$.  
\end{proof}

We have now established that $G_{\alpha}$ is not isomorphic to 
$\SPn{4}{q} : C_2$. Next, we show in Lemmata 
\ref{lem:q0first}--\ref{lem:q0last} that $G_{\alpha}$
cannot be isomorphic to $\LargeRee{q_0}$, where $q_0 = 2^{n_0} > 2$ and $q =
q_0^{r_0}$ for some prime $r_0$. As before, we show this by 
assuming $G_{\alpha} \cong \LargeRee{q_0}$ and study the various 
possibilities for $G_L$. 

\begin{lemma}
\label{lem:q0first}
Suppose $G_{\alpha} \cong \LargeRee{q_0}$ where $q_0 = 2^{n_0} > 2$ and $q =
q_0^{r_0}$ for some prime $r_0$. Then $G_L$ cannot be isomorphic to
$\LargeRee{q_1}$ where $q_1 = 2^{n_1} > 2$ and $q = q_0^{r_1}$ for some prime
$r_1 \neq r_0$. 
\end{lemma}
\begin{proof}
Suppose for a contradiction that $G_L \cong \LargeRee{q_1}$
where $q_1 = 2^{n_1} > 2$ and $q = q_1^{r_1}$ for some prime $r_1 \neq r_0$. 
Consider the element $u_5 \in G$ as per \Cref{tab:ccls}. 
Then by \Cref{lem:largereeccls}, 
$u_5^G \cap G_{\alpha} \neq \varnothing$ and 
$u_5^G \cap G_L \neq \varnothing$. Hence, we can take 
$g \in u_5^G \cap G_{\alpha}$ and $h \in u_5^G
\cap G_L$. Since $g$ and $h$ are conjugate in $G$, there exists an $x \in G$
such that $g = h^x$. Thus, $g \in G_L^x$. Since $G_L^x = G_{L^\prime}$ where
$L^{\prime} = L^x \in \cL$, we can take $L$ to be $L^{\prime}$ and assume that
we have $g \in G_{\alpha} \cap G_L$. Let $\cP_g$ and $\cL_g$ be 
the set of points and the set of lines fixed by $g$, respectively. 
Using \Cref{tab:ccls}, we find that $|\Centraliser{G}{g}| =
q^{8}$, $|\Centraliser{G_{\alpha}}{g}| = q_0^8$ and 
$|\Centraliser{G_L}{g}| = q_1^8$. \\ 

Since the conjugacy class of $g$ does not split in $G_{\alpha}$ 
and $G_L$ by \Cref{lem:largereeccls}, 
we can apply \Cref{lem:cgtranss} to conclude that
$\Centraliser{G}{g}$ acts transitively on $\cP_g$ and $\cL_g$. Next, the 
number of fixed points and lines are 
\[
|\cP_g| = \paren{\frac{q}{q_0}}^8 = 2^{a} \ \textrm{and} \ |\cL_g| =
\paren{\frac{q}{q_1}}^8 = 2^{b}, 
\]
where $a = 8n_0(r_0-1)$ and $b = 8n_1(r_1-1)$. 
By \Cref{lem:suzsub}, the
fixed substructure, $(\cP_g,\cL_g,\cI_g)$ is a generalised quadrangle of
order $(s^{\prime},t^{\prime})$ for some positive integers $s^{\prime}$ and
$t^{\prime}$. Therefore, 
\[
|\cP_g| = (s^{\prime}+1)(s^{\prime}t^{\prime}+1) = 2^a \ \textrm{and} \ 
|\cL_g| = (t^{\prime}+1)(s^{\prime}t^{\prime}+1) = 2^b.
\]
By \Cref{lem:suzupts}, we have that $b = a/2+1$ or dually, $a = b/2+1$. Let us
suppose that $b = a/2+1$. Thus, 
\[
8n_1(r_1-1) = 4n_0(r_0-1)+1.
\]
However, this is a contradiction as the left hand side is even and the right
hand side is odd. We reach a similar contradiction if we consider the dual
case: $a = b/2+1$. Therefore, $G_L \ncong \LargeRee{q_1}$.
\end{proof}

We continue our next case with $G_{\alpha} \cong \LargeRee{q_0}$ 
where $q_0 = 2^{n_0} > 2$ and $q = q_0^{r_0}$ for some prime $r_0$ and 
show that $G_L$ cannot be isomorphic to $C_{q+1}^2 : \GLn{2}{3}$. 
We split this into cases: $r_0 > 3$ and $r_0 = 3$. 
\begin{lemma}
Suppose $G_{\alpha} \cong \LargeRee{q_0}$ where $q_0 = 2^{n_0} > 2$ and $q =
q_0^{r_0}$ for some prime $r_0 > 3$. Then $G_L$ cannot be isomorphic to
$C_{q+1}^2 : \GLn{2}{3}$. 
\end{lemma}

\begin{proof}
Suppose for a contradiction that $G_L \cong C_{q+1}^2 : \GLn{2}{3}$. Then 
\begin{equation}
|\cP| = 
\frac{q^{12}(q^6+1)(q^4-1)(q^3+1)(q-1)}{q_0^{12}(q_0^6+1)(q_0^4-1)(q_0^3+1) 
(q_0-1)} \ \textrm{and} \ |\cL| = \frac{1}{48} q^{12}
(q^6+1)(q^2+1)(q-1)^2(q^2-q+1). \label{eqn:star}
\end{equation}
Thus,
\[
\frac{s+1}{t+1} =
\frac{48(q+1)^2}{q_0^{12}(q_0^6+1)(q_0^4-1)(q_0^3+1)(q_0-1)}.
\]
Cross multiplying, we obtain 
\[
\frac{1}{48}q_0^{12}(q_0^6+1)(q_0^2+1)(q_0-1)^2(q_0^2-q_0+1)(s+1) =
\paren{\frac{q+1}{q_0+1}}^2 (t+1).
\]
By \Cref{lem:coprime} ($i=9$), 
we find that $(q+1)^2/(q_0+1)^2$ divides $s+1$ and 
$q_0^{12}(q_0^6+1)(q_0^2+1)(q_0-1)^2(q_0^2-q_0+1)/48$ 
divides $t+1$. Hence, 
\[
s+1 = k\paren{\frac{q+1}{q_0+1}}^2 \ \textrm{and} \ t+1 = 
k\frac{q_0^{12}(q_0^6+1)(q_0^2+1)(q_0-1)^2(q_0^2-q_0+1)}{48}, 
\]
where $k$ is some positive integer. 
As before, we deduce that $k$ is even. For if $k$ were to be odd, then $s$ is
even, which means that $|\cP| = (s+1)(st+1)$ is odd. This is not the case as
$q/q_0$ divides $|\cP|$. Therefore, $k \equiv 0$ or $2 \pmod 4$. \\ 

Suppose first that $k \equiv 0
\pmod 4$, then $s \equiv -1 \pmod 4$ and $t \equiv -1 \pmod 4$. Whence, $st+1
\equiv 2 \pmod 4$. Thus, the $2$-part of $st+1$ is $2$ and consequently, the
$2$-part of $s+1$ is $(q/q_0)^{12}/2$. 
Since $s+1 = k(q+1)^2/(q_0+1)^2$, the $2$-part of $k$
is $(q/q_0)^{12}/2$. Therefore, 
\begin{equation}
s+1 = k\paren{\frac{q+1}{q_0+1}}^2 > k \cdot \frac{q^2}{4q_0^2} \geqslant 
\frac{1}{2} \paren{\frac{q}{q_0}}^{12} \cdot \frac{q^2}{4q_0^2} =
\frac{1}{2^3} \paren{\frac{q}{q_0}}^{14} \label{eqn:plus}
\end{equation}
and 
\begin{equation}
t+1 = k\frac{q_0^{12}(q_0^6+1)(q_0^2+1)(q_0-1)^2(q_0^2-q_0+1)}{48} \geqslant k
\cdot \frac{q_0^{24}}{96} \geqslant \frac{1}{2} \paren{\frac{q}{q_0}}^{12}
\cdot \frac{q_0^{24}}{96} = \frac{1}{192} (q q_0)^{12}. \label{eqn:plusplus}
\end{equation}

Applying \Cref{lem:modhigman} \ref{item:modhigman1}, we find that 
\[
|\cP| > \frac{(s+1)^2(t+1)}{2} > \frac{1}{2} 
\paren{\frac{(q/q_0)^{14}}{2^3}}^2 \frac{(qq_0)^{12}}{192} = 
\frac{1}{2^{7} \cdot 192} \frac{q^{40}}{q_0^{16}}.
\]
However, this contradicts the upper bound from \Cref{tab:bounds}, namely 
that $|\cP|
\leqslant 2 (q/q_0)^{26}$. So, in fact, $k \equiv 2 \pmod 4$. \\ 

Since $k \equiv 2 \pmod 4$, 
the $2$-part of $s+1$ is $2$. Since $|\cP| = (s+1)(st+1)$, the expression 
for $|\cP|$ in \eqref{eqn:star} implies that the $2$-part of $st+1$ must be
$(q/q_0)^{12}/2$. Now, write $k = 2\ell$ where $\ell$ is odd. 
By \Cref{lem:groupdivcond}, we find that $q+1$ divides $|\cL|-1$. 
Consequently, $(q+1)/(q_0+1)$ divides $|\cL|-1 = (t+1)(st+1)-1 = t^2s +
t(s+1)$. Since $(q+1)/(q_0+1)$ divides $s+1$, we can deduce that
$(q+1)/(q_0+1)$ divides $t^2s$. Moreover, $s$ and $s+1$ are coprime and so, 
$(q+1)/(q_0+1)$ divides $t^2$. Therefore, $s+t$ is coprime to $(q+1)/(q_0+1)$.
To see why, consider a prime $p$ dividing $(q+1)/(q_0+1)$. Then $s \equiv -1
\pmod p$ and $t \equiv 0 \pmod p$, whence, $s+t \equiv -1 \nequiv 0 \pmod p$. 
Furthermore, observe that $\ell$
is also coprime to $s+t$ since $s+t \equiv -2 \pmod p$ for any prime $p$
dividing $\ell$. 
Since $\ell$ is odd, $p$ is also odd and hence, $s+t \nequiv 0
\pmod p$. \\ 

Define 
\[
x \defi \paren{\frac{q+1}{q_0+1}}^2 \ \textrm{and} \ y \defi \frac{q_0^{12}(q_0^6+1)(q_0^2+1)(q_0-1)^2(q_0^2-q_0+1)}{48}
\]
so that $s+1 = 2\ell x$ and $t+1 = 2\ell y$. By \Cref{lem:condsonst}, $s+t$ divides 
\[
st(s+1)(t+1) = st \cdot 2\ell x \cdot 
2\ell 
y.
\]
Since $s+t$ is coprime to $\ell$ and $(q+1)/(q_0+1)$, we conclude that 
$s+t$ divides $4sty$. 
Applying \Cref{lem:moddivcond} with $A = 4y$, we can deduce that $s+t$ divides $y(s-t)^2 = y(s+1-(t+1))^2 = y(2\ell x - 2 \ell y)^2 = 4\ell^2 y(x-y)^2$. 
Since $s+t$ is coprime to $\ell$, it follows that $s+t$ divides $4y(x-y)^2$. 
Therefore, 
\begin{equation}
s+t \leqslant 4y(x-y)^2 = \paren{\paren{\frac{q+1}{q_0+1}}^2 - 
\frac{q_0^{12}(q_0^6+1)(q_0^2+1)(q_0-1)^2(q_0^2-q_0+1)}{48}}^2.
\label{eqn:splust}
\end{equation} \\

Suppose first that $r_0 \geqslant 13$. We prove that $y \leqslant x$. Using \Cref{tab:bounds} ($i=5$), we obtain 
\[
y = \frac{q_0^{12}(q_0^6+1)(q_0^2+1)(q_0-1)^2(q_0^2-q_0+1)}{48} \leqslant
\frac{q_0^{24}}{48} \leqslant \frac{q_0^{24}}{4} \leqslant
\frac{q_0^{2(r_0-1)}}{4} = \paren{\frac{q}{2q_0}}^2 \leqslant
\paren{\frac{q+1}{q_0+1}}^2 = x
\]
and so, $s+t \leqslant 4y (x-y)^2 \leqslant 4yx^2$.
In addition to this, we also have $s+t > s+1 = 2\ell x$. Thus, we have $2 \ell x = s+1 < s+t \leqslant 4yx^2$ and so, $\ell < 2xy$. Writing this out explicitly, we have 
\[
\ell < 2 \paren{\frac{q+1}{q_0+1}}^2 y \leqslant 2 \cdot \paren{\frac{2q}{q_0}}^2 \cdot \frac{q_0^{24}}{48}.
\]

Simplifying, we find  
\[
\ell < \frac{q_0^{24}}{6} \cdot q_0^{2(r_0-1)} < q_0^{2r_0+22}.
\]

Hence,  
\[
s+1 = 2\ell \paren{\frac{q+1}{q_0+1}}^2 < 2 \cdot q_0^{2r_0+22} \cdot
\paren{\frac{2q}{q_0}}^2 = 2 q_0^{2 r_0+22} \cdot 4q_0^{2(r_0-1)} = 2^3
q_0^{4r_0+20}  
\]
and 
\[
t+1 = 2\ell \frac{q_0^{12}(q_0^6+1)(q_0^2+1)(q_0-1)^2(q_0^2-q_0+1)}{48} < 2 
q_0^{2 r_0+22} \cdot \frac{q_0^{24}}{48} = \frac{q_0^{2r_0+46}}{24}.
\]
As a consequence, 
\[
|\cP| = (s+1)(st+1) < (s+1)(s+1)(t+1) = (s+1)^2(t+1) < \paren{2^3
q_0^{4r_0+20}}^2 \frac{q_0^{2r_0+46}}{24} = \frac{8q_0^{10r_0+86}}{3}.
\]
Using the lower bound from \Cref{tab:bounds}, namely, $|\cP| \geqslant
(q/q_0)^{26}/2$, we obtain 
\[
\frac{1}{2} q_0^{26(r_0-1)} \leqslant |\cP| < \frac{8}{3} q_0^{10r_0+86} <
4q_0^{10r_0+86}. 
\]
Simplifying further, we find $q_0^{16r_0} < 8q_0^{112} \leqslant q_0^{113}$, 
and so, $r_0 \leqslant 8$. This is a contradiction as we assumed $r_0
\geqslant 13$. Thus, $r_0 < 13$. \\

Since, $r_0$ is a prime, we have $r_0 \leqslant 11$. 
We show that $x \leqslant y$. Indeed, by \Cref{tab:bounds} 
($i=5$), we have $q_0^{24}/96 \leqslant y$ and so, we find 
{\small
\[
x = \paren{\frac{q+1}{q_0+1}}^2 \leqslant \paren{\frac{2q}{q_0}}^2 =
4q_0^{2(r_0-1)} \leqslant 
4q_0^{20} \leqslant
\frac{q_0^{24}}{96} \leqslant
\frac{q_0^{12}(q_0^6+1)(q_0^2+1)(q_0-1)^2(q_0^2-q_0+1)}{48} = y,
\]
}
In addition to this, we also have $s+t > t+1 = 2\ell y$. 
Combining this with inequality (\ref{eqn:splust})), we obtain 
$2\ell y < s+t < 4y(x-y)^2 \leqslant 4y y^2$. 
Simplifying and using $y \leqslant q_0^{24}/48$ from \Cref{tab:bounds} ($i=5$), 
we find 
\[
\ell < 2y^2 = 2 \cdot \paren{\frac{q_0^{24}}{48}}^2 = \frac{q_0^{48}}{1152} < \frac{q_0^{48}}{2^{10}}. 
\]
Hence, 
\[
s+1 = 2\ell \paren{\frac{q+1}{q_0+1}}^2 < 2 \cdot \frac{q_0^{48}}{2^{10}} \cdot
\paren{\frac{2q}{q_0}}^2 =  \frac{q_0^{48} \cdot q_0^{2(r_0-1)}}{2^7} = 
\frac{q_0^{2r_0+46}}{2^7}  
\]
and 
\[
t+1 = 2\ell \frac{q_0^{12}(q_0^6+1)(q_0^2+1)(q_0-1)^2(q_0^2-q_0+1)}{48} < 2
\cdot \frac{q_0^{48}}{2^{10}} \cdot \frac{q_0^{24}}{48} = 
\frac{q_0^{72}}{2^9 \cdot 48}.
\]
As a consequence, 
\[
|\cP| = (s+1)(st+1) < (s+1)((s+1)(t+1)) = (s+1)^2(t+1) < 
\paren{\frac{q_0^{2r_0+46}}{2^7}}^2 
\frac{q_0^{72}}{2^9 \cdot 48} = \frac{q_0^{4r_0+164}}{2^{23} \cdot 48}.
\]
Using the lower bound from \Cref{tab:bounds}, namely, $|\cP| \geqslant
(q/q_0)^{26}/2$, we obtain 
\[
\frac{1}{2} q_0^{26(r_0-1)} \leqslant |\cP| < 
\frac{q_0^{4r_0+164}}{2^{23} \cdot 48} < \frac{q_0^{4r_0+164}}{2^{28}}.
\]
Simplifying further, we find $q_0^{22r_0} < q_0^{190}/2^{27}$, 
from which, we see that $22r_0 < 190$, \emph{i.e.}, $r_0 \leqslant 8$. Since $r_0$ is a prime greater than $3$ (by hypothesis), we have that $r_0 \in \{5,7\}$. \\

\underline{Suppose now $r_0 = 5$.} Then by \eqref{eqn:plus}, 
\[
s+1 = k\paren{\frac{q+1}{q_0+1}}^2 = k(q_0^4-q_0^3+q_0^2-q_0+1)^2  
\]
and by \eqref{eqn:plusplus}, 
\[
t+1 = 
k\frac{q_0^{12}(q_0^6+1)(q_0^2+1)(q_0-1)^2(q_0^2-q_0+1)}{48}, 
\]
where $k$ is some positive integer whose $2$-part is $2$. 
Now, we write $k$ in $q_0$-ary, \emph{i.e.}, $k = a_0 + a_1 q_0 + \cdots + 
a_n q_0^n$, where $a_0,a_1,\dots,a_n$ are integers such that 
$0 \leqslant a_0,a_1,\dots,a_n < q_0$. Using this expression for $k$, we 
compute $3(st+1)$ modulo powers of $q_0$ and solve for $a_0,\dots,a_n$.\footnote{The
reason for considering $3(st+1)$ instead of simply $st+1$ is so that we don't 
run across modular conditions where we have to divide by $3$. This problem
occurs because $t+1$ is an expression that is divided by $48$.} This
computation is performed using Mathematica (see \Cref{sec:solving}
\Cref{lis:gl23stplus1r5} for $q_0 > 128$ and \Cref{lis:finiteq} lines \ref{line:stplus1r5small1}--\ref{line:stplus1r5small2} for $q_0 \leqslant 128$). 
This data is summarised in \Cref{tab:stplus1r5small1} for $q_0 > 128$ and in \Cref{tab:stplus1r5small1cases} for $q_0 \leqslant 128$. 



From \Cref{tab:stplus1r5small1,tab:stplus1r5small1cases}, we see that 
$k \geqslant
a_{23}q_0^{23} \geqslant q_0^{23}$ for all values of $q_0$. 
From which, we obtain the bounds: 
\[
s+1 = k \paren{\frac{q+1}{q_0+1}}^2 > k \paren{\frac{q_0^5}{2q_0}}^2 \geqslant 
q_0^{23} \cdot \frac{q_0^8}{2^2} = \frac{q_0^{31}}{2^2} 
\]
and 
\[
t+1 = k\frac{q_0^{12}(q_0^6+1)(q_0^2+1)(q_0-1)^2(q_0^2-q_0+1)}{48} > k 
\frac{q_0^{24}}{48}  \geqslant 
q_0^{23} \frac{q_0^{24}}{48} \geqslant \frac{q_0^{47}}{2^6}.
\]
Applying \Cref{lem:modhigman} \ref{item:modhigman1}, we find that 
\[
|\cP| > \frac{(s+1)^2(t+1)}{2} > \frac{1}{2} \paren{\frac{q_0^{31}}{2^2}}^2
\cdot \frac{q_0^{47}}{2^6} > \frac{q_0^{109}}{2^{11}}.
\]
From \Cref{tab:bounds}, we find $|\cP| \leqslant 2(q/q_0)^{26} = 2q_0^{104}$. 
Thus, 
\[
\frac{q_0^{109}}{2^{11}} < |\cP| \leqslant 2q_0^{104},
\]
from which, we obtain $q_0^{5} < 2^{12}$, \emph{i.e.}, $q_0 < 6$. 
This is a contradiction as $q_0 \geqslant 8$. 
Therefore, $r_0 \neq 5$ and so, $r_0 = 7$. 

\begin{longtable}{lll}
\caption{$st+1$ modulo powers of $q_0$ for $G_{\alpha} \cong \LargeRee{q_0}$
where $q = q_0^5$ and $G_L \cong C_{q+1}^2 : \GLn{2}{3}$.\label{tab:stplus1r5small1}}\\
\toprule
\thead{$c$} & \thead{$3(st+1) \pmod {q_0^c}$} & \thead{Solution} \\
\midrule
$1$ & $3(2-a_0)$ & $a_0 = 2$ \\
\midrule
$2$ & $3(4-a_1)q_0$ & $a_1 = 4$ \\
\midrule
$3$ & $3(2-a_2)q_0^2$ & $a_2 = 2$ \\
\midrule
$4$ & $3(-a_3)q_0^3$ & $a_3 = 0$ \\
\midrule
$5$ & $3(-a_4)q_0^4$ & $a_4 = 0$ \\
\midrule
$6$ & $3(-4-a_5)q_0^5$ & $a_5 = q_0-4$ \\
\midrule
$7$ & $3(-9-a_6)q_0^6$ & $a_6 = q_0-9$ \\
\midrule
$8$ & $3(-5-a_7)q_0^7$ & $a_7 = q_0-5$ \\
\midrule
$9$ & $3(-1-a_8)q_0^8$ & $a_8 = q_0-1$ \\
\midrule
$10$ & $3(-1-a_9)q_0^9$ & $a_9 = q_0-1$ \\
\midrule
$11$ & $3(5-a_{10})q_0^{10}$ & $a_{10} = 5$ \\
\midrule
$12$ & $\ds\paren{36-3a_{11}-\frac{431q_0}{8}}q_0^{11}$ & 
$a_{11} = 3q_0/8+12$ \\
\midrule
$13$ & $\ds\paren{17-3a_{12}-\frac{271q_0}{8}}q_0^{12}$ & 
$a_{12} = (q/8+17)/3$ \\
\midrule
$14$ & $\ds \paren{-3a_{13} - \frac{q_0}{8}}q_0^{13}$ & $a_{13} = 5q/8$ \\
\midrule
$15$ & $\ds\paren{-2-3a_{14}}q_0^{14}$ & $a_{14} = (q-2)/3$ \\
\midrule
$16$ & $\ds\paren{-25-3a_{15}}q_0^{15}$ & $a_{15} = (2q-25)/3$ \\
\midrule
$17$ & $\ds\paren{-50-3a_{16}+\frac{301q_0}{4}}q_0^{16}$ & 
$a_{16} = (q/4-50)/3$ \\
\midrule
$18$ & $\ds\paren{-25-3a_{17}+\frac{397q_0}{8}}q_0^{17}$ & 
$a_{17} = (13q/8-25)/3$ \\
\midrule
$19$ & $\ds\paren{-2-3a_{18}+\frac{37q_0}{8}}q_0^{18}$ & 
$a_{18} = (5q/8-2)/3$ \\
\midrule
$20$ & $\ds\paren{-3a_{19}+\frac{119q_0}{4}}q_0^{19}$ & $a_{19} = q/4$ \\
\midrule
$21$ & $\ds\paren{29-3a_{20}+\frac{17q_0}{8}}q_0^{20}$ & 
$a_{20} = (q/8+29)/3$ \\
\midrule
$22$ & $\ds\paren{60-3a_{21}-\frac{701q_0}{8}}q_0^{21}$ & 
$a_{21} = q/8+20$ \\
\midrule
$23$ & $\ds\paren{32-3a_{22}-\frac{127q_0}{2}}q_0^{22}$ & 
$a_{22} = (q/2+32)/3$ \\
\midrule
$24$ & $\ds\paren{-3a_{23}-\frac{223q_0}{128}}q_0^{23}$ & 
$a_{23} = 11q_0/128$ \\
\bottomrule
\end{longtable}

\pagebreak

\begin{longtable}{ll}
\caption{$a_{23}$ for $8 \leqslant q_0 \leqslant 128$.\label{tab:stplus1r5small1cases}}\\
\toprule
\thead{$q_0$} & \thead{$a_{23}$} \\
\midrule
8 & 2\\
\midrule
32 & 3\\
\midrule
128 & 11\\
\bottomrule
\end{longtable}

\underline{Finally, suppose $r_0 = 7$.} Then by \eqref{eqn:plus}, 
\[
s+1 = k\paren{\frac{q+1}{q_0+1}}^2 = k(q_0^6-q_0^5+q_0^4-q_0^3+q_0^2-q_0+1)^2  
\]
and by \eqref{eqn:plusplus}, 
\[
t+1 = 
k\frac{q_0^{12}(q_0^6+1)(q_0^2+1)(q_0-1)^2(q_0^2-q_0+1)}{48}, 
\]
where $k$ is some positive integer whose $2$-part is $2$. 
Now, we write $k = a_0 + a_1 q_0 + \cdots + a_{n} q_0^{n}$, where
$0 \leqslant a_0,\dots,a_{n} < q_0$. Using this expression for $k$, we 
compute $27(st+1)$ modulo powers of $q_0$ and solve for
$a_0,\dots,a_{n}$. 
This computation is performed using Mathematica (see \Cref{sec:solving}
\Cref{lis:gl23stplus1r7} for $q_0 > 2048$ and \Cref{lis:finiteq} lines 
\ref{line:stplus1r7small1}--\ref{line:stplus1r7small2} for $q_0 \leqslant 2048$). 
This data is summarised in \Cref{tab:stplus1r7small1} for 
$q_0 > 2048$ and in \Cref{tab:stplus1r7small1cases} for $q_0 \leqslant 2048$. 

{\afterpage
{\small
\begin{longtable}{lll}	
\caption{$st+1$ modulo powers of $q_0$ for $G_{\alpha} \cong \LargeRee{q_0}$
where $q = q_0^7$ and $G_L \cong C_{q+1}^2 : \GLn{2}{3}$.
\label{tab:stplus1r7small1}} \\
\toprule
\thead{$c$} & \thead{$3(st+1) \pmod {q_0^c}$} & \thead{Solution} \\
\midrule
$1$ & $27(2-a_0)$ & $a_0 = 2$ \\
\midrule
$2$ & $27(4-a_1)q_0$ & $a_1 = 4$ \\
\midrule
$3$ & $27(2-a_2)q_0^2$ & $a_2 = 2$ \\
\midrule
$4$ & $27(-a_3)q_0^3$ & $a_3 = 0$ \\
\midrule
$5$ & $27(-a_4)q_0^4$ & $a_4 = 0$ \\
\midrule
$6$ & $27(-a_5)q_0^5$ & $a_5 = 0$ \\
\midrule
$7$ & $27(-a_6)q_0^6$ & $a_6 = 0$ \\
\midrule
$8$ & $27(-4-a_7)q_0^7$ & $a_7 = q_0-4$ \\
\midrule
$9$ & $27(-9-a_8)q_0^8$ & $a_8 = q_0-9$ \\
\midrule
$10$ & $27(-5-a_9)q_0^9$ & $a_9 = q_0-5$ \\
\midrule
$11$ & $27(-1-a_{10})q_0^{10}$ & $a_{10} = q_0-1$ \\
\midrule
$12$ & $\ds 3\paren{-3-3a_{11}+\frac{49q_0}{8}}q_0^{11}$ & 
$a_{11} = 3q_0/8-1$ \\
\midrule
$13$ & $\ds 3 \paren{-1-3a_{12}+\frac{17q_0}{8}}q_0^{12}$ & 
$a_{12} = (q_0/8-1)/3$ \\
\midrule
$14$ & $\ds 3\paren{-3a_{13}+\frac{143q_0}{8}}q_0^{13}$ & $a_{13} = 5q_0/8$ \\
\midrule
$15$ & $3(16-3a_{14})q_0^{14}$ & $a_{14} = (q_0+16)/3$ \\
\midrule
$16$ & $3(35-3a_{15})q_0^{15}$ & $a_{15} = (2q_0+35)/3$ \\
\midrule
$17$ & $\ds 3 \paren{16-3a_{16}-\frac{129q_0}{4}}q_0^{16}$ & 
$a_{16} = (7q_0/4+16)/3$ \\
\midrule
$18$ & $\ds 3\paren{-2-3a_{17}+\frac{33q_0}{8}}q_0^{17}$ & 
$a_{17} = (17q_0/8-2)/3$ \\
\midrule
$19$ & $\ds 3\paren{-2-3a_{18}+\frac{39q_0}{8}}q_0^{18}$ & 
$a_{18} = (5q_0/8-2)/3$ \\
\midrule
$20$ & $\ds 3\paren{-1-3a_{19}+\frac{5q_0}{4}}q_0^{19}$ & 
$a_{19} = (5q_0/4-1)/3$ \\
\midrule
$\vdots$ & $\vdots$ & $\vdots$ \\
\midrule
$41$ & $\ds \paren{-4-b_{39}-27a_{40}+\frac{55349q_0}{512}}q_0^{40}$ & 
$a_{40} = ((53q + 512b_{40}q)/512 - 4 - b_{39})/27$, \\
& & where $0 \leqslant b_{40} < 27$. \\ 
\bottomrule
\end{longtable}}
}
\begin{longtable}{ll}
\caption{$a_{40}$ for $8 \leqslant q_0 \leqslant 2048$.
\label{tab:stplus1r7small1cases}} \\
\toprule
\thead{$q_0$} & \thead{$a_{40}$} \\
\midrule
8 & 5\\
\midrule
32 & 16\\
\midrule
128 & 120\\
\midrule
512 & 504\\
\midrule
2048 & 993\\
\bottomrule
\end{longtable}

In \Cref{tab:stplus1r7small1}, note that $a_{40} \neq 0$ since 
$b_{39},b_{40} \in \{0,\dots,26\}$. Thus, from \Cref{tab:stplus1r7small1,tab:stplus1r7small1cases}, 
we see that $k \geqslant
a_{40}q_0^{40} \geqslant q_0^{40}$ for all values of $q_0$. 
From which, we obtain the bounds: 
\[
s+1 = k \paren{\frac{q+1}{q_0+1}}^2 > k \paren{\frac{q_0^7}{2q_0}}^2 \geqslant 
q_0^{40} \cdot \frac{q_0^8}{2^2} = \frac{q_0^{48}}{2^2} 
\]
and 
\[
t+1 = k\frac{q_0^{12}(q_0^6+1)(q_0^2+1)(q_0-1)^2(q_0^2-q_0+1)}{48} > k 
\frac{q_0^{24}}{48}  \geqslant 
q_0^{40} \frac{q_0^{24}}{48} \geqslant \frac{q_0^{64}}{2^6}.
\]
Applying \Cref{lem:modhigman} \ref{item:modhigman1}, we find that 
\[
|\cP| > \frac{(s+1)^2(t+1)}{2} > \frac{1}{2} \paren{\frac{q_0^{48}}{2^2}}^2
\cdot \frac{q_0^{64}}{2^6} > \frac{q_0^{160}}{2^{11}}.
\]
From \Cref{tab:bounds}, we find that 
$|\cP| \leqslant 2(q/q_0)^{26} = 2q_0^{156}$. Thus, 
\[
\frac{q_0^{160}}{2^{11}} < |\cP| \leqslant 2q_0^{156},
\]
from which, we obtain $q_0^4 < 2^{12}$, \emph{i.e.}, 
$q_0 < 2^{12/4} = 2^3 = 8$. 
This is a contradiction as $q_0 \geqslant 8$. Therefore, $G_L \ncong 
C_{q+1}^2 : \GLn{2}{3}$.
\end{proof}

Now we consider the case where $r_0 = 3$.
\begin{lemma}
Suppose $G_{\alpha} \cong \LargeRee{q_0}$ where $q_0 = 2^{n_0} > 2$ and $q =
q_0^{3}$. Then $G_L$ cannot be isomorphic to 
$C_{q+1}^2 : \GLn{2}{3}$. 
\end{lemma}
\begin{proof}
Suppose for a contradiction that $G_L \cong C_{q+1}^2 : \GLn{2}{3}$. 
Then we have 
\[
|\cP| = 
\frac{q^{12}(q^6+1)(q^4-1)(q^3+1)(q-1)}{q_0^{12}(q_0^6+1)(q_0^4-1)(q_0^3+1) 
(q_0-1)} \ \textrm{and} \ |\cL| = \frac{1}{48} q^{12}
(q^6+1)(q^2+1)(q-1)^2(q^2-q+1).
\]
Thus, 
\[
\frac{s+1}{t+1} = \frac{48(q+1)^2}{q_0^{12}(q_0^6+1)(q_0^4-1)(q_0^3+1)(q_0-1)}
= \frac{48(q_0^2-q_0+1)}{q_0^{12}(q_0^6+1)(q_0^2+1)(q_0-1)^2}.
\] 
Cross multiplying, we obtain
\[
\frac{q_0^{12}}{16} (q_0^6+1)(q_0^2+1)(q_0-1)^2 (s+1) = 3(q_0^2-q_0+1) (t+1). 
\]
By \Cref{lem:coprime} ($i=10$), 
we find that $3(q_0^2-q_0+1)$ divides $s+1$ and 
$q_0^{12}(q_0^6+1)(q_0^2+1)(q_0-1)^2/16$ divides $t+1$. Hence, 
\[
s+1 = 3k(q_0^2-q_0+1) \ \textrm{and} \ t+1 = 
k\frac{q_0^{12}}{16}(q_0^6+1)(q_0^2+1)(q_0-1)^2, 
\]
where $k$ is some positive integer. 
As before, we deduce that $k$ is even. For if $k$ were to be odd, then $s$ is
even, which means that $|\cP| = (s+1)(st+1)$ is odd. This is not the case as
$q/q_0$ divides $|\cP|$. 
Therefore, $k \equiv 0$ or $2 \pmod 4$. \\ 

Suppose first that $k \equiv 0
\pmod 4$, then $s \equiv -1 \pmod 4$ and $t \equiv -1 \pmod 4$. Whence, $st+1
\equiv 2 \pmod 4$. Thus, the $2$-part of $st+1$ is $2$ and consequently, the
$2$-part of $s+1$ is $(q/q_0)^{12}/2$. 
Since $s+1 = 3k(q_0^2-q_0+1)$, the $2$-part of $k$ is $(q/q_0)^{12}/2$.
From this, we obtain the following bounds: 
\[
s+1 = 3k (q_0^2-q_0+1) > 3k \frac{q_0^2}{2} \geqslant 
3\frac{q^{12}}{2q_0^{12}} \cdot \frac{q_0^2}{2} = \frac{3q_0^{26}}{4} >
\frac{q_0^{26}}{2} 
\]
and 
\[
t+1 = k \frac{q_0^{12}}{16} (q_0^6+1)(q_0^2+1)(q_0-1)^2 > k
\frac{q_0^{12}}{16} \cdot
\frac{q_0^{10}}{2} \geqslant \frac{q^{12}}{2q_0^{12}} \cdot
\frac{q_0^{12}}{16} \cdot \frac{q_0^{10}}{2} = \frac{q_0^{46}}{2^{6}}.
\]

Applying \Cref{lem:modhigman} \ref{item:modhigman1}, we find that 
\[
|\cP| > \frac{(s+1)^2(t+1)}{2} > \frac{1}{2} \paren{\frac{q_0^{26}}{2}}^2
\frac{q_0^{46}}{2^6} = \frac{q_0^{98}}{2^8}.
\]
However, this contradicts the upper bound from \Cref{tab:bounds}, namely that 
$|\cP| \leqslant 2(q/q_0)^{26} = 2q_0^{52}$. 
So, in fact, $k \equiv 2 \pmod 4$. \\

When $k \equiv 2 \pmod 4$, the
$2$-part of $s+1$ is $2$ and the $2$-part of $st+1$ must be
$(q/q_0)^{12}/2 = q_0^{24}/2$. 
Now, we write $k$ in $q_0$-ary, \emph{i.e.}, $k = a_0 + a_1 q_0 + \cdots + 
a_n q_0^n$, where $a_0,a_1,\dots,a_n$ are integers such that 
$0 \leqslant a_0,a_1,\dots,a_n < q_0$. Using this expression for $k$, we 
compute $st+1$ modulo powers of $q_0$ and solve for $a_0,\dots,a_n$. This
computation is performed using Mathematica (see \Cref{sec:solving}
\Cref{lis:gl23stplus1r3}) and it is summarised in \Cref{tab:stplus1r3small1}.

\afterpage{%
{\clearpage
\begin{longtable}{lll}
\caption{$st+1$ modulo powers of $q_0$ for $G_{\alpha} \cong \LargeRee{q_0}$
where $q = q_0^3$ and $G_L \cong C_{q+1}^2 : \GLn{2}{3}$. 
\label{tab:stplus1r3small1}}\\
\toprule
\thead{$c$} & \thead{$st+1 \pmod {q_0^c}$} & \thead{Solution} \\
\midrule
$1$ & $2-3a_0$ & $a_0 = 2(q_0+1)/3$ \\
\midrule
$2$ & $(-3a_1)q_0$ & $a_1 = 0$ \\
\midrule
$3$ & $(-3a_2)q_0^2$ & $a_2 = 0$ \\
\midrule
$4$ & $(-2-3a_3)q_0^3$ & $a_3 = (q_0-2)/3$ \\
\midrule
$5$ & $(-3-3a_4)q_0^4$ & $a_4 = q_0-1$ \\
\midrule
$6$ & $(-3-3a_5)q_0^5$ & $a_5 = q_0-1$ \\
\midrule
$7$ & $(-1-3a_6)q_0^6$ & $a_6 = (2q_0-1)/3$ \\
\midrule
$8$ & $(-3a_7)q_0^7$ & $a_7 = 0$ \\
\midrule
$9$ & $(-3a_8)q_0^8$ & $a_8 = 0$ \\
\midrule
$10$ & $(-2-3a_9)q_0^9$ & $a_9 = (q_0-2)/3$ \\
\midrule
$11$ & $(-3-3a_{10})q_0^{10}$ & $a_{10} = q_0-1$ \\
\bottomrule
\end{longtable}}}

Examining the last row of \Cref{tab:stplus1r3small1}, we see that $k \geqslant
(q_0-1)q_0^{10} \geqslant q_0^{11}/2$. From which, we obtain the bounds: 
\[
s+1 = 3k (q_0^2-q_0+1) > 2k \frac{q_0^2}{2} \geqslant \frac{q_0^{11}}{2} 
\cdot q_0^2 = \frac{q_0^{13}}{2} 
\]
and 
\[
t+1 = k\frac{q_0^{12}}{16}(q_0^6+1)(q_0^2+1)(q_0-1)^2 > k \frac{q_0^{12}}{16} 
\cdot \frac{q_0^{10}}{2} \geqslant \frac{q_0^{11}}{2} \cdot
\frac{q_0^{22}}{2^5} = \frac{q_0^{33}}{2^6}.
\]
Applying \Cref{lem:modhigman} \ref{item:modhigman1}, we find that 
\[
|\cP| > \frac{(s+1)^2(t+1)}{2} > \frac{1}{2} \paren{\frac{q_0^{13}}{2}}^2
\cdot \frac{q_0^{33}}{2^6} > \frac{q_0^{59}}{2^8}.
\]
However, this contradicts the upper bound from \Cref{tab:bounds}, namely that 
$|\cP| \leqslant 2(q/q_0)^{26} = 2q_0^{52}$. Therefore, $G_L \ncong C_{q+1}^2 : \GLn{2}{3}$. 
\end{proof}

\begin{lemma}
Suppose $G_{\alpha} \cong \LargeRee{q_0}$ where $q_0 = 2^{n_0} > 2$ and $q =
q_0^{r_0}$ for some prime $r_0$. Then $G_L$ cannot be isomorphic to
$C_{q+\epsilon\sqrt{2q}+1}^2 : 4S_4$, where $\epsilon \in \{1,-1\}$.  
\end{lemma}
\begin{proof}
Suppose for a contradiction that $G_L \cong 
C_{q+\epsilon\sqrt{2q}+1}^2 : 4S_4$. 
Consider the element $t_4 \in G$ as per \Cref{tab:ccls}. 
Then by \Cref{lem:largereeccls}, 
$t_4^G \cap G_{\alpha} \neq \varnothing$ and 
$t_4^G \cap G_L \neq \varnothing$. Hence, we can take 
$g \in t_4^G \cap G_{\alpha}$ and $h \in t_4^G
\cap G_L$. Since $g$ and $h$ are conjugate in $G$, there exists an $x \in G$
such that $g = h^x$. Thus, $g \in G_L^x$. Since $G_L^x = G_{L^\prime}$ where
$L^{\prime} = L^x \in \cL$, we can take $L$ to be $L^{\prime}$ and assume that
we have $g \in G_{\alpha} \cap G_L$. Let $\cP_g$ and $\cL_g$ be 
the set of points and the set of lines fixed by $g$, respectively. 
Using \Cref{tab:ccls}, we find that 
\[
|\Centraliser{G}{g}| =
q^3(q^2-1)(q^3+1), \ |\Centraliser{G_{\alpha}}{g}| = 
q_0^3(q_0^2-1)(q_0^3+1) \ \textrm{and} \ 
|\Centraliser{G_L}{g}| = 12. 
\]

Since the conjugacy class of $g$ does not split in $G_{\alpha}$ 
and $G_L$ by \Cref{lem:largereeccls}, 
we can apply \Cref{lem:cgtranss} to conclude that
$\Centraliser{G}{g}$ acts transitively on $\cP_g$ and $\cL_g$. Next, the 
number of fixed points and lines are 

\[
|\cP_g| = \frac{q^3(q^2-1)(q^3+1)}{q_0^3(q_0^2-1)(q_0^3+1)} 
\ \textrm{and} \ |\cL_g| = \frac{q^3(q^2-1)(q^3+1)}{12}. 
\]
By \Cref{lem:suzsub}, the
fixed substructure, $(\cP_g,\cL_g,\cI_g)$ is a generalised quadrangle of
order $(s^{\prime},t^{\prime})$ for some positive integers $s^{\prime}$ and
$t^{\prime}$. Therefore, 
\[
|\cP_g| = (s^{\prime}+1)(s^{\prime}t^{\prime}+1) = 
\frac{q^3(q^2-1)(q^3+1)}{q_0^3(q_0^2-1)(q_0^3+1)} \ \textrm{and} \ 
|\cL_g| = (t^{\prime}+1)(s^{\prime}t^{\prime}+1) = 
\frac{q^3(q^2-1)(q^3+1)}{12}.
\]
Hence,
\[
\frac{s^{\prime}+1}{t^{\prime}+1} = \frac{12}{q_0^3(q_0^2-1)(q_0^3+1)}.
\]
Since $q_0^3(q_0^2-1)(q_0^3+1)$ is divisible by $12$, we have 
\[
t^{\prime}+1 = \frac{q_0^3(q_0^2-1)(q_0^3+1)}{12}(s^{\prime}+1).
\]
Write $s^{\prime}+1 = k$. 
As before, we deduce that $k$ is even. For if $k$ were to be odd, then 
$s^{\prime}$ is
even, which means that $|\cP_g| = (s^{\prime}+1)(s^{\prime}t^{\prime}+1)$ is 
odd. This is not the case as
$q/q_0$ divides $|\cP_g|$. Therefore, $k \equiv 0$ or $2 \pmod 4$. \\

Suppose first that $k \equiv 0
\pmod 4$. Then $s^{\prime} \equiv -1 \pmod 4$ and $t^{\prime} \equiv 
-1 \pmod 4$. Whence, $s^{\prime}t^{\prime}+1
\equiv 2 \pmod 4$. Thus, the $2$-part of $s^{\prime}t^{\prime}+1$ is $2$ and 
consequently, the
$2$-part of $s^{\prime}+1$ is $(q/q_0)^{3}/2$. Therefore, 
\[
s^{\prime}+1 = k \geqslant \frac{1}{2} \paren{\frac{q}{q_0}}^{3} 
\]
and 
\[
t^{\prime}+1 = k\frac{q_0^3(q_0^2-1)(q_0^3+1)}{12} 
\geqslant k \frac{q_0^8}{24} \geqslant
\frac{1}{2} \paren{\frac{q}{q_0}}^{3} \cdot \frac{q_0^{8}}{24} =
\frac{1}{48} q^{3} q_0^{5}.
\]

Applying \Cref{lem:modhigman} \ref{item:modhigman1}, we find that 
\[
|\cP_g| > \frac{(s^{\prime}+1)^2(t^{\prime}+1)}{2} > 
\frac{1}{2} 
\paren{\frac{(q/q_0)^{3}}{2}}^2 \frac{q^{3}q_0^5}{48} = 
\frac{1}{384} \frac{q^{9}}{q_0} = \frac{1}{384} q q_0^7 
\paren{\frac{q}{q_0}}^{8}. 
\]
However, this contradicts the following upper bound, namely,  
\[ 
|\cP_g| = \frac{q^3(q^2-1)(q^3+1)}{q_0^3(q_0^2-1)(q_0^3+1)} 
\leqslant 2\paren{\frac{q}{q_0}}^{8}. 
\]
Thus, $k \equiv 2 \pmod 4$. \\ 

Suppose now that $k \equiv 2 \pmod 4$. Then 
the $2$-part of $s^{\prime}+1$ is $2$ and so, the $2$-part of 
$s^{\prime}t^{\prime}+1$ must be
$(q/q_0)^{3}/2$. Write $k \defi 2\ell$ where $\ell$ is odd. 
Define 
\[
a \defi \frac{q_0^3(q_0^2-1)(q_0^3+1)}{12}.
\]
Then by \Cref{lem:condsonst} \ref{item:higman2}, 
$s^{\prime}+t^{\prime}$ divides 
\[
s^{\prime}t^{\prime}(s^{\prime}+1)(t^{\prime}+1) 
= s^{\prime}t^{\prime} \cdot 2\ell \cdot 2\ell a.
\]
Since $s^{\prime}+t^{\prime} \equiv -1-1 \equiv -2 \pmod \ell$, 
it follows that 
$s^{\prime}+t^{\prime}$ is coprime to $\ell$. 
Thus, $s^{\prime}+t^{\prime}$ divides $4as^{\prime}t^{\prime}$. By
\Cref{lem:moddivcond}, we find that $s^{\prime}+t^{\prime}$ divides 
$a(s^{\prime}-t^{\prime})^2 = a((s^{\prime}+1)-(t^{\prime}+1))^2 =
a(2\ell(1-a))^2 = 4a\ell^2 (a-1)^2$. 
Since $s^{\prime}+t^{\prime}$ is coprime to $\ell$, it
follows that $s^{\prime}+t^{\prime}$ divides $4a(a-1)^2$. 
Since $s^{\prime}+t^{\prime} = (s^{\prime}+1)+(t^{\prime}+1)-2 =
2\ell+2\ell a - 2 = 2(\ell(a+1)-1)$, we see that $\ell(a+1)-1$ divides
$2a(a-1)^2$. Thus, $\ell(a+1)-1 \leqslant 2a(a-1)^2$. This implies that $\ell
a \leqslant \ell a + \ell - 1 \leqslant 2a(a-1)^2$. Hence, 
\[
\ell \leqslant 2(a-1)^2 \leqslant 2a^2 = 2
\paren{\frac{q_0^3(q_0^2-1)(q_0^3+1)}{12}}^2 
< 2\paren{\frac{q_0^8}{8}}^2 = \frac{q_0^{16}}{2^5}.
\]
From which, $s+1 = 2\ell < q_0^{16}/2^4$ and 
$t+1 = 2\ell a < q_0^{16}/2^4 \cdot q_0^{8}/2^3 = q_0^{24}/2^7$. 
This gives us an upper bound on the number of points:
\[
|\cP_g| = (s^{\prime}+1)(s^{\prime}t^{\prime}+1) < 
(s^{\prime}+1)(s^{\prime}+1)(t^{\prime}+1) 
= (s^{\prime}+1)^2(t^{\prime}+1) < 
\frac{q_0^{32}}{2^{10}} \cdot \frac{q_0^{24}}{2^7} = \frac{q_0^{56}}{2^{17}}. 
\]
But we also have 
\begin{equation}
|\cP_g| = \frac{q^3(q^2-1)(q^3+1)}{q_0^3(q_0^2-1)(q_0^3+1)} > \frac{1}{2}
\paren{\frac{q}{q_0}}^8. \label{ineqn:pm}
\end{equation}
Combining the two inequalities and solving for $q$, we find that $q^8 < 
q_0^{64}/2^{16} < q_0^{64}$. Since $q =
q_0^{r_0}$, it follows that $8r_0 < 64$, \emph{i.e.}, $r_0 \in \{3,5,7\}$. \\

We now consider the cases where $r_0 \in \{3,5,7\}$.
First, recall that we have
\[
s^{\prime}+1 = k \ \textrm{and} \ 
t^{\prime}+1 = k\frac{q_0^3(q_0^2-1)(q_0^3+1)}{12},
\]
where $k$ is some positive integer whose $2$-part is $2$. 
Moreover, the $2$-part of $st+1$ is $(q/q_0)^3/2 = q_0^{3(r_0-1)}/2$. 
Now, we write $k$ in $q_0$-ary, \emph{i.e.}, $k = a_0 + a_1 q_0 + \cdots + 
a_n q_0^n$, where $a_0,a_1,\dots,a_n$ are integers such that 
$0 \leqslant a_0,a_1,\dots,a_n < q_0$. Using this expression for $k$, we 
compute $3^5(st+1)$ modulo powers of $q_0$ and solve for $a_0,\dots,a_n$
by setting $3^5(st+1) \equiv 0$ modulo powers of $q_0$. 
This computation is performed using Mathematica (see \Cref{sec:solving}
\Cref{lis:pmstplus1r7} for $q_0 > 2048$ and \Cref{lis:finiteq} lines 
\ref{line:pmstplus1r7small1}--\ref{line:pmstplus1r7small2} for 
$q_0 \leqslant 2048$). 
This data is summarised in \Cref{tab:pmstplus1r7small1} for 
$q_0 > 2048$ and in \Cref{tab:pmstplus1r7small1cases} 
for $q_0 \leqslant 2048$. 
\afterpage{
{\small
\begin{longtable}{lll}	
\caption{$s^{\prime}t^{\prime}+1$ modulo powers of 
$q_0$ for $G_{\alpha} \cong \LargeRee{q_0}$
where $q = q_0$ and $G_L \cong C_{q+\epsilon\sqrt{2q}+1}^2 : 4S_4$.
\label{tab:pmstplus1r7small1}} \\
\toprule
\thead{$c$} & \thead{$243(s^{\prime}t^{\prime}+1) \pmod {q_0^c}$} & 
\thead{Solution} \\
\midrule
$1$ & $243(2-a_0)$ & $a_0 = 2$ \\
\midrule
$2$ & $243(-a_1)q_0$ & $a_1 = 0$ \\
\midrule
$3$ & $\ds 81\paren{-3a_2-\frac{q_0}{2}}q_0^2$ & $a_2 = q_0/2$ \\
\midrule
$4$ & $81(-2-3a_3)q_0^3$ & $a_3 = (q_0-2)/3$ \\
\midrule
$5$ & $\ds 81\paren{-1-3a_4+\frac{q_0}{2}}q_0^4$ & $a_4 = (q_0/2-1)/3$ \\
\midrule
$6$ & $\ds 81\paren{-3a_5-\frac{3q_0}{8}}q_0^5$ & $a_5 = 7q_0/8$ \\
\midrule
$7$ & $81(-3-3a_6)q_0^6$ & $a_6 = q_0-1$ \\
\midrule
$8$ & $\ds 81\paren{-3-3a_7+\frac{q_0}{4}}q_0^7$ & $a_7 = 3q_0/4-1$ \\
\midrule
$9$ & $\ds 9\paren{-18-27a_8+\frac{61q_0}{32}}q_0^8$ & 
$a_8 = (-18+q_0+29q_0/32+b_8 q_0)/27$, where $b_8 \in \{0,\dots,26\}$ \\
\midrule 
$10$ & $\ds 9\paren{-b_8-27a_9-\frac{9q_0}{8}}q_0^9$ & 
$a_9 = (q_0/8 - b_8 + b_{9}q_0)/27$, where $b_9 \in \{0,\dots,26\}$ \\ \\
\midrule
$11$ & $\ds 9\paren{1-b_9-27a_{10}-\frac{111q_0}{32}}q_0^{10}$ & 
$a_{10} = (1+17q_0/32 - b_{9} + b_{10}q_0)/27$, 
where $b_{10} \in \{0,\dots,26\}$ \\ \\
\midrule
$12$ & $\ds 9\paren{-4-b_{10}-27a_{11}+\frac{27q_0}{128}}q_0^{11}$ & 
$a_{11} = (27q_0/128 - 4 - b_{10} + b_{11}q_0)/27$, 
where $b_{11} \in \{0,\dots,26\}$ \\ \\
\midrule
$13$ & $\ds 9 \paren{-b_{11}-27a_{12}+\frac{39q_0}{32}}q_0^{12}$ & 
$a_{12} = (7q_0/32 - b_{11} + b_{12}q_0)/27$, 
where $b_{12} \in \{0,\dots,26\}$ \\ \\
\midrule
$14$ & $\ds 9\paren{1-b_{12}-27a_{13}+\frac{3q_0}{8}}q_0^{13}$ & 
$a_{13} = (3q_0/8 + 1 - b_{12} + b_{13}q_0)/27$, 
where $b_{13} \in \{0,\dots,26\}$ \\ \\
\midrule
$15$ & $\ds \paren{-9b_{13}-243a_{14}-\frac{1205q_0}{512}}q_0^{14}$ & 
$a_{14} = (331q_0/512 - 9b_{13} + b_{14}q_0)/243$, 
where $b_{14} \in \{0,\dots,242\}$ \\ \\
\midrule
$16$ & $\ds \paren{-3-b_{14}-243a_{15}-\frac{729q_0}{64}} q_0^{15}$ & 
$a_{15} = (39q_0/64 - 3 - b_{14} + b_{15}q_0)/243$, 
where $b_{15} \in \{0,\dots,242\}$ \\ \\
\bottomrule
\end{longtable}}

\pagebreak

\begin{longtable}{llll}
\caption{$a_4, a_{10}$ and $a_{15}$ for $8 \leqslant q_0 \leqslant 2048$. \label{tab:pmstplus1r7small1cases}}\\
\toprule
\thead{$q_0$} & \thead{$a_4$} & \thead{$a_{10}$} & \thead{$a_{15}$} \\
\midrule
8 & 1 & 7 & 2\\
\midrule
32 & 5 & 8 & 5\\
\midrule
128 & 21 & 59 & 64\\
\midrule
512 & 85 & 123 & 163\\
\midrule
2048 & 341 & 723 & 1968\\
\bottomrule
\end{longtable}
}

\underline{Suppose $r_0 = 3$.} Since the $2$-part of $st+1$ 
is $(q/q_0)^{3}/2 = q_0^{6}/2$, we can use the solution in \Cref{tab:pmstplus1r7small1,tab:pmstplus1r7small1cases} for 
$a_4$. Note that $a_4 \neq 0$. Thus, $k \geqslant q_0^{4}$. 
From this, we obtain the bounds: 
\[
s^{\prime}+1 = k \geqslant q_0^{4}
\]
and 
\[
t^{\prime}+1 = 
k\frac{q_0^{12}(q_0^2-1)(q_0^3+1)}{12} \geqslant q_0^{4}  \cdot
\frac{q_0^8}{24} > \frac{q_0^{12}}{2^5}.
\]
Applying \Cref{lem:modhigman} \ref{item:modhigman1}, we find that 
\[
|\cP| > \frac{(s^{\prime}+1)^2(t^{\prime}+1)}{2} > \frac{1}{2} 
\paren{q_0^{4}}^2
\cdot \frac{q_0^{12}}{2^5} > \frac{q_0^{20}}{2^6}.
\]
This contradicts the bound in \eqref{ineqn:pm}, namely, 
$|\cP| \leqslant 2(q/q_0)^{8} = 2q_0^{16}$. \\

\underline{Suppose $r_0 = 5$.} Since the $2$-part of $st+1$ 
is $(q/q_0)^{3}/2 = q_0^{12}/2$, we can use the solution in \Cref{tab:pmstplus1r7small1,tab:pmstplus1r7small1cases} for 
$a_{11}$. Note that $a_{11} \neq 0$. 
Thus, $k \geqslant q_0^{11}$. 
From this, we obtain the bounds: 
\[
s^{\prime}+1 = k \geqslant q_0^{11}
\]
and 
\[
t^{\prime}+1 = 
k\frac{q_0^{12}(q_0^2-1)(q_0^3+1)}{12} \geqslant q_0^{11}  \cdot
\frac{q_0^8}{24} > \frac{q_0^{19}}{2^5}.
\]
Applying \Cref{lem:modhigman} \ref{item:modhigman1}, we find that 
\[
|\cP| > \frac{(s^{\prime}+1)^2(t^{\prime}+1)}{2} > \frac{1}{2} 
\paren{q_0^{11}}^2
\cdot \frac{q_0^{19}}{2^5} > \frac{q_0^{41}}{2^6}.
\]
This contradicts the bound in \eqref{ineqn:pm}, namely, 
$|\cP| \leqslant 2(q/q_0)^{8} = 2q_0^{32}$. \\

\underline{Suppose $r_0 = 7$.} Since the $2$-part of $st+1$ 
is $(q/q_0)^{3}/2 = q_0^{18}/2$, we can use the solution in \Cref{tab:pmstplus1r7small1,tab:pmstplus1r7small1cases} for 
$a_{15}$. 
Note that $a_{15} \neq 0$ for $q_0 > 2048$ 
since $b_{14}, b_{15} \in \{0,\dots,242\}$. 
From \Cref{tab:pmstplus1r7small1cases}, $a_{15} \neq 0$ for 
$q_0 \leqslant 2048$. Thus, $k \geqslant q_0^{15}$. From this, 
we obtain the bounds: 
\[
s^{\prime}+1 = k \geqslant q_0^{15}
\]
and 
\[
t^{\prime}+1 = 
k\frac{q_0^{12}(q_0^2-1)(q_0^3+1)}{12} \geqslant q_0^{15}  \cdot
\frac{q_0^8}{24} > \frac{q_0^{23}}{2^5}.
\]
Applying \Cref{lem:modhigman} \ref{item:modhigman1}, we find that 
\[
|\cP| > \frac{(s^{\prime}+1)^2(t^{\prime}+1)}{2} > \frac{1}{2} 
\paren{q_0^{15}}^2
\cdot \frac{q_0^{23}}{2^5} > \frac{q_0^{53}}{2^6}.
\]
This contradicts the bound in \eqref{ineqn:pm}, namely, 
$|\cP| \leqslant 2(q/q_0)^{8} = 2q_0^{48}$. 
This concludes all the cases for $r_0$ and so, $G_L \ncong 
C_{q+\epsilon\sqrt{2q}+1}^2 : 4S_4$.
\end{proof}

\begin{lemma}
\label{lem:q0last}
Suppose $G_{\alpha} \cong \LargeRee{q_0}$ where $q_0 = 2^{n_0} > 2$ and $q =
q_0^{r_0}$ for some prime $r_0$. Then $G_L$ cannot be isomorphic to
$C_{q^2+q+1+\epsilon\sqrt{2q}(q+1)} : C_{12}$ where $\epsilon \in \{1,-1\}$. 
\end{lemma}

\begin{proof}
Suppose for a contradiction that $G_L \cong 
C_{q^2+q+1+\epsilon\sqrt{2q}(q+1)} : C_{12}$. 
Consider the element $u_2 \in G$ as per \Cref{tab:ccls}. 
Then by \Cref{lem:largereeccls}, 
$u_2^G \cap G_{\alpha} \neq \varnothing$ and 
$u_2^G \cap G_L \neq \varnothing$. Hence, we can take 
$g \in u_2^G \cap G_{\alpha}$ and $h \in u_2^G
\cap G_L$. Since $g$ and $h$ are conjugate in $G$, there exists an $x \in G$
such that $g = h^x$. Thus, $g \in G_L^x$. Since $G_L^x = G_{L^\prime}$ where
$L^{\prime} = L^x \in \cL$, we can take $L$ to be $L^{\prime}$ and assume that
we have $g \in G_{\alpha} \cap G_L$. Let $\cP_g$ and $\cL_g$ be 
the set of points and the set of lines fixed by $g$, respectively. 
Using \Cref{tab:ccls}, we find that 
\[
|\Centraliser{G}{g}| =
q^3(q^2-1)(q^3+1), \ |\Centraliser{G_{\alpha}}{g}| = 
q_0^3(q_0^2-1)(q_0^3+1) \ \textrm{and} \ 
|\Centraliser{G_L}{g}| = 12. 
\]

Since the conjugacy class of $g$ does not split in $G_{\alpha}$ 
and $G_L$ by \Cref{lem:largereeccls}, 
we can apply \Cref{lem:cgtranss} to conclude that
$\Centraliser{G}{g}$ acts transitively on $\cP_g$ and $\cL_g$. Next, the 
number of fixed points and lines are  

\[
|\cP_g| = \frac{q^{10}(q^2-1)}{q_0^{10}(q_0^2-1)} \textrm{and} \ |\cL_g| = 
\frac{q^{10}(q^2-1)}{12}. 
\]
By \Cref{lem:suzsub}, the
fixed substructure, $(\cP_g,\cL_g,\cI_g)$ is a generalised quadrangle of
order $(s^{\prime},t^{\prime})$ for some positive integers $s^{\prime}$ and
$t^{\prime}$. Therefore, 
\[
|\cP_g| = (s^{\prime}+1)(s^{\prime}t^{\prime}+1) = 
\frac{q^{10}(q^2-1)}{q_0^{10}(q_0^2-1)}\ \textrm{and} \ 
|\cL_g| = (t^{\prime}+1)(s^{\prime}t^{\prime}+1) = 
\frac{q^{10}(q^2-1)}{12}.
\]
Hence,
\[
\frac{s^{\prime}+1}{t^{\prime}+1} = \frac{12}{q_0^{10}(q_0^2-1)}.
\]
Since $q_0^{10}{(q_0^2-1)}$ is divisible by $12$, we have 
\[
t^{\prime}+1 = \frac{q_0^{10}(q_0^2-1)}{12}(s+1).
\]
Write $s^{\prime}+1 = k$. 
As before, we deduce that $k$ is even. For if $k$ were to be odd, then $s$ is
even, which means that $|\cP_g| = (s^{\prime}+1)(s^{\prime}t^{\prime}+1)$ is 
odd. This is not the case as
$q/q_0$ divides $|\cP_g|$. Therefore, $k \equiv 0$ or $2 \pmod 4$. \\ 

Suppose first that $k \equiv 0
\pmod 4$, then $s^{\prime} \equiv -1 \pmod 4$ and $t^{\prime} \equiv 
-1 \pmod 4$. Whence, $s^{\prime}t^{\prime}+1
\equiv 2 \pmod 4$. Thus, the $2$-part of $s^{\prime}t^{\prime}+1$ is $2$ and 
consequently, the
$2$-part of $s^{\prime}+1$ is $(q/q_0)^{10}/2$. Therefore, 
\[
s^{\prime}+1 = k \geqslant \frac{1}{2} \paren{\frac{q}{q_0}}^{10} 
\]
and 
\[
t^{\prime}+1 = k\frac{q_0^{10}(q_0^2-1)}{12} 
\geqslant k \frac{q_0^{12}}{24} \geqslant
\frac{1}{2} \paren{\frac{q}{q_0}}^{10} \cdot \frac{q_0^{12}}{24} =
\frac{1}{48} q^{10} q_0^{2}.
\]

Applying \Cref{lem:modhigman} \ref{item:modhigman1}, we find that 
\[
|\cP_g| > \frac{(s^{\prime}+1)^2(t^{\prime}+1)}{2} > 
\frac{1}{2} 
\paren{\frac{(q/q_0)^{10}}{2}}^2 \frac{q^{10}q_0^2}{48} = 
\frac{1}{384} \frac{q^{30}}{q_0^{18}} = \frac{1}{384} q^{12} 
\paren{\frac{q}{q_0}}^{18}. 
\]
However, this contradicts the following upper bound, namely that 
\[ 
|\cP_g| = \frac{q^{10}(q^2-1)}{q_0^{10}(q_0^2-1)} \leqslant
2\paren{\frac{q}{q_0}}^{12}. 
\]
Thus, $k \equiv 2 \pmod 4$. \\ 

When $k \equiv 2 \pmod 4$, 
the $2$-part of $s^{\prime}+1$ is $2$ and so, the $2$-part of 
$s^{\prime}t^{\prime}+1$ must be
$(q/q_0)^{10}/2$. Hence, $st+1 \geqslant (q/q_0)^{10}/2$. Consequently, 
\[
k^2 \frac{q_0^{12}}{12} \geqslant k^2 \frac{q_0^{10}(q_0^2-1)}{12} = 
(s+1)(t+1) > s^{\prime}t^{\prime}+1 
\geqslant \frac{1}{2} \paren{\frac{q}{q_0}}^{10}.
\]
Therefore,
\[
k > \frac{\sqrt{6} q^{5}}{q_0^{11}}.
\]
Using \Cref{lem:modhigman} \ref{item:modhigman1}, 
we obtain the following lower bound for $|\cP_g|$, namely,
\[
|\cP_g| > \frac{(s^{\prime}+1)^2(t^{\prime}+1)}{2} 
= \frac{k^3 q_0^{10}(q_0^2-1)}{24} 
> \frac{6\sqrt{6} q^{15}}{q_0^{33}} \cdot \frac{q_0^{12}}{48} 
> \frac{\sqrt{6}}{8} \frac{q^{15}}{q_0^{21}} 
= \frac{\sqrt{6}}{8} q_0^{15r_0-21}.
\]
However, we also have, 
\[
|\cP_g| = \frac{q^{10}(q^2-1)}{q_0^{10}(q_0^2-1)} \leqslant 2
\paren{\frac{q}{q_0}}^{12} = 2q_0^{12r_0-12}.
\]
Hence, 
\[
\frac{\sqrt{6}}{8} q_0^{15r_0-21} < 2 q_0^{12r_0-12}.
\]
Rearranging, we find
\[
q_0^{3r_0} < \frac{16}{\sqrt{6}} q_0^{9} < q_0^{10}.
\]
Therefore, $3r_0 < 10$, and so, $r_0 = 3$. 
In which case, we have 
\[
|\cP_g| = \frac{q_0^{30}(q_0^6-1)}{q_0^{10}(q_0^2-1)} =
q_0^{20}(q_0^4+q_0^2+1) \leqslant 2q_0^{24}
\]
and 
\[
|\cL_g| = \frac{q_0^{30}(q_0^6-1)}{12} \geqslant \frac{q_0^{36}}{24}.
\]
Using \Cref{lem:modhigman} \ref{item:modhighman2}, we arrive at the following
inequality:
\[
\paren{\frac{q_0^{36}}{24}}^4 \leqslant |\cL_g|^4 < |\cP_g|^5 \leqslant (2
q_0^{24})^5.
\]
Rearranging, we have 
\[
q_0^{24} < 2^5 \cdot 24^4 < 3^{24}.
\]
This is a contradiction as $q_0 >
8$. Therefore, $G_L \ncong C_{q^2+q+1+\epsilon\sqrt{2q}(q+1)} : C_{12}$. 
\end{proof}

We have now established that $G_{\alpha}$ is not isomorphic to 
$\LargeRee{q_0}$, where $q_0 = 2^{n_0} > 2$ and $q =
q_0^{r_0}$ for some prime $r_0$. Next, we show in Lemmata 
\ref{lem:cq1first}--\ref{lem:cq1last} that $G_{\alpha}$
cannot be isomorphic to $C_{q+1}^2 : \GLn{2}{3}$. 
As before, we show this by 
assuming $G_{\alpha} \cong C_{q+1}^2 : \GLn{2}{3}$ and 
study the various possibilities for $G_L$.

\begin{lemma}
\label{lem:cq1first}
Suppose $G_{\alpha} \cong C_{q+1}^2 : \GLn{2}{3}$. 
Then $G_L$ cannot be isomorphic to $C_{q+\epsilon\sqrt{2q}+1}^2 : 4S_4$, 
where $\epsilon \in \{1,-1\}$.
\end{lemma}
\begin{proof}
Let $G_{\alpha} \cong C_{q+1}^2 : \GLn{2}{3}$. Suppose that $G_L \cong
C_{q+\epsilon\sqrt{2q}+1}^2 : 4S_4$. Then 
\[
|\cP| = (s+1)(st+1) = \frac{1}{48} q^{12}(q^6+1)(q^2+1)(q-1)^2(q^2-q+1)
\]
and 
\[
|\cL| = (t+1)(st+1) = \frac{1}{96}
q^{12}(q-\epsilon\sqrt{2q}+1)^2(q^4-q^2+1)(q^2-1)(q^3+1)(q-1).
\]
Then 
\[
\frac{s+1}{t+1} = \frac{2(q+\epsilon\sqrt{2q}+1)^2}{(q+1)^2}.
\]
Cross multiplying, we obtain
\[
(q+1)^2(s+1) = 2(q+\epsilon\sqrt{2q}+1)^2(t+1).
\]
By \Cref{lem:coprime} ($i = 6$), 
we find that $2(q+\epsilon\sqrt{2q}+1)^2$ divides $s+1$ and
$(q+1)^2$ divides $t+1$. Hence, 
\[
s+1 = 2k(q+\epsilon\sqrt{2q}+1)^2 \ \textrm{and} \ t+1 =
k(q+1)^2, 
\]
where $k$ is some positive integer. 
As before, we deduce that $k$ is even. For if $k$ were to be odd, then $t$ is
even, which means that $|\cL| = (t+1)(st+1)$ is odd. This is not the case as
$q$ divides $|\cL|$. Therefore, $k \equiv 0$ or $2 \pmod 4$. \\ 

Suppose first that $k \equiv 0
\pmod 4$, then $s \equiv -1 \pmod 4$ and $t \equiv -1 \pmod 4$. Whence, $st+1
\equiv 2 \pmod 4$. Thus, the $2$-part of $st+1$ is $2$ and consequently, the
$2$-part of $s+1$ is $q^{12}/32$.
Since $s+1 = 2k(q+\epsilon\sqrt{2q}+1)^2$, the $2$-part of $k$
is $q^{12}/64$. Therefore, 
\[
s+1 = 2k(q+\epsilon\sqrt{2q}+1)^2 > 2k \cdot \frac{q^2}{4} \geqslant 2 \cdot
\frac{q^{12}}{64} \cdot \frac{q^2}{4} = \frac{q^{14}}{2^7}
\]
and 
\[
t+1 = k(q+1)^2 > k \cdot q^2 \geqslant \frac{q^{12}}{64} \cdot q^2 =
\frac{q^{14}}{2^6}. 
\]
Applying \Cref{lem:modhigman} \ref{item:modhigman1}, we find that 
\[
|\cP| > \frac{(s+1)^2(t+1)}{2} > \frac{1}{2} \paren{\frac{q^{14}}{2^7}}^2
\frac{q^{14}}{2^6} = \frac{q^{42}}{2^{20}}.
\]
However, this contradicts the upper bound from \Cref{tab:bounds}, namely that 
$|\cP|
\leqslant q^{24}/48$. Thus, $k \equiv 2 \pmod 4$. \\

When $k \equiv 2 \pmod 4$, the
$2$-part of $s+1$ is $4$ and so, the $2$-part of $st+1$ must be
$q^{12}/64$. Now, write $k = 2\ell$ where $\ell$ is odd. 
By \Cref{lem:groupdivcond}, we find that $q+1$ divides $|\cP|-1 =
(s+1)(st+1)-1 = s^2t + s(t+1)$. Since $q+1$ divides $t+1$, we can deduce that
$q+1$ divides $s^2t$. Moreover, $t$ and $t+1$ are coprime and so, $q+1$
divides $s^2$. Similarly, by \Cref{lem:groupdivcond}, we have that 
$q+\epsilon\sqrt{2q}+1$ divides $|\cL|-1 = (t+1)(st+1) - 1 = t^2s + t(s+1)$. Since
$q+\epsilon\sqrt{2q}+1$ divides $s+1$, it follows that 
$q+\epsilon\sqrt{2q}+1$ divides
$t^2 s$ and whence, it divides $t^2$. Therefore, $s+t$ is coprime to $q+1$ and
$q+\epsilon\sqrt{2q}+1$. Indeed, if we take any prime $p$ that divides $q+1$ or
$q+\epsilon\sqrt{2q}+1$, then $s+t \equiv -1 \pmod p$. Furthermore, observe that 
$\ell$
is also coprime to $s+t$ since $s+t \equiv -2 \pmod p$ for any prime $p$
dividing $\ell$. Since $\ell$ is odd, $p$ is also odd and hence, $s+t \nequiv 0
\pmod p$. Finally, recall that $s+t$ divides
$st(s+1)(t+1) = st \cdot 
8\ell^2 (q+\epsilon\sqrt{2q}+1)^2(q+1)^2$ and based on the argument
above, we can conclude that $s+t$ divides $8st$. By \Cref{lem:moddivcond}, we
have that $s+t$ divides 
\[
8(s-t)^2 = 8(4\ell(q+\epsilon\sqrt{2q}+1)^2-2\ell(q+1)^2)^2 
= 32\ell^2(q^2+\epsilon4 
q \sqrt{2q} +\epsilon
4\sqrt{2q} + 6q+1).
\] 
Since $s+t$ is coprime to $\ell$ and the $2$-part of $s+t$ is at most $8$, it
follows that $s+t$ divides 
\[
8(q^2 +\epsilon 4q \sqrt{2q} +\epsilon 4\sqrt{2q} + 6q+1)^2.
\]
Therefore, 
\[
s+t \leqslant 8(q^2 +\epsilon 4q\sqrt{2q} 
+\epsilon 4 \sqrt{2q} + 6q + 1)^2 \leqslant 8
\cdot (2q^2)^2 = 32q^4. 
\]
Moreover, 
\[
s+t = 4\ell(q+\epsilon\sqrt{2q}+1)^2-1 + 2\ell (q+1)^2-1 \geqslant 
4\ell\paren{\frac{q}{2}}^2
+ 2\ell(q)^2 = 3\ell q^2.
\]
Hence, 
\[
2\ell q^2 < 3\ell q^2 \leqslant s+t \leqslant 32q^4,
\]
from which, it follows that $\ell \leqslant 16q^2$. Therefore, 
\[
s+1 = 4\ell(q+\epsilon\sqrt{2q}+1)^2 \leqslant 4 \ell \cdot 4q^2 \leqslant 4 \cdot 4q^2
\cdot 16q^2 = 2^{8} q^4
\]
and 
\[
t+1 = 2\ell(q+1)^2 \leqslant 2\ell \cdot 2q^2 \leqslant 2 \cdot 16 q^2 \cdot 2q^2 =
2^{6} q^4.
\]
As a consequence, 
\[
|\cP| = (s+1)(st+1) < (s+1)(s+1)(t+1) = (s+1)^2(t+1) \leqslant (2^{8}
q^4)^2(2^6 q^4) = 2^{22} q^{12}.
\]
Using the lower bound from \Cref{tab:bounds}, namely that $|\cP| \geqslant
q^{24}/96$, we obtain $q^{24}/96 \leqslant |\cP| < 2^{22} q^{12}$. Simplifying
further, we find $q^{12} < 96 \cdot 2^{22}$, \emph{i.e.}, $q < 6$. This is a
contradiction as $q \geqslant 8$. Therefore, $G_L \ncong 
C_{q+\epsilon\sqrt{2q}+1}^2 : 4S_4$.
\end{proof}

\begin{lemma}
\label{lem:cq1last}
Suppose $G_{\alpha} \cong C_{q+1}^2 : \GLn{2}{3}$. 
Then $G_L$ cannot be isomorphic to 
$C_{q^2+q+1+\epsilon\sqrt{2q}(q+1)} : C_{12}$, where $\epsilon \in \{1,-1\}$.
\end{lemma}
\begin{proof}
Suppose for a contradiction that $G_L \cong
C_{q^2+q+1+\epsilon\sqrt{2q}(q+1)} : C_{12}$. Then 
\[
|\cP| = (s+1)(st+1) = \frac{1}{48} q^{12}(q^6+1)(q^2+1)(q-1)^2(q^2-q+1)
\]
and 
\[
|\cL| = (t+1)(st+1) = \frac{1}{12}
q^{12}(q^2+1)(q^2+q+1-\epsilon\sqrt{2q}(q+1))(q^4-1)(q^3+1)(q-1).
\]
Then 
\[
\frac{s+1}{t+1} =
\frac{(q^6+1)(q^2+1)(q-1)^2(q^2-q+1)}{4(q^2+1)(q^2+q+1-\epsilon\sqrt{2q}(q+1))
(q^4-1)(q^3+1)(q-1)} = \frac{(q^2+q+1+\epsilon\sqrt{2q}(q+1))}{4(q+1)^2}.
\]
Cross multiplying, we obtain
\[
4(q+1)^2(s+1) = (q^2+q+1+\epsilon\sqrt{2q}(q+1))(t+1).
\]
By \Cref{lem:coprime} ($i = 11$), 
we find that $(q^2+q+1+\epsilon\sqrt{2q}(q+1))$ divides 
$s+1$ and $4(q+1)^2$ divides $t+1$. Hence, 
\[
s+1 = k(q^2+q+1+\epsilon\sqrt{2q}(q+1)) \ \textrm{and} \ t+1 =
4k(q+1)^2, 
\]
where $k$ is some positive integer. 
As before, we deduce that $k$ is even. For if $k$ were to be odd, then $s$ is
even, which means that $|\cP| = (s+1)(st+1)$ is odd. This is not the case as
$q$ divides $|\cP|$. Therefore, $k \equiv 0$ or $2 \pmod 4$. \\ 

Suppose first that $k \equiv 0
\pmod 4$, then $s \equiv -1 \pmod 4$ and $t \equiv -1 \pmod 4$. Whence, $st+1
\equiv 2 \pmod 4$. Thus, the $2$-part of $st+1$ is $2$ and consequently, the
$2$-part of $s+1$ is $q^{12}/32$.
Since $s+1 = k(q^2+q+1+\epsilon\sqrt{2q}(q+1))$, the $2$-part of $k$
is $q^{12}/32$. Therefore, 
\[
s+1 = k(q^2+q+1+\epsilon\sqrt{2q}(q+1)) > k\cdot \frac{q^2}{2} \geqslant
\frac{q^{12}}{32} \cdot \frac{q^2}{2} = \frac{q^{14}}{2^6}
\]
and 
\[
t+1 = 4k(q+1)^2 > 4k \cdot q^2 \geqslant 4\cdot \frac{q^{12}}{32} \cdot q^2 =
\frac{q^{14}}{2^3}. 
\]
Applying \Cref{lem:modhigman} \ref{item:modhigman1}, we find that 
\[
|\cP| > \frac{(s+1)^2(t+1)}{2} > \frac{1}{2} \paren{\frac{q^{14}}{2^6}}^2
\frac{q^{14}}{2^3} = \frac{q^{42}}{2^{15}}.
\]
However, this contradicts the upper bound from \Cref{tab:bounds}, namely that 
$|\cP|
\leqslant q^{24}/48$. Thus, $k \equiv 2 \pmod 4$. \\

When $k \equiv 2 \pmod 4$, the
$2$-part of $s+1$ is $2$ and so, the $2$-part of $st+1$ must be
$q^{12}/32$. Now, write $k = 2\ell$ where $\ell$ is odd. 
By \Cref{lem:groupdivcond}, we find that $q+1$ divides $|\cP|-1 =
(s+1)(st+1)-1 = s^2t + s(t+1)$. Since $q+1$ divides $t+1$, we can deduce that
$q+1$ divides $s^2t$. Moreover, $t$ and $t+1$ are coprime and so, $q+1$
divides $s^2$. Similarly, by \Cref{lem:groupdivcond}, we have that 
$q^2+q+1+\epsilon\sqrt{2q}(q+1)$ divides $|\cL|-1 = (t+1)(st+1) - 1 = t^2s + 
t(s+1)$. Since $q^2+q+1+\epsilon\sqrt{2q}(q+1)$ divides $s+1$, it follows that 
$q^2+q+1+\epsilon\sqrt{2q}(q+1)$ divides
$t^2 s$ and whence, it divides $t^2$. Therefore, $s+t$ is coprime to $q+1$ and
$q^2+q+1+\epsilon\sqrt{2q}(q+1)$. Indeed, if we take any prime $p$ that divides 
$q+1$ or $q^2+q+1+\epsilon\sqrt{2q}(q+1)$, then $s+t \equiv -1 \pmod p$. 
Furthermore, observe that $\ell$
is also coprime to $s+t$ since $s+t \equiv -2 \pmod p$ for any prime $p$
dividing $\ell$. Since $\ell$ is odd, $p$ is also odd and hence, $s+t \nequiv 0
\pmod p$. Finally, recall that $s+t$ divides
$st(s+1)(t+1) = st \cdot 
16\ell^2 (q^2+q+1+\epsilon\sqrt{2q}(q+1))(q+1)^2$ and based on the 
argument above, we can conclude that $s+t$ divides $16st$. By 
\Cref{lem:moddivcond}, we have that $s+t$ divides 
\begin{align*}
16(s-t)^2 ={}& 16(2\ell(q^2+q+1+\epsilon\sqrt{2q}(q+1))-8\ell(q+1)^2)^2 \\
={}& 64\ell^2(3q^2+7q+3-\epsilon\sqrt{2q}(q+1))^2.
\end{align*}
Since $s+t$ is coprime to $\ell$ and the $2$-part of $s+t$ is at most $16$, it
follows that $s+t$ divides 
\[
16(3q^2+7q+3-\epsilon\sqrt{2q}(q+1))^2.
\]
Therefore, 
\[
s+t \leqslant 16(3q^2+7q+3-\epsilon\sqrt{2q}(q+1))^2 \leqslant 16\cdot (8q^2)^2 =
2^{10} q^4.
\]
Moreover, 
\[
s+t = 2\ell(q^2+q+1+\epsilon\sqrt{2q}(q+1))-1 + 8\ell(q+1)^2-1 
\geqslant 8\ell q^2.
\]
Hence, 
\[
8\ell q^2 \leqslant s+t \leqslant 2^{10} q^4,
\]
from which, it follows that $\ell \leqslant 2^{7}q^2$. Therefore, 
\[
s+1 = 2\ell(q^2+q+1+\epsilon\sqrt{2q}(q+1)) \leqslant 2\ell \cdot 2q^2 \leqslant 2 \cdot 
2^7 q^2 \cdot 2q^2 = 2^{9} q^4
\]
and 
\[
t+1 = 8\ell(q+1)^2 \leqslant 8\ell \cdot 2q^2 \leqslant 8 \cdot 2^7 q^2 \cdot 2q^2 =
2^{11} q^4.
\]
As a consequence, 
\[
|\cP| = (s+1)(st+1) < (s+1)(s+1)(t+1) = (s+1)^2(t+1) \leqslant (2^{9}
q^4)^2(2^{11} q^4) = 2^{29} q^{12}.
\]
Using the lower bound from \Cref{tab:bounds}, namely, $|\cP| \geqslant
q^{24}/96$, we obtain $q^{24}/96 \leqslant |\cP| < 2^{29} q^{12}$. Simplifying
further, we find $q^{12} < 96 \cdot 2^{29}$, \emph{i.e.}, $q < 8$. This is a
contradiction as $q \geqslant 8$. Therefore, 
$G_L \ncong C_{q^2+q+1+\epsilon\sqrt{2q}(q+1)} : C_{12}$.
\end{proof}

We have now established that $G_{\alpha}$ is not isomorphic to 
$C_{q+1}^2 : \GLn{2}{3}$. Next, we show in Lemma \ref{lem:last} 
that $G_{\alpha}$ 
cannot be isomorphic to $C_{q+\epsilon\sqrt{2q}+1}^2 : 4S_4$, 
where $\epsilon \in \{1,-1\}$. 
As before, we show this by 
assuming $G_{\alpha} \cong C_{q+\epsilon\sqrt{2q}+1}^2 : 4S_4$ and 
study the various possibilities for $G_L$. 

\begin{lemma}
\label{lem:last}
Suppose $G_{\alpha} \cong C_{q+\epsilon\sqrt{2q}+1}^2 : 4S_4$ where $\epsilon \in \{1,-1\}$. 
Then $G_L$ cannot be isomorphic to \\
$C_{q^2+q+1+\delta\sqrt{2q}(q+1)} : C_{12}$ where $\delta \in \{1,-1\}$. 
\end{lemma}
\begin{proof}
Let $G_{\alpha} \cong C_{q+\epsilon\sqrt{2q}+1}^2 : 4S_4$ and $G_L \cong
C_{q^2+q+1+\delta\sqrt{2q}(q+1)} : C_{12}$. Then 
\[
|\cP| = (s+1)(st+1) = \frac{1}{96}
q^{12}(q-\epsilon\sqrt{2q}+1)^2(q^4-q^2+1)(q^2-1)(q^3+1)(q-1)
\]
and
\[
|\cL| = (t+1)(st+1) = \frac{1}{12}
q^{12}(q^2+1)(q^2+q+1-\delta\sqrt{2q}(q+1))(q^4-1)(q^3+1)(q-1).
\]
Then 
\[
\frac{s+1}{t+1} = \frac{(q^2+q+1+\delta\sqrt{2q}(q+1))}{8(q+\epsilon\sqrt{2q}+1)^2}.
\]
Cross multiplying, we obtain
\[
8(q+\epsilon\sqrt{2q}+1)^2(s+1) = (q^2+q+1+\delta\sqrt{2q}(q+1))(t+1).
\]
By \Cref{lem:coprime} ($i =12$), 
we find that $(q^2+q+1+\delta\sqrt{2q}(q+1))$ divides 
$s+1$ and $8(q+\epsilon\sqrt{2q}+1)^2$ divides $t+1$. Hence, 
\[
s+1 = k(q^2+q+1+\delta\sqrt{2q}(q+1)) \ \textrm{and} \ t+1 =
8k(q+\epsilon\sqrt{2q}+1)^2, 
\]
where $k$ is some positive integer. 
As before, we deduce that $k$ is even. For if $k$ were to be odd, then $s$ is
even, which means that $|\cP| = (s+1)(st+1)$ is odd. This is not the case as
$q$ divides $|\cP|$. Therefore, $k \equiv 0$ or $2 \pmod 4$. \\

Suppose first that $k \equiv 0
\pmod 4$, then $s \equiv -1 \pmod 4$ and $t \equiv -1 \pmod 4$. Whence, $st+1
\equiv 2 \pmod 4$. Thus, the $2$-part of $st+1$ is $2$ and consequently, the
$2$-part of $s+1$ is $q^{12}/64$.
Since $s+1 = k(q^2+q+1+\delta\sqrt{2q}(q+1))$, the $2$-part of $k$ is 
$q^{12}/64$. Therefore, 
\[
s+1 = k(q^2+q+1+\delta\sqrt{2q}(q+1)) > k \cdot \frac{q^2}{2} \geqslant
\frac{q^{12}}{64} \cdot \frac{q^2}{2} = \frac{q^{14}}{2^7} 
\]
and 
\[
t+1 = k(q+\epsilon\sqrt{2q}+1)^2 > k \cdot \frac{q^2}{2} \geqslant 
\frac{q^{12}}{64} \cdot \frac{q^2}{2} = \frac{q^{14}}{2^7}. 
\]
Applying \Cref{lem:modhigman} \ref{item:modhigman1}, we find that 
\[
|\cP| > \frac{(s+1)^2(t+1)}{2} > \frac{1}{2} \paren{\frac{q^{14}}{2^7}}^2
\frac{q^{14}}{2^7} = \frac{q^{42}}{2^{21}}.
\]
However, this contradicts the upper bound from \Cref{tab:bounds}, 
namely that $|\cP|
\leqslant q^{18}/2$. Thus, $k \equiv 2 \pmod 4$. \\

When $k \equiv 2 \pmod 4$, the
$2$-part of $s+1$ is $2$ and so, the $2$-part of $st+1$ must be
$q^{12}/32$. Now, write $k = 2\ell$ where $\ell$ is odd. 
By \Cref{lem:groupdivcond}, we find that $q+\epsilon\sqrt{2q}+1$ divides $|\cP|-1 =
(s+1)(st+1)-1 = s^2t + s(t+1)$. Since $q+\epsilon\sqrt{2q}+1$ divides $t+1$, we can 
deduce that $q+\epsilon\sqrt{2q}+1$ divides $s^2t$. Moreover, $t$ and $t+1$ are 
coprime and so, $q+\epsilon\sqrt{2q}+1$ divides $s^2$. 
Similarly, by \Cref{lem:groupdivcond}, we have that 
$q^2+q+1+\delta\sqrt{2q}(q+1)$ divides $|\cL|-1 = (t+1)(st+1) - 1 = t^2s + 
t(s+1)$. Since $q^2+q+1+\delta\sqrt{2q}(q+1)$ divides $s+1$, it follows that 
$q^2+q+1+\delta\sqrt{2q}(q+1)$ divides
$t^2 s$ and whence, it divides $t^2$. Therefore, $s+t$ is coprime to
$q+\epsilon\sqrt{2q}+1$ and
$q^2+q+1+\delta\sqrt{2q}(q+1)$. Indeed, if we take any prime $p$ that divides
$q+\epsilon\sqrt{2q}+1$ or $q^2+q+1+\delta\sqrt{2q}(q+1)$, then 
$s+t \equiv -1 \pmod p$. Furthermore, observe that $\ell$
is also coprime to $s+t$ since $s+t \equiv -2 \pmod p$ for any prime $p$
dividing $\ell$. Since $\ell$ is odd, $p$ is also odd and hence, $s+t \nequiv 0
\pmod p$. Finally, recall that $s+t$ divides
$st(s+1)(t+1) = st \cdot 
32\ell^2 (q^2+q+1+\delta\sqrt{2q}(q+1))(q+\epsilon\sqrt{2q}+1)^2$ and based on the argument
above, we can conclude that $s+t$ divides $32st$. By \Cref{lem:moddivcond}, we
have that $s+t$ divides 
\begin{align*}
32(s-t)^2 ={}& 32(2\ell(q^2+q+1+\delta\sqrt{2q}(q+1))-16\ell(q+\epsilon\sqrt{2q}+1)^2)^2 \\
={}& 128\ell^2(7q^2+31q+(16\epsilon-\delta)q\sqrt{2q}+
(16\epsilon-\delta) \sqrt{2q}+7)^2.
\end{align*} 
Since $s+t$ is coprime to $\ell$ and the $2$-part of $s+t$ is at most $32$, it
follows that $s+t$ divides 
\[
32(7q^2+31q+(16\epsilon-\delta)q\sqrt{2q}+
(16\epsilon-\delta) \sqrt{2q}+7)^2.
\]
Therefore, 
\[
s+t \leqslant 32(7q^2+31q+(16\epsilon-\delta)q\sqrt{2q}+
(16\epsilon-\delta) \sqrt{2q}+7)^2 \leqslant 32 \cdot
(32q^2)^2  = 2^{15} q^4. 
\]
Moreover, 
\[
s+t = 2\ell(q^2+q+1+\delta\sqrt{2q}(q+1))-1 + 16\ell (q+\epsilon\sqrt{2q}+1)^2-1 
\geqslant 4 \ell q^2.
\]
Hence, 
\[
4\ell q^2 \leqslant s+t \leqslant 2^{15} q^4,
\]
from which, it follows that $\ell \leqslant 2^{13}q^2$. Therefore, 
\[
s+1 = 2\ell(q^2+q+1+\delta\sqrt{2q}(q+1)) \leqslant 2\ell \cdot 2q^2 \leqslant 2 \cdot 
2^{13} q^2 \cdot 2 q^2 = 2^{15} q^4
\]
and 
\[
t+1 = 16\ell(q+\epsilon\sqrt{2q}+1)^2 \leqslant 16\ell \cdot 4q^2 \leqslant 16 \cdot
2^{13} q^2 \cdot 4q^2 = 2^{19} q^4.
\]
As a consequence, 
\[
|\cP| = (s+1)(st+1) < (s+1)(s+1)(t+1) = (s+1)^2(t+1) \leqslant (2^{15}
q^4)^2(2^{19} q^4) = 2^{49} q^{12}.
\]
Using the lower bound from \Cref{tab:bounds}, namely that 
$|\cP| \geqslant
q^{24}/384$, we obtain $q^{24}/384 \leqslant |\cP| < 2^{49} q^{12}$. 
Simplifying further, we find $q^{12} < 384 \cdot 2^{49}$, \emph{i.e.}, 
$q < 28$. It can be verified via Mathematica (\Cref{sec:finite} 
\Cref{lis:finite} lines \ref{line:small7}--\ref{line:small8}) 
that there are no solutions for $s$ and $t$ when $q < 28$. Therefore, 
$G_L \ncong C_{q^2+q+1+\delta\sqrt{2q}(q+1)} : C_{12}$.
\end{proof}

This concludes all the cases. 
\begin{proof}[Proof of \Cref{thm:largeree}]
Since the point-line 
dual of a generalised quadrangle is a generalised quadrangle, 
\Cref{thm:largeree} follows from 
Lemmata \ref{cor:reductionf4}--\ref{lem:last}. 
\end{proof}

\subsection*{Acknowledgements}
This research was supported by the Commonwealth through an Australian Government Research Training Program Scholarship [DOI: \href{https://doi.org/10.82133/C42F-K220}{https://doi.org/10.82133/C42F-K220}].


\defbibheading{bibintoc}{\phantomsection
\addcontentsline{toc}{chapter}{References}\section*{References}
\markboth{References}{}}

\printbibliography[heading=bibintoc, title={References}]

\pagebreak
\appendix

\section{Computational Procedures}
\label{app:cp}
In this appendix, we 
include the code in Mathematica used for bounding various quantities, 
manipulating algebraic expressions and solving algebraic equations.  

\subsection{Derivative}
\label{sec:derivative}
We prove that a function $f(x) > 0$ for all $x \geqslant 1$ by showing that 
$f$ is increasing and $f(1) > 0$. 
\begin{lstlisting}[language=Mathematica,caption={Derivative}]
f[x_] := (x^(3/2) + 1)*(x^(1/2) + 1)^5 - (x + 1)^4;
g[x_] = FullSimplify[f'[x]]
Reduce[g[x] > 0, x, Reals]
f[1]
\end{lstlisting}

\subsection{Bounds}
\label{sec:bounds}
Suppose $f(q)$ and $g(q)$ are polynomials with real coefficients
having the same degree, \emph{i.e.}, $\partial f = \partial g$. 
The following Mathematica code is used to find constants $a$ and $b$ such 
that $a g(q) \leqslant f(q) \leqslant b g(q)$ for all $q \geqslant 8$. 
\begin{lstlisting}[language=Mathematica,caption={Find Lower and Upper Bounds}]
FindBounds[f_, g_, q_, c_] := N[Reduce[ForAll[q, q >= c, a*g <= f <= b*g && a > 0 && b > 0], {a, b}]];

f := 1/2*q^9*(q^6 + 1)*(q^2 + 1)*(q - 1);
g := q^18;
FindBounds[f, g, q, 8]

f := 1/2*q^8*(q^4 - q^2 + 1)*(q^3 + 1)*(q + 1);
g := q^16;
FindBounds[f, g, q, 8]

f := 1/2*q^8*(q^6 + 1)*(q^2 - q + 1);
g := q^16;
FindBounds[f, g, q, 8]

f := q^12*(q^6 + 1)*(q^4 - 1)*(q^3 + 1)*(q - 1);
g := q^26;
FindBounds[f, g, q, 8]

f := 1/48*q^12*(q^6 + 1)*(q^2 + 1)*(q - 1)^2*(q^2 - q + 1);
g := q^24;
FindBounds[f, g, q, 8]

f := 1/96*q^12*(q + u + 1)^2*(q^4 - q^2 + 1)*(q^2 - 1)*(q^3 + 1)*(q - 1);
g := q^24;
q := u^2/2;
FindBounds[f, g, u, 4]

f := 1/12*q^12*(q^2 + 1)*(q^2 + q + 1 - u*(q + 1))*(q^4 - 1)*(q^3 + 1)*(q - 1);
g := q^24;
q := u^2/2;
FindBounds[f, g, u, 4]
\end{lstlisting}

\subsection{Finite Number of Cases}
\label{sec:finite}
In this subsection, we solve the following problem using Mathematica. Given a
positive integer $n$, we wish to find integers $s,t \geqslant 2$ such that $n
= (s+1)(st+1)$ satisfying Higman's inequality: $\sqrt{s} \leqslant t 
\leqslant s^2$ and divisibility condition: $s+t \divides st(st+1)$. Letting $k
= s+1$, we can rewrite Higman's inequality as: 
\begin{equation}
k((k-1)^{3/2}+1) \leqslant (s+1)(st+1) \leqslant k((k-1)^3+1).
\label{eqn:high}
\end{equation}
We use inequality \eqref{eqn:high} in line \ref{line:5} of \Cref{lis:finite}.
\lstset{escapeinside={(*@}{@*)}}
\begin{lstlisting}[language=Mathematica,caption={Finding parameters s and t
given the number of points}\label{lis:finite}]
ClearAll[s, t, q, k, a, n]
FindGQParameters[n_] := Module[{k, s, t},
   d := Divisors[n];
   For[i = 1, i <= Length[d], i++,
    k := d[[i]];
    If[k*(1 + (k - 1))^(3/2) <= n && n <= k*(1 + (k - 1)^3), (*@\label{line:5}@*)
     s := k - 1;
     If[Mod[n/k - 1, s] == 0,
      t := (n/k - 1)/s;
      If[Mod[s*t*(s*t + 1), s + t] == 0,
       Print["(s,t)=(", s, ",", t, ")"]
       ];
      ];
     ];
    ];
   ];

(* G_{\alpha} = PGL(2,3) *) (*@\label{line:pgl231}@*)
q = 8;
(* Order of PGL(2,13) is 2184 *)
n := q^12*(q^6 + 1)*(q^4 - 1)*(q^3 + 1)*(q - 1)/2184; 
FindGQParameters[n] (*@\label{line:pgl232}@*)


(* G_{\alpha} = SUn(3,q) : C_2 or PGU(3,q) : C_2 and G_L = C_x : [...] *) (*@\label{line:small1}@*)
q := 8; 
n := 1/2*q^9*(q^6 + 1)*(q^2 + 1)*(q - 1);
FindGQParameters[n] (*@\label{line:small2}@*)

(* G_{\alpha} = C_{q+-sqrt{2q}+1}^2 : [96] and G_L = C_{q^2+q+1...} : C_{12} *) (*@\label{line:small7}@*)
q := 8;
n := 1/96*
   q^12*(q - Sqrt[2*q] + 1)^2*(q^4 - q^2 + 1)*(q^2 - 1)*(q^3 + 
     1)*(q - 1);
FindGQParameters[n]
n := 1/96*
   q^12*(q + Sqrt[2*q] + 1)^2*(q^4 - q^2 + 1)*(q^2 - 1)*(q^3 + 
     1)*(q - 1);
FindGQParameters[n] (*@\label{line:small8}@*)
\end{lstlisting}

\pagebreak
\subsection{Computing Quotients}
\label{sec:quotient}

\begin{lstlisting}[language=Mathematica,caption={Computing
Quotients}\label{lis:quotient}]
f := q^12*(q^6 + 1)*(q^2 + 1)*(q - 1)^2*(q^2 - q + 1) - 48;
g := q + 1;
(* Check -1 is a zero of f *)
f /. q -> -1 (*@\label{line:fminus1}@*)
(* Polynomial Long Division *)
PolynomialQuotientRemainder[f, g, q] (*@\label{line:f2}@*)

f := q^12*(q - e*u + 1)^2*(q^4 - q^2 + 1)*(q^2 - 1)*(q^3 + 1)*(q - 1) - 96;
g := q + e*u + 1;
q := u^2/2;

(* Find zeros of g *)
Simplify[Solve[g == 0, u], Assumptions -> e^2 == 1] (*@\label{line:f3}@*)
(* Check zeros of g are zeros of f *)
Simplify[f /. u -> -e + I, Assumptions -> e^2 == 1] (*@\label{line:f4}@*)
(* Polynomial Long Division *)
Expand[Simplify[PolynomialQuotientRemainder[f, g, u], Assumptions -> e^2 == 1]] (*@\label{line:f5}@*)
\end{lstlisting}

\subsection{Euclidean Algorithm for Polynomials}
\label{sec:eucalg}

\begin{lstlisting}[language=Mathematica,caption={Euclidean Algorithm for
Polynomials}\label{lis:eucalg}]
f := x^6 + 1;
g := x^8 + x^7 - x^5 - x^4 - x^3 + x + 1;
(* a(x)f(x)+b(x)g(y)=d(x) *)
{d, {a, b}} = PolynomialExtendedGCD[f, g, x] (*@\label{line:poly1}@*)

f := x^6 + 1;
g := x^12 + x^11 - x^9 - x^8 + x^6 - x^4 - x^3 + x + 1;
{d, {a, b}} = PolynomialExtendedGCD[f, g, x] (*@\label{line:poly2}@*)
\end{lstlisting}

\pagebreak 
\begin{lstlisting}[language=Mathematica,caption={Bounding GCD using 
Euclidean Algorithm for Polynomials}\label{lis:eucalgbound}]
F1[q_] := 1/2*q^24*(q^12 - q^6 + 1)*(q^9 + 1)*(q^3 + 1) - 1;
F2[q_] := 1/2*q^24*(q^18 + 1)*(q^6 - q^3 + 1) - 1;
G[q_] := 
  Factor[Cancel[(q^36*(q^18 + 1)*(q^12 - 1)*(q^9 + 1)*(q^3 - 
          1))/(q^12*(q^6 + 1)*(q^4 - 1)*(q^3 + 1)*(q - 1))]] - 1;

{d, {u1, v1}} = PolynomialExtendedGCD[F1[q], G[q], q]; (*@\label{line:F1G}@*)
{d2, {u2, v2}} = PolynomialExtendedGCD[F2[q], G[q], q];(*@\label{line:F2G}@*)

a1 := Denominator[Cancel[u1]]
b1 := Denominator[Cancel[v1]] (*@\label{line:F1G1}@*)
FactorInteger[a1]
FactorInteger[b1] (*@\label{line:b1}@*)
a1
b1/2

a2 := Denominator[Cancel[u2]]
b2 := Denominator[Cancel[v2]] (*@\label{line:F2G1}@*)
FactorInteger[a2]
FactorInteger[b2] (*@\label{line:b2}@*)
a2
b2/32


(* Find q such that q >= b1/2 *)
N[Log[2, b1/2]]

m := 1; (*@\label{line:F1Gfinite}@*)
q := 2^(2*m + 1);
While[m <= 100,
 Print[GCD[F1[q], G[q]]];
 m := Evaluate[m + 1];
 q := 2^(2*m + 1);]

ClearAll[m, q]

m := 1;
q := 2^(2*m + 1);
While[m <= 100,
 Print[GCD[F2[q], G[q]]];
 m := Evaluate[m + 1];
 q := 2^(2*m + 1);] (*@\label{line:F1Gfiniteend}@*)
\end{lstlisting}

\subsection{Algebraic Manipulation}
\label{sec:algman}
\begin{lstlisting}[language=Mathematica,caption={Algebraic Manipulation}]
a := q*(q^2 + 1) + q^4*(q^2 + 1);
b := q^7*(q^2 + 1) + q^10;
d := FullSimplify[Sqrt[a^2 - 4*b], q >= 8];
Simplify[Expand[(a^2 - 2*b + a*d)/(2*b)]]
\end{lstlisting}

\subsection{Solving for $k$}
\label{sec:solving}
For the following programs, we iteratively compute $st+1$ modulo powers of
$q_0$ and fill in the solutions for $a_i$. The code shown is the final version
with all the values of $a_i$ filled in.  
\begin{lstlisting}[language=Mathematica,caption={$G_{\alpha} \cong \SUn{3}{q}
: C_2$ or $\PSUn{3}{q} : C_2$ and $G_L \cong \LargeRee{q_0}$ where $q =
q_0^3$.}\label{lis:stplus1r3}]
s := k*q^3/2*(q^6 + 1)*(q^2 + 1)*(q - 1) - 1;
t := k*(q^6 - q^3 + 1)*(q^4 + q^2 + 1) - 1;
n := 10;
Array[a, n];
k := Sum[a[i]*q^(i - 1), {i, 1, n + 1}];
(* Note that a_i = a[i+1] since the array indexing starts at 1. *)
a[1] := 2;
a[2] := 0;
a[3] := q - 2;
a[4] := 0;
a[5] := 1;
a[6] := q/2 - 1;
a[7] := 0;
a[8] := 0;
a[9] := 3*q/4 - 1;
PolynomialRemainder[s*t + 1, q^9, q]
\end{lstlisting}

\begin{lstlisting}[language=Mathematica,caption={$G_{\alpha} \cong \SUn{3}{q}
: C_2$ or $\PSUn{3}{q} : C_2$ and $G_L \cong \LargeRee{q_0}$ where $q =
q_0^5$.}\label{lis:stplus1r5}]
s := k*(q - 1)*(q^2 + 1)*(q^6 + 1) - 1;
t := 2*k*q^3*(q^4 - q^3 + q^2 - q + 1)^2*(q^4 + q^3 + q^2 + q + 1)*(q^8 + q^7 - q^5 - q^4 - q^3 + q + 1) - 1;
n := 20;
Array[a, n];
k := Sum[a[i]*q^(i - 1), {i, 1, n + 1}];
a[1] := q - 2;
a[2] := q - 3;
a[3] := q - 1;
a[4] := 3;
a[5] := 6;
a[6] := 6;
a[7] := q - 18;
a[8] := q - 63;
a[9] := q - 107;
a[10] := 65;
a[11] := 562;
a[12] := 1350;
a[13] := 528;
a[14] := q - 4488;
a[15] := q - 16011;
a[16] := q - 18351;
a[17] := 26175;
a[18] := 174776;
PolynomialRemainder[s*t + 1, q^18, q]
\end{lstlisting}

\pagebreak 
\begin{lstlisting}[language=Mathematica,caption={$G_{\alpha} \cong \SUn{3}{q}
: C_2$ or $\PSUn{3}{q} : C_2$ and $G_L \cong \LargeRee{q_0}$ where $q =
q_0^5$.}\label{lis:stplus1r7}]
s := k*(q - 1)*(q^2 + 1)*(q^6 + 1)/e - 1;
t := 2*k*q^9*(q^6 - q^5 + q^4 - q^3 + q^2 - q + 1)^2*(q^6 + q^5 + q^4 + q^3 + q^2 + q + 1)*(q^12 + q^11 - q^9 - q^8 + q^6 - q^4 - q^3 + q + 1)/e - 1;
n := 28;
Array[a, n];
ClearAll[a];
k := Sum[a[i]*q^(i - 1), {i, 1, n + 1}];
a[1] := q - 2*e;
a[2] := q - 2*e - 1;
a[3] := q - 1;
a[4] := q - 1;
a[5] := q - 2*e - 1;
a[6] := q - 2*e - 1;
a[7] := 2*e - 1;
a[8] := 2*e;
a[9] := q - 2*e;
a[10] := 2*e - 1;
a[11] := 10*e;
a[12] := 10*e;
a[13] := 0;
a[14] := 4*e;
a[15] := 18*e;
a[16] := 14*e;
a[17] := q - 4*e;
a[18] := q - 1;
a[19] := q - 1;
a[20] := q - 52*e - 1;
a[21] := q - 128*e - 1;
a[22] := q - 116*e - 1;
a[23] := q - 100*e - 1;
a[24] := q - 204*e - 1;
a[25] := q - 350*e - 1;
PolynomialRemainder[e*(s*t + 1), q^25, q]
\end{lstlisting}

\begin{lstlisting}[language=Mathematica,caption={Solving for $k$ for a
finite set of $q$.}\label{lis:finiteq}]
solveai[s0_, t0_, q_, n_] := Module[{a},
   Array[a, n];
   k := Sum[a[i]*q^(i - 1), {i, 1, n + 1}];
   s := k*s0 - 1;
   t := k*t0 - 1;
   Do[
    sol := 
     Solve[Expand[PolynomialMod[s*t + 1, q^i]/q^(i - 1)] == 0, 
      Modulus -> q];
    a[i] := Evaluate[a[i] /. sol[[1]]];
    (*Print["a[",i-1,"]:=",a[i]]*), {i, n}
    ];
   Print["a_", n - 1, ":=", a[n]];
   ];

(* G_{\alpha} = SU(3,q) : C_2 or PGU(3,q) : C_2 and G_L \cong 2F_4(q_0) where 
q = q_0^5. *) (*@\label{line:stplus1r51}@*)
q := 8;
Print["q=", q];
solveai[(q - 1)*(q^2 + 1)*(q^6 + 1), 2*q^3*(q^4 - q^3 + q^2 - q + 1)^2*(q^4 + q^3 + q^2 + q + 1)*(q^8 + q^7 - q^5 - q^4 - q^3 + q + 1), 8, 21]
While[q <= 174776,
 Print["q=", q];
 n := 18;
 s0 := (q - 1)*(q^2 + 1)*(q^6 + 1);
 t0 := 2*q^3*(q^4 - q^3 + q^2 - q + 1)^2*(q^4 + q^3 + q^2 + q + 1)*(q^8 + q^7 - q^5 - q^4 - q^3 + q + 1);
 solveai[s0, t0, q, n];
 q := Evaluate[q*4];] (*@\label{line:stplus1r52}@*)


(* G_{\alpha} = SU(3,q) : C_2 or PGU(3,q) : C_2 and G_L \cong 2F_4(q_0) where q = q_0^7. *) (*@\label{line:stplus1r71}@*)
q := 8;
While[q <= 350*7,
 Print["q=",q];
 n := 25;
 If[Mod[q, 7] == 1, e = 7, e = 1];
 s0 := (q - 1)*(q^2 + 1)*(q^6 + 1)/e;
 t0 := 2*q^9*(q^6 - q^5 + q^4 - q^3 + q^2 - q + 1)^2*(q^6 + q^5 + q^4 + q^3 + q^2 + q + 1)*(q^12 + q^11 - q^9 - q^8 + q^6 - q^4 - q^3 + q + 1)/e;
 solveai[s0, t0, q, n];
 q := Evaluate[q*4];] (*@\label{line:stplus1r72}@*)

(* G_{\alpha} = 2F_4(q_0) where q = q_0^5 and G_L = C_{q+1}^2 : GL(2,3) *) (*@\label{line:stplus1r5small1}@*)
q := 8;
While[q <= 129,
 Print["q=", q];
 n := 24;
 s0 := (q^4 - q^3 + q^2 - q + 1)^2;
 t0 := q^12*(q^6 + 1)*(q^2 + 1)*(q - 1)^2*(q^2 - q + 1)/48;
 solveai[s0, t0, q, n];
 q := Evaluate[q*4];] (*@\label{line:stplus1r5small2}@*)
 
(* G_{\alpha} = 2F_4(q_0) where q = q_0^7 and G_L = C_{q+1}^2 : GL(2,3) *) (*@\label{line:stplus1r7small1}@*)
q := 8;
While[q <= 2048,
 Print["q=", q];
 n := 41;
 s0 := (q^6 - q^5 + q^4 - q^3 + q^2 - q + 1)^2;
 t0 := q^12*(q^6 + 1)*(q^2 + 1)*(q - 1)^2*(q^2 - q + 1)/48;
 solveai[s0, t0, q, n];
 q := Evaluate[q*4];] (*@\label{line:stplus1r7small2}@*)
 
 (* G_{\alpha} = 2F_4(q_0) where q = q_0^r_0 and G_L = C_{q\pm\sqrt{2q}+1}^2 : [96] *) (*@\label{line:pmstplus1r7small1}@*)
 q := 8;
While[q <= 2048,
 Print["q=", q];
 (* n := 5;
 n := 11; *)
 n := 16;
 s0 := 1;
 t0 := q^3*(q^2 - 1)*(q^3 + 1)/12;
 solveai[s0, t0, q, n];
 q := Evaluate[q*4];] (*@\label{line:pmstplus1r7small2}@*)
\end{lstlisting}

\pagebreak 
\begin{lstlisting}[language=Mathematica,caption={$G_{\alpha} \cong
\LargeRee{q_0}$ where $q = q_0^3$ and $G_L \cong C_{q+1}^2 : 
\GLn{2}{3}$.}\label{lis:gl23stplus1r3}]
s := 3*k*(q^2 - q + 1) - 1;
t := k*q^12*(q^6 + 1)*(q^2 + 1)*(q - 1)^2/16 - 1;
n := 15;
Array[a, n];

k := Sum[a[i]*q^(i - 1), {i, 1, n + 1}];
a[1] := 2*(q + 1)/3;
a[2] := 0;
a[3] := 0;
a[4] := (q - 2)/3;
a[5] := q - 1;
a[6] := q - 1;
a[7] := (2*q - 1)/3;
a[8] := 0;
a[9] := 0;
a[10] := (q - 2)/3;
a[11] := q - 1;
PolynomialRemainder[(s*t + 1), q^11, q]
\end{lstlisting}

\begin{lstlisting}[language=Mathematica,caption={$G_{\alpha} \cong
\LargeRee{q_0}$ where $q = q_0^5$ and $G_L \cong C_{q+1}^2 : 
\GLn{2}{3}$.}\label{lis:gl23stplus1r5}]
s := k*(q^4 - q^3 + q^2 - q + 1)^2 - 1;
t := k*q^12*(q^6 + 1)*(q^2 + 1)*(q - 1)^2*(q^2 - q + 1)/48 - 1;
n := 28;
Array[a, n];

k := Sum[a[i]*q^(i - 1), {i, 1, n + 1}];
a[1] := 2;
a[2] := 4;
a[3] := 2;
a[4] := 0;
a[5] := 0;
a[6] := q - 4;
a[7] := q - 9;
a[8] := q - 5;
a[9] := q - 1;
a[10] := q - 1;
a[11] := 5;
a[12] := 12 + 3*q/8;
a[13] := (17 + q/8)/3;
a[14] := 5*q/8;
a[15] := (q - 2)/3;
a[16] := (2*q - 25)/3;
a[17] := (q/4 - 50)/3;
a[18] := (13*q/8 - 25)/3;
a[19] := (5*q/8 - 2)/3;
a[20] := q/4;
a[21] := (q/8 + 29)/3;
a[22] := q/8 + 20;
a[23] := (q/2 + 32)/3;
a[24] := 11 q/128;
PolynomialRemainder[3*(s*t + 1), q^25, q]
\end{lstlisting}

\pagebreak 

\begin{lstlisting}[language=Mathematica,caption={$G_{\alpha} \cong
\LargeRee{q_0}$ where $q = q_0^7$ and $G_L \cong C_{q+1}^2 : 
\GLn{2}{3}$.}\label{lis:gl23stplus1r7}]
s := k*(q^6 - q^5 + q^4 - q^3 + q^2 - q + 1)^2 - 1;
t := k*q^12*(q^6 + 1)*(q^2 + 1)*(q - 1)^2*(q^2 - q + 1)/48 - 1;
n := 50;
Array[a, n];
k := Sum[a[i]*q^(i - 1), {i, 1, n + 1}];
a[1] := 2;
a[2] := 4;
a[3] := 2;
a[4] := 0;
a[5] := 0;
a[6] := 0;
a[7] := 0;
a[8] := q - 4;
a[9] := q - 9;
a[10] := q - 5;
a[11] := q - 1;
a[12] := 3*q/8 - 1;
a[13] := (q/8 - 1)/3;
a[14] := 5*q/8;
a[15] := (q + 16)/3;
a[16] := (2*q + 35)/3;
a[17] := (7*q/4 + 16)/3;
a[18] := (17*q/8 - 2)/3;
a[19] := (5*q/8 - 2)/3;
a[20] := (5*q/4 - 1)/3;
a[21] := (5*q/8 - 2)/3;
a[22] := 3*q/8 - 8;
a[23] := (11*q/4 - 49)/3;
a[24] := 43*q/128 - 9;
a[25] := q/2;
a[26] := (103*q/64 - 2)/3;
a[27] := (89*q/64 - 1)/3;
a[28] := (47*q/128 + 1)/3;
a[29] := (47*q/32 + 28)/3;
a[30] := (263*q/128 + 58)/3;
a[31] := (23*q/16 + 29)/3;
a[32] := (25*q/128 - 4)/3;
a[33] := (7*q/16 + 1)/3;
a[34] := (175*q/128 - 1)/3;
a[35] := (7*q/32 - 7)/3;
a[36] := 81*q/2048 - 11;
a[37] := ((2005*q + 2048*b36*q)/2048 - 635)/27;
a[38] := ((693*q + 1024*b37*q)/1024 - 351 - b36)/27;
a[39] := ((1815*q + 2048*b38*q)/2048 + 46 - b37)/27;
a[40] := ((1945*q + 2048*b39*q)/2048 - 19 - b38)/27;
a[41] := ((53*q + 512*b40*q)/512 - 4 - b39)/27;
PolynomialRemainder[27*(s*t + 1), q^43, q]
\end{lstlisting}

\pagebreak 

\begin{lstlisting}[language=Mathematica,caption={$G_{\alpha} \cong
\LargeRee{q_0}$ where $q = q_0^{r_0}$ and $G_L 
\cong C_{q\pm\sqrt{2q}+1}^2 : [96]$.}\label{lis:pmstplus1r7}]
s := k - 1;
t := k*q^3*(q^2 - 1)*(q^3 + 1)/12 - 1;
n := 20;
Array[a, n];
k := Sum[a[i]*q^(i - 1), {i, 1, n + 1}];
a[1] := 2;
a[2] := 0;
a[3] := q/2;
a[4] := (q - 2)/3;
a[5] := (q/2 - 1)/3;
a[6] := 7*q/8;
a[7] := q - 1;
a[8] := 3*q/4 - 1;
a[9] := (-18 + q + 29*q/32 + b8*q)/27;
a[10] := (q/8 - b8 + b9*q)/27;
a[11] := (1 + 17*q/32 - b9 + b10*q)/27;
a[12] := (27*q/128 - 4 - b10 + b11*q)/27;
a[13] := (7*q/32 - b11 + b12*q)/27;
a[14] := (3*q/8 + 1 - b12 + b13*q)/27;
a[15] := (331*q/512 - 9*b13 + b14*q)/243;
a[16] := (39*q/64 - 3 - b14 + b15*q)/243;
PolynomialRemainder[243*(s*t + 1), q^18, q]
\end{lstlisting}

\end{document}